\documentclass{article}%
\usepackage{amssymb}
\usepackage{amsfonts}
\usepackage{amsmath}
\usepackage{color}
\usepackage{graphicx}
\usepackage{hyperref}
\usepackage{cite}%
\providecommand{\U}[1]{\protect\rule{.1in}{.1in}}
\newtheorem{theorem}{Theorem}[section]

\newtheorem{corollary}[theorem]{Corollary}

\newtheorem{definition}[theorem]{Definition}
\newtheorem{example}[theorem]{Example}

\newtheorem{lemma}[theorem]{Lemma}
\newtheorem{notation}[theorem]{Notation}

\newtheorem{proposition}[theorem]{Proposition}
\newtheorem{remark}[theorem]{Remark}

\newenvironment{proof}[1][Proof]{\noindent \textbf{#1.} }{\  \rule{0.5em}{0.5em}}
\begin{document}

\title{Constrained Rough Paths}
\author{Thomas Cass\thanks{Department of Mathematics, Imperial College London, 180
Queen's Gate, London, SW7 2AZ, United Kingdom}, Bruce K.
Driver\thanks{Department of Mathematics, University of California, San Diego,
La Jolla, California 92093 USA}, and Christian Litterer\thanks{Centre de
Math\'{e}matiques Appliqu\'{e}es, Ecole Polytechnique, Route de Saclay, 91128
Palaiseau Cedex, France. \newline\newline The research of Bruce K. Driver was
supported in part by NSF Grant DMS – 1106270. \newline The research of Christian
Litterer was partially supported by ERC grant 321111 RoFiRM.} }
\maketitle

\begin{abstract}
We introduce a notion of rough paths on embedded submanifolds and demonstrate
that this class of rough paths is natural. On the way we develop a notion of
rough integration and an efficient and intrinsic theory of rough differential
equations (RDEs) on manifolds. The theory of RDEs is then used to construct
parallel translation along manifold valued rough paths. Finally, this
framework is used to show there is a one to one correspondence between rough
paths on a $d$ -- dimensional manifold and rough paths on $d$ -- dimensional
Euclidean space. This last result is a rough path analogue of Cartan's
development map and its stochastic version which was developed by Eeels and
Elworthy and Malliavin.

\end{abstract}
%\date{\today \ \emph{File:} \jobname{.tex}}
\tableofcontents

\numberwithin{equation}{section}

\section{Introduction\label{sec.1}}

In the series of papers \cite{lyons-94,lyons-95,lyons-98}, Terry Lyons
introduced and began the development of the theory of \textit{rough paths}\ on
a Banach space $W.$ This theory allow us to model the evolution of interacting
systems, driven by highly irregular non-differentiable inputs, modelled as
differential equations driven by a rough path $\mathbf{X.}$ The theory of
rough paths provides existence and uniqueness of solutions to such equations,
moreover the solutions depend continuously on the driver $\mathbf{X.}$ Among
the many applications arising from the interplay of rough paths and stochastic
analysis are the study of solutions to stochastic differential equations
driven by Gaussian signals see e.g. \cite{CFV}, \cite{CF}, \cite{CHLT},
\cite{FO}, \cite{CLL} and the analysis of broad classes of stochastic partial
differential equations (SPDEs) \cite{CFO}, \cite{DGT}, \cite{HW}, \cite{H}.
Rough paths also provide us with alternative ways to think about and encode
the information presented in a dynamical system.

A rough path of order $p\in\lbrack2,3)$ on $\left[  0,T\right]  $ with values
in a Banach space $\left(  W,\left\vert \cdot\right\vert \right)  $ is a pair
of functions
\[
\mathbf{X}_{s,t}:=\left(  x_{s,t},\mathbb{X}_{s,t}\right)  \in W\oplus
W\otimes W,
\]
which may be thought of as the increments of the path itself and a
second-order term $\mathbb{X}_{s,t}.$ Rough paths are characterised by an
algebraic property analogous to the homomorphism property of the Chen series
of a path (also known as the multiplicative property) and an analytic $p-$
variation type constraint on $\mathbf{X.}$ A great variety of stochastic
classical processes \ may be lifted to rough paths. For example, every
$\mathbb{R}^{d}$ -- valued continuous semi-martingale $\left(  x_{s}\right)
_{0\leq s\leq T}$ (e.g. Brownian motion) and large classes of Gaussian
processes including fractional Brownian motion (fBM) with Hurst parameter
$H>1/4$ may almost surely be augmented by a process $\mathbb{X}_{s,t}$. For
continuous semi-martingales for example one may define
\[
\mathbb{X}_{s,t}:=\int_{s\leq\tau\leq t}x_{s,\tau}\otimes dx_{\tau}%
\]
where the integral can either be interpreted as an It\^{o} integral or a
Fisk-Stratonovich integral. The resulting $\mathbb{X}_{s,t}$ typically
\textit{does} depend on which integral is used.

Up until now, with the exception of \cite{Cass2012a}, rough path theory has
essentially been restricted to dynamical systems on linear state spaces. In
view of the fact that many (if not most) natural dynamical systems come with
geometric constraints it is natural and necessary to develop a theory of
constrained rough paths. This paper tries to fill this gap by developing an
efficient extension of the rough path theory in Banach spaces to a theory
valid in manifolds. The main difficulty in doing so is that the Banach space
rough path theory makes heavy use of the underlying linear structure. One
difficulty in this program is the lack of a simple unambiguous infinitessimal
characterisation of rough paths analogous to the notion of tangent vector
which is abundantly used in understanding smooth paths in a manifold.

In \cite{Cass2012a}, an abstract theory of rough paths on a manifold was
developed. The approach taken in \cite{Cass2012a} is in the spirit of
distribution theory. Namely, rather than directly defining a path in a
manifold the authors view a rough path as a special kind of current in the
manifold, i.e. a certain class of (at level one) linear functional on the
space sufficiently regular one forms on the manifold. If the manifold is a
Banach space, this current corresponds to the map taking one forms to their
rough path integral. The definition in \cite{Cass2012a} is global and
intrinsic, but it relies on the non-trivial concept of Lip-$\gamma$ manifolds
in order to obtain uniform estimates. This does not immediately provide the
tools to explore probabilistic applications, and we avoid the use of
Lip-$\gamma$ manifolds in our presentation.

In the present paper we develop a more direct approach based on choosing an
embedding of our manifold $M$ into some Euclidean space $E=\mathbb{R}^{N}.$
(The embedding of $M$ into $E$ is described by introducing local constraints,
see Definition \ref{def.3.1} below.) We introduce a notion of weakly geometric
rough paths on manifolds, the class of rough paths underlying much of the
modern study of the interactions of rough paths and probability. Our
constructions do not rely on smooth approximation arguments. While, as in
virtually all of stochastic analysis on manifolds, similar results may at
least in the finite dimensional setting be obtained by means of smooth
approximations. However, we believe avoiding this step is crucial and makes
the mathematical ideas and arguments involved more transparent. Our definition
of a rough path is natural in that it is the maximal class that permits a
consistent definition of rough integration and, though our proofs and
definitions will sometimes depend on the embedding, we will however show that
the choice of embedding is not important. In fact, the theory is intrinsic to
the manifolds, see Cass, Driver, Litterer \cite{CDL}, where we clarify the
relations of intrinsic and embedded definitions of rough paths on manifolds.
Along the way we develop the intrinsic theory of rough ordinary differential
equations, in the spirit of the classical definition of semi-martingale
solutions to stochastic differential equations on a manifold along with the
rough analogue of Cartan's rolling map. This shows that rough paths on a $d$
-- dimensional manifold are in one to one correspondence with rough paths on
$d$ -- dimensional Euclidean space. The by now classical rolling construction
of Brownian motion on a manifold appears to have first been discovered in
Eeells and Elworthy \cite{EE} and the relation to of the stochastic
development map to SDEs on the orthogonal frame bundle was realised in
Elworthy \cite{ELW1}, \cite{ELW2}, \cite{ELW3}. The frame bundle approach was
also extensively explored by Malliavin, see e.g. \cite{MAL}. \ It is hoped
that the results of this paper will be the foundation of future work that
explores the properties of Gaussian processes such as fBM and their SDEs in
manifold settings.

The paper is organized as follows. Section \ref{sec.2} is devoted to
introducing fundamental definitions and some preliminary results in Banach
space - valued rough path theory. Section \ref{sec.3} is where we define a
rough path in an embedded manifold. In short, we say that a weakly geometric
rough path $\mathbf{X=}\left(  x,\mathbb{X}\right)  $ in $E\times E\otimes E$
is a weakly geometric rough path in $M$ provided the trace of $\mathbf{X}$
$\left(  x\right)  $ lies in $M$ and $\int\tilde{\alpha}\left(  d\mathbf{X}%
\right)  =0$ for all $\tilde{\alpha}\in\Omega^{1}\left(  U,W\right)  $ where
$U$ is some open neighborhood of $M$ and $\tilde{\alpha}|_{TM}\equiv0.$ We
denote the space of weakly geometric rough paths in $M$ by $WG_{p}\left(
M\right)  .$ For $\mathbf{X\in}WG_{p}\left(  M\right)  $ and $\alpha\in
\Omega^{1}\left(  M,W\right)  $ (the set of smooth, $W$-valued one-form
defined on $U$) we let $\int\alpha\left(  d\mathbf{X}\right)  :=\int%
\tilde{\alpha}\left(  d\mathbf{X}\right)  $ where we make an arbitrary choice
of $\tilde{\alpha}\in\Omega^{1}\left(  U,W\right)  $ where $U$ is a (small)
open neighborhood of $M$ such that $\tilde{\alpha}|_{TM}=\alpha.$ This
definition is independent of the choice $\tilde{\alpha}$ of extension of
$\alpha$ by the very definition of $\mathbf{X}\in WG_{p}\left(  M\right)  .$
In Proposition \ref{pro.3.35} we show that $\mathbf{X}=\left(  x,\mathbb{X}%
\right)  \in WG_{p}\left(  E\right)  $ is in $WG_{p}\left(  M\right)  $ iff
$x_{t}\in M$ for all $t$ and the second order component $\left(
\mathbb{X}\right)  $ satisfies, for some $\delta>0,$
\[
\sup_{0<t-s<\delta}\frac{\left\vert \left(  I_{E}\otimes Q\left(
x_{s}\right)  \right)  \mathbb{X}_{s,t}\right\vert }{\omega\left(  s,t\right)
^{3/p}}<\infty
\]
wherein $Q$ is the orthogonal projection onto the normal bundle and $\omega$
the control of the rough path. This latter criteria is a precise formulation
of the intuition that in order for $\mathbf{X}$ to be in $WG_{p}\left(
M\right)  $ we must have $\mathbb{X}_{s,t}\in E\otimes E$ is close to being in
$T_{x_{s}}M\otimes T_{x_{s}}M.$

We then go on to use the tools developed to characterise some fundamental
properties of rough paths on manifolds. It is shown - as in the finite
dimensional vector space setting - that weakly geometric and geometric rough
paths (that arise from smooth approximations) are essentially the same and a
rough path is completely characterised by its integral against vector valued
one forms (Corollary \ref{cor.3.25} ). We obtain two alternative
characterisations of rough paths on a manifold: as a rough path on the ambient
space that integrates one forms extended from the manifold consistently on the
ambient space and as the projection of a rough path on the ambient space.
Finally, we study the pushforward of rough paths under smooth maps between
manifolds and deduce that the class of rough paths is not sensitive to the
choice of embedding. Although the definition of $WG_{p}\left(  M\right)  $ and
$\int\alpha\left(  d\mathbf{X}\right)  $ depends on the embedding of $M$
inside of $E,$ nevertheless as explain in Subsection \ref{sub.3.4}, the
results in this paper are in fact essentially independent of the embedding.

Section \ref{sec.4} is devoted to the notion rough differential equations on
$M.$ Theorem \ref{the.4.2} shows that one may solve a rough differential
equation on $M$ by extending the vector fields defining the differential
equation to the ambient space and then applying the Euclidean rough path
theory to the resulting dynamical system. The output is a weakly geometric
rough path in $M$ which does not depend on any of the choices made in the
extensions. Later in Theorem \ref{the.4.5} we derive an equivalent intrinsic
characterisations of these solutions.

Section \ref{sec.5} develops the notion of rough parallel translation along
manifold valued rough paths. \textit{Parallel translation} along a rough path
$\mathbf{X}$ in $M$ is defined as a rough path $\mathbf{U}$ in the orthogonal
frame bundle $O\left(  M\right)  $ over $M$ which solves a prescribed RDE on
$O\left(  M\right)  $ driven by\textbf{ }$\mathbf{X},$ see Definition
\ref{def.5.13}. It is shown in Proposition \ref{pro.5.15} that the
RDE\ defining $\mathbf{U}$ does not explode and so $\mathbf{U}$ exists on the
full time interval, $\left[  0,T\right]  .$ It is then shown in Theorems
\ref{the.5.16} and \ref{the.5.17} that two natural classes of RDE's on
$O\left(  M\right)  $ give rise to an element $\mathbf{U}\in WG_{p}\left(
O\left(  M\right)  \right)  $ each of which is parallel translation along
$\mathbf{X}:=\pi_{\ast}\left(  \mathbf{U}\right)  $ where $\pi:O\left(
M\right)  \rightarrow M$ is the natural projection map on $O\left(  M\right)
.$ Here $\pi_{\ast}\left(  \mathbf{U}\right)  $ denotes the pushforward of
$\mathbf{U}$ by $\pi,$ see Proposition \ref{pro.3.38}.

In Section \ref{sec.6} we show in Corollary \ref{cor.6.12} that there is
(similar to the smooth theory) a one to one correspondence between rough paths
on the orthogonal frame bundle $O\left(  M\right)  $ to $M$ and rough paths on
the Euclidean space $\mathbb{R}^{d}\times so\left(  d\right)  .$ Furthermore
Theorems \ref{the.6.18}, and Corollary \ref{cor.6.19} show there are
one-to-one correspondence between rough paths on $M,$ \textquotedblleft
horizontal\textquotedblright\ rough paths on $O\left(  M\right)  \ $(see
Definition \ref{def.6.14}), and rough paths on $\mathbb{R}^{d}.$

The paper is completed with two appendices. In Appendix \ref{app.A} we gather
together some needed results of the Banach space-valued rough path theory
while Appendix \ref{app.B} explains a few details on how to view $O\left(
M\right)  $ as an embedded submanifold which are needed in Section \ref{sec.6}
of the paper.

\section{Background Rough Path Results\label{sec.2}}

\subsection{Basic notations\label{sub.2.1}}

In this section we introduce some basic notations for rough paths on Banach
spaces. In addition, we gather some elementary preliminary results that will
prove useful in the sequel. Some additional rough path theory results on
Banach spaces needed in this paper may also be found in Appendix \ref{app.A}.
Throughout this section, $V$, $W$ and $U$ will denote real Banach spaces. For
simplicity in this paper, we will typically assume that all Banach spaces are
finite dimensional. If $\left(  V,\left\vert \cdot\right\vert \right)  $ is a
Banach space we will abuse notation and write $\left\vert \cdot\right\vert $
for one of the tenor norms on $V\otimes V.$ Because $\dim V<\infty,$ the
choice of tensor norm on $V\otimes V$ is unimportant. For $\mathbb{X}\in
V\otimes V$ we denote its symmetric and anti-symmetric part to be
$\mathbb{X}^{s}$ and $\mathbb{X}^{a}$ respectively. The following definition
and (abuse of) notation will frequently be used in the sequel.

\begin{definition}
[Truncated Tensor Algebra]\label{def.2.1} Let $T_{2}\left(  V\right)
:=\mathbb{R}\oplus V\oplus V\otimes V$ which we make into an algebra by using
the multiplication in the full tensor algebra and then disregarding any terms
that appear in $V^{\otimes3}\oplus V^{\otimes4}\ldots.$ In more detail, if
$a,b\in\mathbb{R},$ $x,y\in V,$ and $\mathbb{X},\mathbb{Y}\in V\otimes V,$
then%
\[
\left(  a,x,\mathbb{X}\right)  \left(  b,y,\mathbb{Y}\right)  :=\left(
ab,ay+bx,a\mathbb{Y}+x\otimes y+b\mathbb{X}\right)  .
\]
In the future we will typically write $a+x+\mathbb{X}$ for $\left(
a,x,\mathbb{X}\right)  .$
\end{definition}

\begin{notation}
\label{not.2.2}If $B:V\times V\rightarrow W$ is a bilinear form with values in
a vector space $W$ then by the universal property of the tensor product there
is a unique linear map, $\hat{B}:V\otimes V\rightarrow W$ such that $B\left(
a,b\right)  =\hat{B}\left(  a\otimes b\right)  $ for all $a,b\in V.$ Given
$A\in V\otimes V,$ it will be useful to abuse notation and abbreviate $\hat
{B}\left(  A\right)  $ as $B\left(  a,b\right)  |_{a\otimes b=A}.$ For example
if $A=\sum_{i=1}^{\ell}a_{i}\otimes b_{i}\in V\otimes V$ it follows that%
\[
B\left(  a,b\right)  |_{a\otimes b=A}:=\sum_{i=1}^{\ell}B\left(  a_{i}%
,b_{i}\right)  .
\]

\end{notation}

Throughout this paper we let $T$ denote a positive finite real number, $p$ be
a fixed real number in the interval $[2,3)$ and $\omega$ a \textit{control}
whose definition we now recall.

\begin{definition}
\label{def.2.3}A \textbf{control} $\omega:\Delta_{\left[  0,T\right]
}:=\left\{  \left(  s,t\right)  :0\leq s\leq t\leq T\right\}  $ $\rightarrow
\mathbb{R}_{+}$ is a continuous non-negative function which is superadditive,
positive off the diagonal, and zero on the diagonal in $\Delta_{\left[
0,T\right]  }.$
\end{definition}

\begin{definition}
[Rough Paths]\label{def.2.4}For a Banach space $V$, the set of ($\omega$ --
controlled $V-$valued) $p$-rough paths consists of pairs $\mathbf{X=}\left(
x,\mathbb{X}\right)  $ of continuous paths%
\[
x:\left[  0,T\right]  \rightarrow V\text{ \ and }\mathbb{X}:\Delta_{\left[
0,T\right]  }\rightarrow V\otimes V,
\]
satisfying the following conditions:

\begin{enumerate}
\item \textit{The Chen identity}; i.e.%
\begin{equation}
\mathbb{X}_{s,t}=\mathbb{X}_{s,u}+\mathbb{X}_{u,t}+x_{s,u}\otimes
x_{u,t}~\forall~0\leq s\leq u\leq t\leq T, \label{equ.2.1}%
\end{equation}
where here, as throughout, $x_{s,t}:=x_{t}-x_{s}$ will denote the increment of
the path $x$ over $\left[  s,t\right]  .$

\item A $p-$variation\textit{\ regularity constraint}:%
\begin{equation}
\sup_{0\leq s<t\leq T}\text{\ }\frac{\left\vert x_{s,t}\right\vert }%
{\omega\left(  s,t\right)  ^{1/p}}<\infty\text{ and }\sup_{0\leq s<t\leq
T}\text{\ }\frac{\left\vert \mathbb{X}_{s,t}\right\vert }{\omega\left(
s,t\right)  ^{2/p}}<\infty. \label{equ.2.2}%
\end{equation}

\end{enumerate}
\end{definition}

We can identify a rough path as a map taking values in the tensor algebra.

\begin{remark}
\label{rem.2.5} It is often convenient to identify a rough path,
$\mathbf{X=}\left(  x,\mathbb{X}\right)  ,$ with the function $\mathbf{X}%
:\Delta_{\left[  0,T\right]  }\rightarrow T_{2}\left(  V\right)  $ defined by%
\[
\mathbf{X}_{s,t}:=1+x_{s,t}+\mathbb{X}_{s,t}\text{ for }\left(  s,t\right)
\in\Delta_{\left[  0,T\right]  }.
\]
Using this identification, Chen's identity becomes the following
multiplicative property of $\mathbf{X};$
\begin{equation}
\mathbf{X}_{s,t}=\mathbf{X}_{s,u}\mathbf{X}_{u,t}\mathbf{\ }~\forall~0\leq
s\leq u\leq t\leq T, \label{equ.2.3}%
\end{equation}
where multiplication is given as in Definition \ref{def.2.1}.
\end{remark}

The collection of $V$ -- valued $p$ -- rough paths controlled by $\omega$ is
denoted by $R_{p}\left(  \left[  0,T\right]  ,V,\omega\right)  $ (also denoted
by $R_{p}\left(  V\right)  $ where no confusion arises).

\begin{example}
\label{exa.2.6}Suppose $x:\left[  0,T\right]  \rightarrow V$ is a continuous
bounded variation path. Then a simple example of a $p$ -- rough path is the
(truncated) signature $\left(  S_{2}\left(  x\right)  _{s,t}:=\mathbf{X}%
_{s,t}~|~0\leq s<t\leq T\right)  $ defined by
\begin{equation}
\mathbf{X}_{s,t}:=1+x_{s,t}+\mathbb{X}_{s,t}\in T_{2}\left(  V\right)
\label{equ.2.4}%
\end{equation}
where
\begin{equation}
\mathbb{X}_{s,t}:=\int_{s<t_{1}<t_{2}<t}dx_{t_{1}}\otimes dx_{t_{2}}=\int%
_{s}^{t}x_{s,u}\otimes dx_{u} \label{equ.2.5}%
\end{equation}
and the latter integral being the Lebesgue-Stieltjes integral. For the control
we may take
\[
\omega\left(  s,t\right)  =\left\vert x\right\vert _{1\text{-var;}\left[
s,t\right]  }:=\sup_{D=\left\{  t_{i}:s=t_{0}<t_{1}<...<t_{n}=t\right\}  }%
\sum_{i=1}^{n}\left\vert x_{t_{i-1},t_{i}}\right\vert .
\]
In this case, $\mathbb{X}$ is not an extra piece of information but is in fact
determined by the basic path $x.$
\end{example}

\begin{remark}
\label{rem.2.7}If $x:\left[  0,T\right]  \rightarrow V$ is continuous and of
bounded variation and $\mathbb{X}$ is given as in Eq. (\ref{equ.2.5}), then as
a consequence of the fundamental theorem of calculus the symmetric part of
$\mathbb{X}_{st}$ satisfies,%
\begin{equation}
\mathbb{X}_{st}^{s}=\frac{1}{2}x_{s,t}\otimes x_{s,t}~\forall~0\leq s\leq
t\leq T. \label{equ.2.6}%
\end{equation}
\bigskip
\end{remark}

In this paper we are interested in the following two important subsets of
$R_{p}\left(  V\right)  .$

\begin{definition}
[$WG_{p}\left(  V\right)  $ and $G_{p}\left(  V\right)  $]\label{def.2.8}Let
$\mathbf{X}$ be a $V-$valued $p$-rough path, i.e. $\mathbf{X}\in R_{p}\left(
V\right)  .$

\begin{enumerate}
\item We say that $\mathbf{X}$ is a \textbf{geometric }$p$\textbf{-rough
path}, and write $\mathbf{X}$ $\in G_{p}\left(  V\right)  ,$ if $\mathbf{X}$
belongs to the closure of the set:%
\[
\left\{  \mathbf{Y:Y=}S_{2}\left(  y\right)  ,\text{ }y\text{ continuous and
of finite }1\text{-variation}\right\}
\]
with respect to the topology induced by the metric (\ref{equ.A.1}).

\item We say that $\mathbf{X}$ is a \textbf{weakly geometric }$p$%
\textbf{-rough path}, and write $\mathbf{X}\in WG_{p}\left(  V\right)  ,$ if
Eq. (\ref{equ.2.6}) holds.
\end{enumerate}
\end{definition}

\begin{remark}
\label{rem.2.9}If $\dim\left(  V\right)  <\infty$ and $p<q,$ then we have the
strict inclusions $G_{p}\left(  V\right)  \subset WG_{p}\left(  V\right)
\subset G_{q}\left(  V\right)  $ (see Corollary 8.24 of \cite{FV}) and so one
typically does not have to pay much attention to the difference between
geometric and weakly geometric rough paths. However, in infinite dimensions
the compactness argument used in the proof that $WG_{p}\left(  V\right)
\subset G_{q}\left(  V\right)  $ breaks down.
\end{remark}

\subsection{Approximate rough paths and integration\label{sub.2.2}}

The following notation will be used heavily in this paper.

\begin{notation}
[$\simeq$ and $\simeq_{\delta}$]\label{not.2.10}Let $\omega$ be a control, and
assume $g$ and $h$ are continuous functions from $\Delta_{\left[  0,T\right]
}$ into some Banach space $W.$ \ Then we will write%
\[
g_{s,t}\simeq h_{s,t}%
\]
if there exists $\delta>0$ and a constant $C\left(  \delta\right)  >0$ such
that for all $s$ and $t$ in $\left[  0,T\right]  $ satisfying $\left\vert
s-t\right\vert \leq\delta$ we have
\[
\left\vert g_{s,t}-h_{s,t}\right\vert \leq C\left(  \delta\right)
\omega\left(  s,t\right)  ^{3/p}.
\]
If we wish to emphasize the dependence on $\delta$ then we will write
$g_{s,t}\simeq_{\delta}h_{s,t}.$
\end{notation}

\begin{remark}
\label{rem.2.11}As a typical application of this notation, let us note that if
$g:\left[  0,T\right]  \rightarrow V$ is continuous and such that
$g_{s,t}\simeq0$ then, because the increments form an additive function on
$\Delta_{\left[  0,T\right]  }$, it must be that $g$ is constant. Indeed, if
$D=\left\{  t_{i}:k=0,1...,n\right\}  $ is any partition of $\left[
0,t\right]  \subseteq\left[  0,T\right]  $ with $\left\vert D\right\vert
\leq\delta$ then
\[
\left\vert g_{0,t}\right\vert =\left\vert \sum_{i=1}^{n}g_{t_{i},t_{i+1}%
}\right\vert \leq\sum_{i=1}^{n}\left\vert g_{t_{i},t_{i+1}}\right\vert \leq
C\left(  \delta\right)  \left\vert D\right\vert ^{3/p-1}\omega\left(
0,t\right)
\]
which tends to $0$ as $\left\vert D\right\vert \rightarrow0.$
\end{remark}

This elementary remark may be strengthened to apply to rough paths. The
difficulty of course that the second (and higher) order processes are no
longer additive with respect to $\left(  s,t\right)  .$\ The following lemma
is due to Lyons \cite{lyons-98}, and is used to powerful effect in his
Extension Theorem.

\begin{lemma}
\label{lem.2.12}Suppose $\left(  x,\mathbb{X}\right)  ,$ $\left(
y,\mathbb{Y}\right)  \in R_{p}\left(  V\right)  $ satisfy $a_{s,t}%
:=x_{s,t}-y_{s,t}\simeq0$ and $\mathbb{A}_{s,t}:=\mathbb{X}_{s,t}%
-\mathbb{Y}_{s,t}\simeq0.$ Then the two rough paths coincide, i.e. $\left(
x,\mathbb{X}\right)  =\left(  y,\mathbb{Y}\right)  $. In particular, this
taking $\left(  y,\mathbb{Y}\right)  $ to be the zero rough path in
$R_{p}\left(  V\right)  $ we may conclude if $\left(  x,\mathbb{X}\right)  \in
R_{p}\left(  V\right)  $ satisfies $x_{s,t}\simeq0$ and $\mathbb{X}%
_{s,t}\simeq0,$ then $x_{st}=0$ and $\mathbb{X}_{s,t}=0$ for all $0\leq s\leq
t\leq T.$
\end{lemma}

\begin{proof}
Since $a_{s,t}$ is additive we must have for every partition $D$ of $\left[
s,t\right]  $%
\[
\left\vert a_{s,t}\right\vert =\left\vert \sum_{i:t_{i}\in D}a_{t_{i},t_{i+1}%
}\right\vert \rightarrow0\text{ as }\left\vert D\right\vert \rightarrow0,
\]
and hence $x_{s,t}=y_{s,t}\mathbf{.}$ It follows from \cite{LCL}, Lemma 3.4
that $\mathbb{A}_{s,t}$ is also additive and repeating the argument with
$\mathbb{A}_{s,t}$ in place of $a_{s,t}$ yields the claim.
\end{proof}

We say a functional $\mathbf{Z:=}\left(  z,\mathbb{Z}\right)  $ defined by
\[
\mathbf{Z}_{s,t}:=1+z_{s,t}+\mathbb{Z}_{s,t}\in T_{2}\left(  V\right)
,~\forall~\left(  s,t\right)  \in\Delta_{\left[  0,T\right]  }%
\]
is an \textbf{almost rough path} if it satisfies the requirements of
Definition \ref{def.2.4} except identity $\left(  \ref{equ.2.1}\right)  ,$ but
instead
\[
\mathbb{Z}_{s,u}-\mathbb{Z}_{s,t}-\mathbb{Z}_{t,u}-z_{s,t}\otimes
z_{t,u}\simeq0,
\]
holds, i.e. it approximately satisfies the multiplicative identity $\left(
\ref{equ.2.1}\right)  .$The following theorem due to Lyons is a cornerstone
for the development of the integration for rough paths. It states that for
every almost rough path there exists a unique rough path that is
\textquotedblleft close.\textquotedblright\ Note that the uniqueness follows
from Lemma \ref{lem.2.12} above.

\begin{theorem}
\label{the.2.13}Let $\mathbf{Z:=}\left(  z,\mathbb{Z}\right)  $ be an almost
rough path on $V.$ Then there exists a unique rough path $\mathbf{X=}\left(
x,\mathbb{X}\right)  \in R_{p}\left(  V\right)  $ \ such that $x_{s,t}$
$\simeq z_{s,t}$ and $\mathbb{Z}_{s,t}\simeq\mathbb{X}_{s,t}.$
\end{theorem}

The following result due to Terry Lyons \cite{lyons-98} allows us to define
the integral of a rough path against a sufficiently regular one form.

\begin{theorem}
\label{the.2.14}Suppose that $\mathbf{Z}\in WG_{p}\left(  V\right)  $ and
$\alpha\in C^{2}\left(  V,\operatorname*{End}\left(  V,W\right)  \right)  $ is
a one form on $V$ with values in $W.$ Then there is a unique $\mathbf{X}\in
WG_{p}\left(  W\right)  $ such that $x_{0}=0,$%
\begin{equation}
X_{s,t}^{1}\simeq\alpha\left(  z_{s}\right)  Z_{s,t}^{1}+\alpha^{\prime
}\left(  z_{s}\right)  \mathbb{Z}_{s,t},\text{ } \label{equ.2.7}%
\end{equation}
and
\begin{equation}
\mathbb{X}_{s,t}\simeq\alpha\left(  z_{s}\right)  \otimes\alpha\left(
z_{s}\right)  \mathbb{Z}_{s,t}. \label{equ.2.8}%
\end{equation}

\end{theorem}

In the future we will denote this $\mathbf{X}$ by $\int\alpha\left(
d\mathbf{Z}\right)  $ and use it as the definition for the rough integral. The
proof is a consequence of Theorem \ref{the.2.13}. The rough path integral has
a number of important properties, in particular the map taking%
\[
\mathbf{Z\rightarrow}\int\alpha\left(  d\mathbf{Z}\right)
\]
is continuous in the rough path metric $\left(  \ref{equ.A.1}\right)  .$

\subsection{Rough differential equations\label{sub.2.3}}

The following definition of a rough differential equation (RDE) is in the
spirit of Davie \cite{Davie} and may be found for example in Friz, Hairer
\cite[Proposition 8.4]{FH}.

\begin{definition}
[RDE]\label{def.2.15}Let $\mathbf{Z}\in WG_{p}\left(  W\right)  $ and
$Y:V\rightarrow\operatorname{Hom}\left(  W,V\right)  $ be a $C^{1}$ -- map.
Then $\mathbf{X}\in WG_{p}\left(  E\right)  $ solves the RDE,%
\begin{equation}
d\mathbf{X}=Y\left(  x\right)  d\mathbf{Z} \label{equ.2.9}%
\end{equation}
if and only if%
\begin{align*}
x_{s,t}  &  \simeq Y\left(  x_{s}\right)  z_{s,t}+Y^{\prime}\left(
x_{s}\right)  Y\left(  x_{s}\right)  \mathbb{Z}_{s,t}\\
\mathbb{X}_{s,t}  &  \simeq\left[  Y\left(  x_{s}\right)  \otimes Y\left(
x_{s}\right)  \right]  \mathbb{Z}_{s,t}%
\end{align*}
where
\[
Y^{\prime}\left(  x_{s}\right)  Y\left(  x_{s}\right)  \left[  a\otimes
b\right]  :=\left(  Y^{\prime}\left(  x_{s}\right)  Y\left(  x_{s}\right)
a\right)  b=\left(  \partial_{Y\left(  x_{s}\right)  a}Y\right)  \left(
x_{s}\right)  b.
\]
Alternatively if we let $Y_{b}\left(  x\right)  :=Y\left(  x\right)  b,$ then%
\[
Y^{\prime}\left(  x_{s}\right)  Y\left(  x_{s}\right)  \left[  a\otimes
b\right]  :=Y_{b}^{\prime}\left(  x_{s}\right)  Y_{a}\left(  x_{s}\right)
=\left(  \partial_{Y_{a}\left(  x_{s}\right)  }Y_{b}\right)  \left(
x_{s}\right)  .
\]

\end{definition}

Existence and uniqueness of solutions for RDEs defined by sufficiently regular
vector fields is due to Lyons \cite{lyons-98}. The following theorem is an
easy consequence of Theorem 10.14 of \cite[Theorem 10.14]{FV}.

\begin{theorem}
[RDE existence and uniqueness]\label{the.2.16}Let $p\in\lbrack2,3),$
$\mathbf{Z}\in WG_{p}\left(  W,\left[  0,T\right]  \right)  ,$ $Y:V\rightarrow
\operatorname{Hom}\left(  W,V\right)  $ be a smooth map and for $k\in\left\{
0,1,2,\dots\right\}  ,$ let
\[
\left\Vert Y^{\left(  k\right)  }\right\Vert _{\infty}:=\sup\left\{
\left\Vert \left(  \partial_{v_{1}}\dots\partial_{v_{k}}Y\right)  \left(
x\right)  \right\Vert _{\operatorname{Hom}\left(  W,V\right)  }:x\in V\text{
and }v_{i}\in V\text{ with }\left\Vert v_{i}\right\Vert _{V}=1\right\}  .
\]
If%
\begin{equation}
M_{Y}:=\max\left\{  \left\Vert Y\right\Vert _{\infty},\left\Vert Y^{\prime
}\right\Vert _{\infty},\left\Vert Y^{\prime\prime}\right\Vert _{\infty
}\right\}  <\infty, \label{equ.2.10}%
\end{equation}
then there exists a unique $\mathbf{X}\in WG_{p}\left(  V,\left[  0,T\right]
\right)  $ that solves the RDE $\left(  \ref{equ.2.9}\right)  $ over $\left[
0,T\right]  $ in the sense of Definition \ref{def.2.15}. In addition, there
exists a constant $C_{p}\ $(depending only on $p)$ such that%
\begin{equation}
\left\Vert x\right\Vert _{p-var;\left[  u,v\right]  }\leq C_{p}\max\left(
M\left\Vert \mathbf{Z}\right\Vert _{p-var;\left[  u,v\right]  },M^{p}%
\left\Vert \mathbf{Z}\right\Vert _{p-var;\left[  u,v\right]  }^{p}\right)
~\forall~0\leq u<v\leq T. \label{equ.2.11}%
\end{equation}

\end{theorem}

The following corollary is a locallization of Theorem \ref{the.2.16} which
will prove useful later.

\begin{corollary}
[Local RDE existence]\label{cor.2.17}Let $U\subset V$ be an open neighborhood
, $U_{1}$ be a precompact open neighborhood with closure in $U,$ and
$Y:U\rightarrow\operatorname{Hom}\left(  W,V\right)  $ be a smooth map. Then
there exists $\delta>0$ such that for all $\left(  x,t_{0}\right)  \in
U_{1}\times\left[  0,T\right]  ,$%
\begin{equation}
d\mathbf{X}=Y\left(  x\right)  d\mathbf{Z,}\text{ }x_{t_{0}}=x
\label{equ.2.12}%
\end{equation}
has a unique solution $\mathbf{X}\in WG_{p}\left(  V,\left[  t_{0}%
,t_{0}+\delta\wedge T\right]  \right)  $ in the sense of Definition
\ref{def.2.15} (naturally with trace $x_{t}\in U$ for all $t\in\left[
t_{0},t_{0}+\delta\wedge T\right]  $).
\end{corollary}

\begin{proof}
Choose another open precompact subset, $U_{2},$ of $V$ so that $\bar{U}%
_{1}\subset U_{2}\subset$ $\bar{U}_{2}\subset U$ and choose $\varphi\in
C_{c}^{\infty}\left(  U\right)  $ such that $\varphi=1$ on $\bar{U}_{2}.$ Let
$\tilde{Y}=\varphi Y$ which we then extend to be zero outside of $U.$ Clearly,
$M_{\tilde{Y}}<\infty$ where $M_{\tilde{Y}}$ is as in Eq. (\ref{equ.2.10})
with $Y$ replaced by $\tilde{Y}.$

Recall that if $u\left(  s,t\right)  :=\left\Vert \mathbf{Z}\right\Vert
_{p-var;\left[  s,t\right]  }$ for $\left(  s,t\right)  \in\Delta_{\left[
0,T\right]  },$ then $u^{p}\left(  s,t\right)  $ is a control and in
particular, $u\left(  s,t\right)  $ is continuous on $\Delta_{\left[
0,T\right]  }$ and vanishes on the diagonal. Therefore if $\varepsilon
:=\operatorname{dist}\left(  U_{1},U_{2}^{c}\right)  >0,$ then there exists
(by the uniform continuity of $u)$ a $\delta>0$ such that
\[
C_{p}\max\left(  M_{\tilde{Y}}\left\Vert \mathbf{Z}\right\Vert _{p-var;\left[
t_{0},t_{0}+\delta\wedge T\right]  },M_{\tilde{Y}}^{p}\left\Vert
\mathbf{Z}\right\Vert _{p-var;\left[  t_{0},t_{0}+\delta\wedge T\right]  }%
^{p}\right)  <\varepsilon\text{ }\forall~t_{0}\in\left[  0,T\right]  .
\]
By Theorem \ref{the.2.16}, given any $\left(  x,t_{0}\right)  \in U_{1}%
\times\left[  0,T\right]  $ there exists a unique $\mathbf{X}\in WG_{p}\left(
V,\left[  t_{0},T\right]  \right)  $ that solves%
\begin{equation}
d\mathbf{X}=\tilde{Y}\left(  x\right)  d\mathbf{Z,}\text{ }x_{t_{0}}=x.
\label{equ.2.13}%
\end{equation}
By the choice of $\delta,$ the bound in Eq. $\left(  \ref{equ.2.11}\right)  ,$
and the triangle inequality, it follows that $x_{t}\in U_{2}\subseteq U$ for
all $t\in\left[  t_{0},t_{0}+\delta\wedge T\right]  .$ As $Y=\tilde{Y}$ on
$U_{2}$ it follows that $\mathbf{X}$ also solves $\left(  \ref{equ.2.12}%
\right)  $ on $\left[  t_{0},t_{0}+\delta\wedge T\right]  .$
\end{proof}

The solutions of rough differential satisfy a universal limit theorem which
states that the map taking $\mathbf{X}$ to the solution $\mathbf{Z}$ is
continuous in the $p-$ variation metric on rough paths (see \cite{lyons-98}).
We also remark that the original definition of the solution of a rough
differentiable equations (see Lyons \cite{lyons-98}) is given in terms of a
fixed point of a rough integral on $V\oplus W.$

The next lemma implies that for sufficiently regular vector fields an
RDE\ solution blows up if and only if both the trace \textit{and} the
second-order process of the solution explode. In other words, it is not
possible for the explosion of a solution of an RDE to be caused only by the
explosion of the second-order process of the solution.

\begin{lemma}
[Augmentations for free]\label{lem.2.18}Let $\mathbf{Z}\in WG_{p}\left(
W\right)  $ and $Y:V\rightarrow\operatorname{Hom}\left(  W,V\right)  $ be a
smooth map and consider the RDE%
\begin{equation}
d\mathbf{X}=Y\left(  x\right)  d\mathbf{Z}\text{ with }x\left(  0\right)
=x_{0} \label{equ.2.14}%
\end{equation}
where $x_{0}$ is given. Suppose that we can solve this equation for the trace
part, i.e. we can find a path $x$ such that
\begin{equation}
x_{s,t}\simeq Y\left(  x_{s}\right)  z_{s,t}+Y^{\prime}\left(  x_{s}\right)
Y\left(  x_{s}\right)  \mathbb{Z}_{s,t} \label{equ.2.15}%
\end{equation}
holds for $0\leq s,t\leq T$. Then there exists a lift $\mathbf{X}\in
WG_{p}\left(  V\right)  $ of $x$ that solves $\left(  \ref{equ.2.14}\right)  $
over $\left[  0,T\right]  .$
\end{lemma}

\begin{proof}
We can augment the trace solution $x$ to a full rough path solution
$\mathbf{X}:=\left(  x,\mathbb{X}\right)  $ as follows. Let
\begin{align*}
\mathbb{A}_{s,t}  &  :=\left[  Y\left(  x_{s}\right)  \otimes Y\left(
x_{s}\right)  \right]  \mathbb{Z}_{s,t}\text{ and }\\
\mathbf{A}_{s,t}  &  :=1+x_{s,t}+\mathbb{A}_{s,t}.
\end{align*}
Note that $Y$ is bounded on $x$ and therefore $\mathbf{A}$ has finite
$p-$variation in the sense of $\left(  \ref{equ.2.2}\right)  .$ It now
suffices to check that $\mathbf{A}$ is an almost multiplicative functional in
the language of Lyons. For this it will be enough to check that that
$\mathbb{A}$ approximately (in the sense of Notation\ \ref{not.2.10})
satisfies Chen's identity, which we now do. If $0\leq s\leq t\leq u\leq T,$
then%
\[
\mathbb{A}_{t,u}=\left[  Y\left(  x_{t}\right)  \otimes Y\left(  x_{t}\right)
\right]  \mathbb{Z}_{t,u}\simeq\left[  Y\left(  x_{s}\right)  \otimes Y\left(
x_{s}\right)  \right]  \mathbb{Z}_{t,u}%
\]
so that
\begin{align}
\mathbb{A}_{s,u}-\mathbb{A}_{s,t}-\mathbb{A}_{t,u}  &  \simeq\left[  Y\left(
x_{s}\right)  \otimes Y\left(  x_{s}\right)  \right]  \left[  \mathbb{Z}%
_{s,u}-\mathbb{Z}_{s,t}-\mathbb{Z}_{t,u}\right] \nonumber\\
&  =\left[  Y\left(  x_{s}\right)  \otimes Y\left(  x_{s}\right)  \right]
\left[  z_{s,t}\otimes z_{t,u}\right]  . \label{equ.2.16}%
\end{align}
Similarly we have%
\[
x_{s,t}\otimes x_{t,u}\simeq Y\left(  x_{s}\right)  z_{s,t}\otimes Y\left(
x_{t}\right)  z_{t,u}\simeq Y\left(  x_{s}\right)  z_{s,t}\otimes Y\left(
x_{s}\right)  z_{t,u}%
\]
which combined with Eq. (\ref{equ.2.16}) shows%
\[
\mathbb{A}_{s,u}-\mathbb{A}_{s,t}-\mathbb{A}_{t,u}-x_{s,t}\otimes
x_{t,u}\simeq0
\]
which is to say that $\mathbf{A}_{s,t}$ is an almost multiplicative
functional. Thus by Theorem \ref{the.2.13} there exists $\mathbb{X}_{s,t}$
such that $\mathbb{X}_{s,t}\simeq\mathbb{A}_{s,t}$ and $\mathbf{X}%
_{s,t}=1+x_{s,t}+\mathbb{X}_{s,t}$ solves the RDE in Eq. (\ref{equ.2.14}).
\end{proof}

\section{Geometric and weakly geometric rough paths on manifolds\label{sec.3}}

In this section we will introduce the notions of geometric and weakly
geometric rough paths on manifolds. The section is split in four parts.
Subsection \ref{sub.3.1} introduces the basic geometric notations and facts
needed for the rest of the paper. The definitions of constrained rough paths
(now called geometric and weakly geometric rough paths) and their path
integrals are the introduced in Subsection \ref{sub.3.2}, see Definitions
\ref{def.3.15}, \ref{def.3.17}, and \ref{def.3.24}. Basic properties of these
definitions are then established. The main result of Subsection \ref{sub.3.3},
is Proposition \ref{pro.3.35} which gives a more effective criteria for
checking that an ambient rough path is in fact a weakly geometric rough path.
The final subsection (\ref{sub.3.4}) explores the behavior of constrained
rough paths under change of coordinates and more general smooth
transformations, see Proposition \ref{pro.3.38}. This result is then used to
demonstrate that our constrained rough paths may be formulated to be
independent of the choice of embedding, see Corollaries \ref{cor.3.40},
\ref{cor.3.41}, and Definition \ref{def.3.42}.

\subsection{Basic geometric definitions\label{sub.3.1}}

Let $E=\mathbb{R}^{N}$ and $E^{\prime}=\mathbb{R}^{N^{\prime}}$ be Euclidean
spaces and let $\langle a,b\rangle=a\cdot b=\sum_{i=1}^{N}a_{i}b_{i}$ for all
$a,b\in E.$ If $U$ is an open neighborhood in $E$ and $F:U\rightarrow
E^{\prime}$ is a smooth map, then for $x\in U$ and $v\in E$ we let
$\partial_{v}F\left(  x\right)  :=\frac{d}{dt}|_{0}F\left(  x+tv\right)  $ be
the directional derivative of $F$ at $x$ along $v.$ We will further let
$F^{\prime}\left(  x\right)  :E\rightarrow E^{\prime}$ and $F^{\prime\prime
}\left(  x\right)  :E\otimes E\rightarrow V:=\mathbb{R}^{N-d}$ be the
differential and Hessian of $F$ respectively which are defined by $F^{\prime
}\left(  x\right)  v:=\left(  \partial_{v}F\right)  \left(  x\right)  $ and
$F^{\prime\prime}\left(  x\right)  \left[  v\otimes w\right]  :=\left(
\partial_{v}\partial_{w}F\right)  \left(  x\right)  $ for all $x\in U$ and
$v,w\in E.$

Throughout the rest of this paper, $M^{d}$ will be a $d$ -- dimensional
embedded submanifold of a Euclidean space $E:=\mathbb{R}^{N}.$ The reader may
find the necessary geometric background in any number of places including,
\cite{Lee2013,Warner1983,Driver2004}. To fix notation let us recall a
formulation an embedded submanifold which will be most useful for our purposes.

\begin{definition}
\label{def.3.1}A subset $M$ of $E$ is an \textbf{embedded submanifold} of $E$
of dimension $d\in\left\{  1,\dots,N\right\}  $ provided for each $m\in M$
there is an open neighborhood $U$ in $E$ containing $m$ and smooth
\textbf{local defining function} $F:U\rightarrow\mathbb{R}^{N-d}$ such that
\[
U\cap M=\left\{  x\in U:F\left(  x\right)  =0\right\}
\]
and $F^{\prime}\left(  x\right)  :E\rightarrow\mathbb{R}^{N-d}$ is surjective
for $x\in U.$
\end{definition}

Recall that the \textbf{tangent plane }to $M$ at $m\in M$ is $\tau
_{m}M:=\operatorname*{Nul}\left(  F^{\prime}\left(  m\right)  \right)  .$
Because of the implicit function theorem, to each $v\in\tau_{m}M$ there exists
a smooth path $\sigma_{v}:\left(  -\varepsilon,\varepsilon\right)  \rightarrow
E$ such that $\sigma_{v}\left(  \left(  -\varepsilon,\varepsilon\right)
\right)  \subset M,$ $\sigma_{v}\left(  0\right)  =m,$ and $\sigma_{v}%
^{\prime}\left(  0\right)  =v.$ From these considerations, one shows $\tau
_{m}M\ni v\rightarrow\dot{\sigma}_{v}\left(  0\right)  \in T_{m}M$ is a linear
isomorphism of vector spaces; we will often use this isomorphism to identify
$\tau_{m}M$ with $T_{m}M.$

\begin{remark}
\label{rem.3.2}Around each point $m\in M$ there exists an open set $U$ in $E$
and a smooth map $\pi:U\rightarrow M\cap U$ such that $\pi\left(  x\right)
=x$ for all $x\in M\cap U.$ As a consequence of this fact, any smooth function
$f:M\rightarrow W$ defined near $m$ has a smooth extension, $f\circ\pi,$ to a
neighborhood of $\ m$ in $E.$
\end{remark}

\begin{notation}
\label{not.3.3}Letting $F:U\rightarrow\mathbb{R}^{N-d}$ be a local defining
function for $M$ as above, we define smooth functions $Q_{F},P_{F}%
:U\rightarrow\operatorname*{End}\left(  E\right)  $ by%
\begin{align}
Q_{F}(x)  &  :=F^{\prime}(x)^{\ast}\left(  F^{\prime}(x)F^{\prime}(x)^{\ast
}\right)  ^{-1}F^{\prime}(x)\text{ and}\label{equ.3.1}\\
P_{F}\left(  x\right)   &  :=I_{E}-Q_{F}\left(  x\right)  =I-F^{\prime
}(x)^{\ast}\left(  F^{\prime}(x)F^{\prime}(x)^{\ast}\right)  ^{-1}F^{\prime
}(x). \label{equ.3.2}%
\end{align}

\end{notation}

\begin{remark}
\label{rem.3.4}We make a number of comments.

\begin{enumerate}
\item The surjectivity assumption of $F^{\prime}\left(  x\right)  $ guarantees
that $F^{\prime}(x)F^{\prime}(x)^{\ast}$ is invertible.

\item One may easily verify that $Q_{F}\left(  x\right)  $ is orthogonal
projection onto $\operatorname*{Ran}\left(  F^{\prime}(x)^{\ast}\right)
=\operatorname*{Nul}\left(  F^{\prime}\left(  x\right)  \right)  ^{\perp}$ and
$P_{F}\left(  x\right)  $ is orthogonal projection onto $\operatorname*{Nul}%
\left(  F^{\prime}\left(  x\right)  \right)  .$

\item For $m\in U\cap M$ we have that $P_{F}\left(  m\right)  $ $\left(
Q_{F}\left(  m\right)  \right)  $ is the orthogonal projection onto $\tau
_{m}M$ $\left(  \left[  \tau_{m}M\right]  ^{\perp}\right)  $ and hence is
independent of the choice of local defining function. We will simply write
$P\left(  m\right)  $ and $Q\left(  m\right)  $ (or, sometimes, $P_{m}$ and
$Q_{m})$ for $P_{F}\left(  m\right)  $ and $Q_{F}\left(  m\right)  $ when
$m\in M.$
\end{enumerate}
\end{remark}

\begin{remark}
\label{rem.3.5}In the proofs that follow we will often use the following
identities%
\begin{equation}
F^{\prime}\left(  x\right)  =F^{\prime}\left(  x\right)  Q_{F}\left(
x\right)  \text{ and }Q_{F}\left(  x\right)  =A_{F}\left(  x\right)
F^{\prime}\left(  x\right)  \label{equ.3.3}%
\end{equation}
which hold for all $x\in M,$ where
\begin{equation}
A_{F}\left(  x\right)  :=F^{\prime}(x)^{\ast}(F^{\prime}(x)F^{\prime}%
(x)^{\ast})^{-1}\in\operatorname{Hom}\left(  \mathbb{R}^{N-d},E\right)  .
\label{equ.3.4}%
\end{equation}

\end{remark}

The last geometric notions we need are vector fields, one forms, and their
covariant derivatives.

\begin{definition}
[Vector Fields]\label{def.3.6}A\textbf{ smooth vector field} on $M$ is a
smooth function $Y:M\rightarrow E$ such that $Q\left(  m\right)  Y\left(
m\right)  =0$ for all $m\in M,$ i.e. $Y\left(  m\right)  \in T_{m}M$ for all
$m\in M.$ Let $\Gamma\left(  TM\right)  $ denote the collection of smooth
vector fields\textbf{ }on $M.$
\end{definition}

\begin{example}
\label{exa.3.7}For $z\in\mathbb{R}^{N}$ we let $V_{z}\in\Gamma\left(
TM\right)  $ be defined by $V_{z}\left(  x\right)  :=P_{x}z$ for all $x\in M.$
\end{example}

\begin{definition}
\label{def.3.8}A\textbf{ smooth one form} on $M$ with values in a finite
dimensional vector space $W$ is a smooth function $\alpha$ on $M$ with
$\alpha_{m}\in\operatorname{Hom}\left(  T_{m}M,W\right)  $ for all $m\in M.$
Here we can describe the smoothness assumption of $\alpha$ by requiring $M\in
m\rightarrow\alpha_{m}P_{m}\in\operatorname{Hom}\left(  E,W\right)  $ to be a
smooth function. Let $\Omega^{1}\left(  M,W\right)  $ denote the set of smooth
one forms on $M$ with values in $W.$
\end{definition}

\begin{example}
\label{exa.3.9}The function $\alpha_{m}:=P_{m}$ is in $\Omega^{1}\left(
M,E\right)  .$ If $f:M\rightarrow W$ is a smooth function then $\alpha
:=df\in\Omega^{1}\left(  M,W\right)  $ where as usual $df\left(  v_{m}\right)
=v_{m}f.$
\end{example}

\begin{definition}
[Levi-Civita Covariant Derivative]\label{def.3.10}Suppose that $v_{m}\in
T_{m}M,$ $Y\in\Gamma\left(  TM\right)  ,$ and $\alpha\in\Omega^{1}\left(
M,W\right)  ,$ then the covariant derivative at $v_{m}$ of $Y$ and $\alpha$
are given respectively by%
\[
\nabla_{v_{m}}Y=P_{m}\left(  \partial_{v}Y\right)  \left(  m\right)  \in
T_{m}M\text{ and }\nabla_{v_{m}}\alpha=\partial_{v_{m}}\left(  \alpha\circ
P\right)  \in\operatorname{Hom}\left(  T_{m}M,W\right)  .
\]

\end{definition}

The next lemma and proposition records some basic well known properties of the
Levi-Civita covariant derivative.

\begin{lemma}
\label{lem.3.11}If $P$ and $Q$ be the orthogonal projection operators as in
Remark \ref{rem.3.4}, then $dP=-dQ$ and $PdQ=dQP.$
\end{lemma}

\begin{proof}
Differentiate the identities $I=P+Q$ and $0=PQ,$ which hold on $M$ gives the
new identities in the statement.
\end{proof}

\begin{proposition}
\label{pro.3.12}Let $Y\in\Gamma\left(  TM\right)  ,$ $\alpha\in\Omega
^{1}\left(  M,W\right)  ,$ and $\Gamma:=dQ\in\Omega^{1}\left(
M,\operatorname*{End}\left(  E\right)  \right)  .$ Then;

\begin{enumerate}
\item $\nabla_{v_{m}}Y=\left(  \partial_{v}Y\right)  \left(  m\right)
+\Gamma\left(  v_{m}\right)  Y\left(  m\right)  ,$

\item The product rule holds;%
\[
v_{m}\left(  \alpha\left(  Y\right)  \right)  =\left(  \nabla_{v_{m}}%
\alpha\right)  \left(  Y\left(  m\right)  \right)  +\alpha_{m}\left(
\nabla_{v_{m}}Y\right)  .
\]

\item If $\alpha_{m}=\tilde{\alpha}_{m}|_{T_{m}M}$ where $\tilde{\alpha
}:M\rightarrow\operatorname{Hom}\left(  E,W\right)  $ is a smooth function
then
\[
\nabla_{v_{m}}\alpha=\left(  \partial_{v}\tilde{\alpha}\right)  \left(
m\right)  P_{m}-\tilde{\alpha}_{m}\Gamma\left(  v_{m}\right)  .
\]

\item If $\alpha_{m}=\tilde{\alpha}_{m}|_{T_{m}M}$ as in item 3. and we
further assume that $\tilde{\alpha}_{x}=\tilde{\alpha}_{x}\circ P_{x}$ for
$x\in M$ near $m,$ then%
\[
\left(  \nabla_{v}\alpha\right)  \left(  w\right)  =\tilde{\alpha}_{m}%
^{\prime}\left[  v\otimes w\right]  \text{ for all }v,w\in T_{m}M.
\]

\end{enumerate}
\end{proposition}

\begin{proof}
In the proof below recall that $Y=PY$ as $Y\in\Gamma\left(  TM\right)  .$

\begin{enumerate}
\item Differentiating the identity, $Y=PY,$ shows
\[
\left(  \partial_{v}Y\right)  \left(  m\right)  =dP\left(  v_{m}\right)
Y\left(  m\right)  +P\left(  m\right)  \left(  \partial_{v}Y\right)  \left(
m\right)  =-\Gamma\left(  v_{m}\right)  Y\left(  m\right)  +\nabla_{v_{m}}Y
\]
which proves item 1.

\item Since $\alpha\left(  Y\right)  =\alpha\left(  PY\right)  =\left(  \alpha
P\right)  Y$ and $\alpha P:M\rightarrow\operatorname*{End}\left(  E\right)  $
is a smooth function the ordinary product rule shows,%
\begin{align*}
v_{m}\left(  \alpha\left(  Y\right)  \right)   &  =\left(  \partial_{v}\left(
\alpha P\right)  \left(  m\right)  \right)  Y\left(  m\right)  +\alpha
_{m}P_{m}\left(  \partial_{v}Y\right)  \left(  m\right) \\
&  =\left(  \nabla_{v_{m}}\alpha\right)  \left(  Y\left(  m\right)  \right)
+\alpha_{m}\left(  \nabla_{v_{m}}Y\right)  .
\end{align*}

\item If $\alpha_{m}=\tilde{\alpha}_{m}|_{T_{m}M}$ as in item 3., then using
the standard product rule again,
\[
\nabla_{v_{m}}\alpha=\partial_{v}\left(  \alpha P\right)  \left(  m\right)
=\partial_{v}\left(  \tilde{\alpha}P\right)  \left(  m\right)  =\left(
\partial_{v}\tilde{\alpha}\right)  _{m}P_{m}+\tilde{\alpha}_{m}dP\left(
v_{m}\right)  =\left(  \partial_{v}\tilde{\alpha}\right)  \left(  m\right)
P_{m}-\tilde{\alpha}_{m}\Gamma\left(  v_{m}\right)  .
\]

\item From the definitions,%
\[
\left(  \nabla_{v}\alpha\right)  \left(  w\right)  =\left[  v_{m}\left(
\alpha\circ P\right)  \right]  w=\left[  v_{m}\left(  \tilde{\alpha}\circ
P\right)  \right]  w=\left[  v_{m}\tilde{\alpha}\right]  w=\tilde{\alpha}%
_{m}^{\prime}\left(  v\otimes w\right)  .
\]

\end{enumerate}
\end{proof}

\subsection{(Weakly) Geometric Rough Paths on $M$\label{sub.3.2}}

In the following let $M$ be a manifold embedded in $E:=\mathbb{R}^{N}$ and $F$
the (local) defining function as introduced in Notation \ref{not.3.3}. In the
setting of embedded manifolds there is a natural notion of geometric rough
paths that is induced by the rough metric on the ambient Euclidean space $E$.
To help prepare the precise definition of a geometric rough path on a manifold
we introduce the following set of paths.

Assume that $M\subseteq E$ and let $C_{bv}\left(  \left[  0,T\right]
,E\right)  $ denote the set of continuous bounded variation paths taking
values in $E.$ Recall the definition of the truncated signature $S_{2}$ in
$\left(  \ref{equ.2.4}\right)  .$ For any real number $p\in\lbrack2,3),$ we
define $\bar{G}_{p}\left(  M\right)  $ to be closure of the lifts of
continuous bounded variation paths in $M;$ that is, $\bar{G}_{p}\left(
M\right)  $ is the closure of
\[
\left\{  S_{2}\left(  x\right)  :x\in C_{bv}\left(  \left[  0,T\right]
,E\right)  ,x_{t}\in M\text{ for all }t\in\left[  0,T\right]  \right\}
\]
with respect to the topology induced by the $p-$ variation rough path metric
on $E$.

\begin{remark}
\label{rem.3.13}For $p>1$ continuous bounded variation paths on $E$ are in the
closure of the smooth paths taken in the $p-$variation metric (see e.g.
\cite[Lemma 5.30]{FV}). In addition the truncated signature is a locally
Lipschitz continuous map under the (inhomogeneous) rough $p-$ variation metric
(see e.g. \cite[Theorem 9.10]{FV}). Combining these two facts shows we could
have replaced $C_{bv}\left(  \left[  0,T\right]  ,E\right)  $ in the
definition of $\bar{G}_{p}\left(  M\right)  $ by the smooth paths $C^{\infty
}\left(  \left[  0,T\right]  ,E\right)  .$ This justifies referring to the
lifts of $1-$ rough paths as \textquotedblleft smooth\textquotedblright\ rough paths.
\end{remark}

\begin{lemma}
\label{lem.3.14}Suppose $M$ is a closed subset of $E,$ then the trace of any
$\mathbf{X}$ in $\bar{G}_{p}\left(  M\right)  $ lies in $M.$
\end{lemma}

\begin{proof}
By definition $\mathbf{X}$ can be approximated by a sequence of smooth rough
paths $\mathbf{X}^{n}$ (see Remark \ref{rem.3.13}) with trace in $M.$ The
traces of the approximating sequence converges in $p-$variation and therefore
also converges pointwise. Since $M$ is assumed to be closed, the proof is complete.
\end{proof}

\begin{definition}
[Geometric rough paths]\label{def.3.15}We define \textbf{geometric }%
$p-$\textbf{rough paths} on $M$ to be those elements of $\bar{G}_{p}\left(
M\right)  $ whose trace $x$ lies inside $M.$ The set of geometric $p-$rough
paths on $M$ will be denoted by $G_{p}\left(  M\right)  .$ In other words, we
have%
\[
G_{p}\left(  M\right)  =\left\{  \mathbf{X=}\left(  x,\mathbb{X}\right)
\in\bar{G}_{p}\left(  M\right)  :x_{t}\in M\text{ for all }t\in\left[
0,T\right]  \right\}  .
\]

\end{definition}

It follows from Lemma \ref{lem.3.14} that $\bar{G}_{p}\left(  M\right)
=G_{p}\left(  M\right)  $ when $M$ is a closed subset of $E.$ The next example
explains why it is important that we take the closure of paths in $M$, and why
it will not be sufficient to \textit{only} assume that the trace of limiting
object lies in $M.$

\begin{example}
\label{exa.3.16}Let $M=\left\{  e_{1},e_{2}\right\}  ^{\perp}\subset
\mathbb{R}^{N}.$ Then for any $v,w\in\mathbb{R}^{N}$ there exists (see
\cite{LCL},\cite{FH}) a so-called pure area geometric rough path,
$\mathbf{X=}\left(  x,\mathbb{X}\right)  $ with the property that $x=0$, the
constant path zero, and
\[
\mathbb{X}_{s,t}=\left(  v\otimes w-w\otimes v\right)  \left(  t-s\right)  .
\]
On the other hand if $\mathbf{X\in}WG_{p}\left(  M\right)  $ we would
certainly have $\mathbb{X}_{s,t}\in M\otimes M$ for all $s$ and $t.$ Put
another way if $\mathbf{X\in}WG_{p}\left(  M\right)  $ then $\left[  Q\otimes
I\right]  \mathbb{X}_{s,t}=0=\left[  I\otimes Q\right]  \mathbb{X}_{s,t}$
where $Q$ is orthogonal projection onto $M^{\perp}=\operatorname*{span}%
\left\{  e_{1},e_{2}\right\}  .$ An approximate version of this requirement
will appear again in the general manifold setting as well, see Corollary
\ref{cor.3.20} below.
\end{example}

A second set of rough paths on a manifold is, in structure, related to the
weakly geometric rough paths in the classical Banach space setting.

\begin{definition}
[Weakly geometric rough paths]\label{def.3.17}We say that $\mathbf{X=}\left(
x,\mathbb{X}\right)  $ is a \textbf{weakly geometric }$p-$\textbf{ rough path}
on the manifold $M$ if: $\mathbf{X}$ is in\ $WG_{p}\left(  E\right)  $, its
trace $x$ lies in $M$ and for any finite dimensional subspace $W$ and any
$\tilde{\alpha}\in\Omega^{1}\left(  E,W\right)  $ such that $\tilde{\alpha
}|_{TM}\equiv0$ we have $\int\tilde{\alpha}\left(  \mathbf{dX}\right)
\equiv0.$ The set of weakly geometric rough paths will be denoted by
$WG_{p}\left(  M\right)  .$
\end{definition}

In the following we will often make use of the following simple consequence of
Taylor's theorem.

\begin{lemma}
\label{lem.3.18}If $f:E:=\mathbb{R}^{N}\rightarrow\mathbb{R}^{l}$ is a $C^{3}$
-- function which is constant on $M\subset E,$ then for $x,y\in M$ we have,

\begin{enumerate}
\item $f^{\prime}\left(  x\right)  \left(  y-x\right)  =O\left(  \left\vert
y-x\right\vert ^{2}\right)  ,$ and

\item $f^{\prime}\left(  x\right)  \left(  y-x\right)  +\frac{1}{2}%
f^{\prime\prime}\left(  x\right)  \left[  \left(  y-x\right)  \otimes\left(
y-x\right)  \right]  =O\left(  \left\vert y-x\right\vert ^{3}\right)  .$
\end{enumerate}
\end{lemma}

\begin{proof}
By Taylor's theorem,%
\[
f\left(  y\right)  -f\left(  x\right)  =f^{\prime}\left(  x\right)  \left(
y-x\right)  +O\left(  \left\vert y-x\right\vert ^{2}\right)
\]
and%
\[
f\left(  y\right)  -f\left(  x\right)  =f^{\prime}\left(  x\right)  \left(
y-x\right)  +\frac{1}{2}f^{\prime\prime}\left(  x\right)  \left[  \left(
y-x\right)  \otimes\left(  y-x\right)  \right]  +O\left(  \left\vert
y-x\right\vert ^{3}\right)  .
\]
Since $f$ is constant on $M$ and $x,y\in M,$ it follows that $f\left(
y\right)  -f\left(  x\right)  =0$ and the results follow from the previously
displayed equations.
\end{proof}

An obvious class of one forms having the property that $\alpha|_{TM}\equiv0$
are those which locally have the form $\alpha=\varphi F^{\prime}$ , where
$\varphi$ is a smooth function and $F$ is a local defining function for the
manifold. The following lemma gives simplified description of the level-one
component for the integral of any such one form.

\begin{lemma}
\label{lem.3.19}Let $\mathbf{X=}\left(  x,\mathbb{X}\right)  \in WG_{p}\left(
E\right)  $ be a weakly geometric $p-$rough path such that the trace $x$ is in
$M.$ Suppose that $F\in C^{\infty}\left(  U,V\right)  ,$ with $V=\mathbb{R}%
^{N-d},$ is a smooth function which locally defines $M$ as in Definition
\ref{def.3.1} and which has been chosen so that there is a subinterval
$\left[  s,t\right]  \subseteq\left[  0,T\right]  $ with%
\[
\left\{  x_{u}:u\in\left[  s,t\right]  \right\}  \subset U.
\]
Assume $W$ is a finite dimensional vector space, and suppose $\varphi\in
C^{\infty}\left(  U,\operatorname{Hom}\left(  V,W\right)  \right)  .$ Let
$\alpha\in\Omega^{1}\left(  E,W\right)  $ be the one form defined by
\[
\alpha\left(  x\right)  \xi=\varphi\left(  x\right)  F^{\prime}\left(
x\right)  \xi\in W\text{ for all }\xi\in E.
\]
Then for every $\left[  u,v\right]  \subseteq\left[  s,t\right]  $ we have
\begin{equation}
\left[  \int\alpha\left(  d\mathbf{X}\right)  \right]  _{u,v}^{1}\simeq
\alpha\left(  x_{u}\right)  x_{u,v}+\alpha^{\prime}\left(  x_{s}\right)
\mathbb{X}_{u,v}\simeq\left[  \left(  \varphi^{\prime}\cdot F^{\prime}\right)
\left(  x_{u}\right)  \right]  \mathbb{X}_{u,v} \label{equ.3.5}%
\end{equation}
where $\left(  \varphi^{\prime}\cdot F^{\prime}\right)  \left(  m\right)  $
denotes the linear map from $E\otimes E\rightarrow W$ determined by%
\begin{equation}
\left[  \left(  \varphi^{\prime}\cdot F^{\prime}\right)  \left(  m\right)
\right]  \xi_{1}\otimes\xi_{2}:=\left[  \varphi^{\prime}\left(  m\right)
\xi_{1}\right]  \left[  F^{\prime}\left(  m\right)  \xi_{2}\right]  =\left(
\partial_{\xi_{1}}\varphi\right)  \left(  m\right)  \left(  \partial_{\xi_{2}%
}F\right)  \left(  m\right)  . \label{equ.3.6}%
\end{equation}

\end{lemma}

\begin{proof}
The product rule (written in the notation introduced in Eq. (\ref{equ.3.6}))
gives%
\begin{equation}
\alpha^{\prime}=\varphi F^{\prime\prime}+\varphi^{\prime}\cdot F^{\prime}.
\label{equ.3.7}%
\end{equation}
This identity combined with Eq. (\ref{equ.2.7}) then implies,%
\begin{align}
\left[  \int\alpha\left(  d\mathbf{X}\right)  \right]  _{u,v}^{1}\simeq &
\alpha\left(  x_{u}\right)  x_{u,v}+\alpha^{\prime}\left(  x_{u}\right)
\mathbb{X}_{u,v}\nonumber\\
&  =\varphi\left(  x_{u}\right)  F^{\prime}\left(  x_{u}\right)
x_{u,v}+\left[  \varphi\left(  x_{u}\right)  F^{\prime\prime}\left(
x_{u}\right)  +\left(  \varphi^{\prime}\cdot F^{\prime}\right)  \left(
x_{u}\right)  \right]  \mathbb{X}_{u,v}\nonumber\\
&  =\varphi\left(  x_{u}\right)  \left[  F^{\prime}\left(  x_{u}\right)
x_{u,v}+F^{\prime\prime}\left(  x_{u}\right)  \mathbb{X}_{u,v}\right]
+\left[  \left(  \varphi^{\prime}\cdot F^{\prime}\right)  \left(
x_{u}\right)  \right]  \mathbb{X}_{u,v}. \label{equ.3.8}%
\end{align}
Since $F^{\prime\prime}$ is symmetric and $\mathbf{X}=\left(  x,\mathbb{X}%
\right)  $ is a weakly geometric rough path it follows that
\[
F^{\prime\prime}\left(  x_{u}\right)  \mathbb{X}_{u,v}=F^{\prime\prime}\left(
x_{u}\right)  \mathbb{X}_{u,v}^{s}=\frac{1}{2}F^{\prime\prime}\left(
x_{u}\right)  \left[  x_{u,v}\otimes x_{u,v}\right]
\]
and therefore by Lemma \ref{lem.3.18},
\[
F^{\prime}\left(  x_{u}\right)  x_{u,v}+F^{\prime\prime}\left(  x_{u}\right)
\mathbb{X}_{u,v}=F^{\prime}\left(  x_{u}\right)  x_{u,v}+\frac{1}{2}%
F^{\prime\prime}\left(  x_{u}\right)  \left[  x_{u,v}\otimes x_{u,v}\right]
\simeq0.
\]
Combining this estimate with Eq. (\ref{equ.3.8}) gives (\ref{equ.3.5}).
\end{proof}

\begin{corollary}
\label{cor.3.20}Let $\mathbf{X=}\left(  x,\mathbb{X}\right)  \in WG_{p}\left(
M\right)  \subset WG_{p}\left(  E\right)  $ and $F:U\rightarrow V:=\mathbb{R}%
^{N-d}$ and $\left[  s,t\right]  \subseteq\left[  0,T\right]  $ be as in Lemma
\ref{lem.3.19}. Then for $s\leq u\leq v\leq t,$%
\begin{align}
I_{E}\otimes F^{\prime}\left(  x_{u}\right)  \mathbb{X}_{u,v}  &
\simeq0\simeq F^{\prime}\left(  x_{u}\right)  \otimes I_{E}\mathbb{X}%
_{u,v}\text{ and }\label{equ.3.9}\\
I_{E}\otimes Q\left(  x_{u}\right)  \mathbb{X}_{u,v}  &  \simeq0\simeq
Q\left(  x_{u}\right)  \otimes I_{E}\mathbb{X}_{u,v}, \label{equ.3.10}%
\end{align}
where $Q$ is defined in Notation \ref{not.3.3} and Remark \ref{rem.3.4}.
\end{corollary}

\begin{proof}
Choose $\varphi\in C_{c}^{\infty}\left(  U,E\right)  $ such that
$\varphi\left(  x\right)  =x$ for $x$ in a neighborhood $\left\{  x_{u}%
:u\in\left[  s,t\right]  \right\}  $ and let $\alpha\in\Omega^{1}\left(
E,E\otimes V\right)  $ be defined by
\[
\alpha\left(  \xi_{x}\right)  :=\varphi\left(  x\right)  \otimes dF\left(
\xi_{x}\right)  =\varphi\left(  x\right)  \otimes F^{\prime}\left(  x\right)
\xi\text{ for all }\xi_{x}\in TE\cong E\times E.
\]
Then $\alpha|_{TM}=0$ and therefore $\int\alpha\left(  d\mathbf{X}\right)
\equiv0.$ By Theorem \ref{the.2.14}, Lemma \ref{lem.3.19}, and the fact that
$\varphi^{\prime}\left(  x_{u}\right)  =I_{E},$ it follows that%
\[
0\simeq\alpha\left(  x_{u}\right)  x_{u,v}+\alpha^{\prime}\left(
x_{u}\right)  \mathbb{X}_{u,v}\simeq\left[  \left(  \varphi^{\prime}\otimes
F^{\prime}\right)  \left(  x_{u}\right)  \right]  \mathbb{X}_{u,v}=I\otimes
F^{\prime}\left(  x_{u}\right)  \mathbb{X}_{u,v}%
\]
for $s\leq u\leq v\leq t$ and the left member of Eq. (\ref{equ.3.9}) is
proved. This also easily proves the left member of Eq. (\ref{equ.3.10}) since
\[
I_{E}\otimes Q\left(  x_{u}\right)  \mathbb{X}_{u,v}=\left[  I_{E}\otimes
A_{F}\left(  x_{u}\right)  \right]  I_{E}\otimes F^{\prime}\left(
x_{u}\right)  \mathbb{X}_{u,v}%
\]
where $A_{F}\left(  x\right)  :=F^{\prime}(x)^{\ast}(F^{\prime}(x)F^{\prime
}(x)^{\ast})^{-1}\in\operatorname{Hom}\left(  \mathbb{R}^{N-d},E\right)  $ as
in Remark \ref{rem.3.5}. The other approximate identities in Eqs.
(\ref{equ.3.9}) and (\ref{equ.3.10}) follow similarly, one need only now
define $\alpha\in\Omega^{1}\left(  E,V\otimes E\right)  $ by
\[
\alpha\left(  \xi_{x}\right)  :=dF\left(  \xi_{x}\right)  \otimes
\varphi\left(  x\right)  =F^{\prime}\left(  x\right)  \xi\otimes\varphi\left(
x\right)  \text{ for all }\xi_{x}\in TE\cong E\times E.
\]

\end{proof}

\begin{remark}
\label{rem.3.21}The conditions in Eqs. (\ref{equ.3.9}) and (\ref{equ.3.10})
are equivalent. Indeed the proof of Corollary \ref{cor.3.20} has already shown
Eq. (\ref{equ.3.9}) implies Eq. (\ref{equ.3.10}). For the converse direction
we need only observe that $F^{\prime}\left(  x_{u}\right)  =F^{\prime}\left(
x_{u}\right)  Q\left(  x_{u}\right)  $ so that, for example,
\[
I_{E}\otimes F^{\prime}\left(  x_{u}\right)  \mathbb{X}_{u,v}=\left[
I_{E}\otimes F^{\prime}\left(  x_{u}\right)  \right]  I_{E}\otimes Q\left(
x_{u}\right)  \mathbb{X}_{u,v}.
\]

\end{remark}

\begin{corollary}
\label{cor.3.22}If $\mathbf{X}=\left(  x,\mathbb{X}\right)  \in WG_{p}\left(
M\right)  ,$ then $\mathbb{X}_{s,t}\simeq\left[  P_{x_{s}}\otimes P_{x_{s}%
}\right]  \mathbb{X}_{s,t}.$
\end{corollary}

\begin{proof}
The result follows by observing that%
\[
\mathbb{X}_{s,t}=\left(  P_{x_{s}}+Q_{x_{s}}\right)  \otimes\left(  P_{x_{s}%
}+Q_{x_{s}}\right)  \mathbb{X}_{s,t}%
\]
and then using Eq. (\ref{equ.3.10}) to conclude $\left(  I_{E}\otimes
Q_{x_{s}}\right)  \mathbb{X}_{s,t}\simeq0,\ \left(  Q_{x_{s}}\otimes
I_{E}\right)  \mathbb{X}_{s,t}\simeq0,$ and
\[
\left(  Q_{x_{s}}\otimes Q_{x_{s}}\right)  \mathbb{X}_{s,t}=\left(
I_{E}\otimes Q_{x_{s}}\right)  \left(  Q_{x_{s}}\otimes I_{E}\right)
\mathbb{X}_{s,t}\simeq0.
\]

\end{proof}

The following lemma prepares the definition of the integral of a rough path
against smooth one forms.

\begin{lemma}
\label{lem.3.23}Suppose $\mathbf{X}\in WG_{p}\left(  M\right)  ,$ $U$ is an
open neighborhood of $M,$ and $\alpha,~\beta\in\Omega^{1}\left(  U,W\right)
.$ If $\alpha|_{TM}=\beta|_{TM},$ then
\begin{equation}
\int\alpha\left(  d\mathbf{X}\right)  =\int\beta\left(  d\mathbf{X}\right)  .
\label{equ.3.11}%
\end{equation}

\end{lemma}

\begin{proof}
The one form, $\psi:=\alpha-\beta\in\Omega^{1}\left(  U,W\right)  ,$ vanishes
on $TM$ and so by Definition \ref{def.3.17} $\int\psi\left(  d\mathbf{X}%
\right)  \equiv\mathbf{0}.$ As the rough path integral is linear on
$\Omega^{1}\left(  U,W\right)  $ at level one it immediately follows that%
\[
0=\left[  \int\psi\left(  d\mathbf{X}\right)  \right]  _{s,t}^{1}=\left[
\int\alpha\left(  d\mathbf{X}\right)  \right]  _{s,t}^{1}-\left[  \int%
\beta\left(  d\mathbf{X}\right)  \right]  _{s,t}^{1}.
\]
Moreover by Corollary \ref{cor.3.22},%
\begin{align*}
\left[  \int\alpha\left(  d\mathbf{X}\right)  \right]  _{s,t}^{2}  &
\simeq\alpha_{x_{s}}\otimes\alpha_{x_{s}}\mathbb{X}_{s,t}\simeq\alpha_{x_{s}%
}\otimes\alpha_{x_{s}}\left[  P_{x_{s}}\otimes P_{x_{s}}\right]
\mathbb{X}_{s,t}\\
&  =\beta_{x_{s}}\otimes\beta_{x_{s}}\left[  P_{x_{s}}\otimes P_{x_{s}%
}\right]  \mathbb{X}_{s,t}\simeq\beta_{x_{s}}\otimes\beta_{x_{s}}%
\mathbb{X}_{s,t}\\
&  \simeq\left[  \int\beta\left(  d\mathbf{X}\right)  \right]  _{s,t}^{2}.
\end{align*}
The last two displayed equations along with Lemma \ref{lem.2.12} now gives Eq.
(\ref{equ.3.11}).
\end{proof}

Another proof of this lemma could be given along the lines of Proposition
\ref{pro.3.29} below. The previous lemma justifies the following definition of
integration $\alpha\in\Omega^{1}\left(  M,W\right)  $ along a weakly geometric
rough path $\mathbf{X\in}WG_{p}\left(  M\right)  .$

\begin{definition}
\label{def.3.24}The \textbf{rough path integral }of a rough path
$\mathbf{X\in}WG_{p}\left(  M\right)  $ along a smooth one form $\alpha
\in\Omega^{1}\left(  M,W\right)  $ is defined by%
\begin{equation}
\mathbf{Y}=\int\tilde{\alpha}\left(  d\mathbf{X}\right)  \text{ as in Theorem
\ref{the.2.14}} \label{equ.3.12}%
\end{equation}
where $\tilde{\alpha}\in\Omega^{1}\left(  U,W\right)  $ is any extension of
$\alpha$ to a one form on some open neighborhood $M$ in $E.$ [Later, in
Proposition \ref{pro.3.29}, we will show how to characterize $\int%
\alpha\left(  d\mathbf{X}\right)  $ without using any extension of $\alpha.]$
\end{definition}

As a corollary we immediately see that the rough integrals against smooth one
forms are sufficient to characterise a rough path.

\begin{corollary}
\label{cor.3.25}Suppose that $\mathbf{X}$ and $\mathbf{Y}$ are two elements of
$WG_{p}\left(  M\right)  $ such that $x=y,$ and which satisfy%
\[
\int\alpha\left(  d\mathbf{X}\right)  =\int\alpha\left(  d\mathbf{Y}\right)
\]
for all $\alpha\in\Omega^{1}\left(  M,V\right)  .$ Then $\mathbf{X}%
=\mathbf{Y}$ as elements of $WG_{p}\left(  E\right)  .$
\end{corollary}

\begin{proof}
Let $\tilde{\alpha}\in\Omega^{1}\left(  E,V\right)  $ so that $\tilde{\alpha}$
is a smooth extension of $\alpha=\tilde{\alpha}|_{TM}$. By Lemma
\ref{lem.3.23} we have
\[
\int_{s}^{t}\tilde{\alpha}\left(  d\mathbf{X}\right)  =\int_{s}^{t}%
\alpha\left(  d\mathbf{X}\right)  =\int_{s}^{t}\alpha\left(  d\mathbf{Y}%
\right)  =\int_{s}^{t}\tilde{\alpha}\left(  d\mathbf{Y}\right)  .
\]
Let $\mathbb{A}_{s,t}:=$ $\mathbb{Y}_{s,t}-\mathbb{X}_{s,t.}$ It follows that
\[
\tilde{\alpha}_{x_{s}}\left(  x_{s,t}\right)  +\tilde{\alpha}_{x_{s}}^{\prime
}\left(  \mathbb{X}_{s,t}\right)  \simeq\tilde{\alpha}_{x_{s}}\left(
y_{s,t}\right)  +\tilde{\alpha}_{x_{s}}^{\prime}\left(  \mathbb{Y}%
_{s,t}\right)  \simeq\tilde{\alpha}_{x_{s}}\left(  x_{s,t}\right)
+\tilde{\alpha}_{x_{s}}^{\prime}\left(  \mathbb{Y}_{s,t}\right)
\]
which implies $\tilde{\alpha}_{x_{s}}^{\prime}\left(  \mathbb{A}_{s,t}\right)
\simeq0$ for every $\tilde{\alpha}\in\Omega^{1}\left(  E,V\right)  $ and every
$s$ and $t$ in $\left[  0,T\right]  .$ If we choose $\tilde{\alpha}\in
\Omega^{1}\left(  E,E\otimes E\right)  $ to be defined by $\tilde{\alpha
}\left(  \xi_{x}\right)  =\tilde{\alpha}_{x}\xi=x\otimes\xi,$ then
$\tilde{\alpha}_{x_{s}}^{\prime}\left[  \eta\otimes\xi\right]  =\eta\otimes
\xi$ for all $\eta,\xi\in E.$ So for this $\tilde{\alpha}$ it follows that
$\mathbb{A}_{s,t}=\tilde{\alpha}_{x_{s}}^{\prime}\left(  \mathbb{A}%
_{s,t}\right)  \simeq0$ and the result follows from Lemma \ref{lem.2.12}.
\end{proof}

Analogous to the Banach space setting every geometric $p$ -rough path on a
manifold is a weakly geometric $p$ -rough path.

\begin{proposition}
\label{pro.3.26}For $p\in\lbrack2,3)$ we have

\begin{enumerate}
\item $G_{p}\left(  M\right)  \subseteq WG_{p}\left(  M\right)  $.

\item Suppose $\mathbf{X\in}WG_{p}\left(  E\right)  $ and $\mathbf{X\in
}G_{p^{\prime}}\left(  M\right)  $ for some $3>p^{\prime}>p,$ then
$\mathbf{X\in}WG_{p}\left(  M\right)  .$
\end{enumerate}
\end{proposition}

\begin{proof}
Let $\mathbf{X\in}G_{p}\left(  M\right)  .$ For the first claim, we note by
definition there exists a sequence $x_{n}$ of smooth paths in $M$ such that
the lifts $\mathbf{X}_{n}:=S_{2}\left(  x_{n}\right)  \in G_{p}\left(
M\right)  \subset G_{p}\left(  E\right)  $ converge to $\mathbf{X}$ in the
rough $p$ -variation metric on $G_{p}\left(  E\right)  $. Let $\tilde{\alpha
}\in\Omega^{1}\left(  E,W\right)  $ be such that $\tilde{\alpha}|_{TM}%
\equiv0,$then
\begin{equation}
\int\tilde{\alpha}\left(  dx_{n}\right)  \mathbf{=}0\text{, hence }%
0=S_{2}\left(  \int\tilde{\alpha}\left(  dx_{n}\right)  \right)  =\int%
\tilde{\alpha}\left(  d\mathbf{X}_{n}\right)  \mathbf{\rightarrow}\int%
\tilde{\alpha}\left(  d\mathbf{X}\right)  \label{equ.3.13}%
\end{equation}
as $n\rightarrow\infty.$ By definition the trace $x$ lies in $M,$ and it is
immediate that we have $\mathbf{X\in}WG_{p}\left(  M\right)  .$ For the second
claim we approximate $\mathbf{X}$ in $p^{\prime}$ variation to deduce that
$\left(  \ref{equ.3.13}\right)  $ holds provided $\tilde{\alpha}|_{TM}%
\equiv0.$
\end{proof}

In the following we will frequently rely on localisation arguments.

\begin{remark}
[localisation]\label{rem.3.27}Suppose $\mathbf{X}=\left(  x,\mathbb{X}\right)
\in WG_{p}\left(  E\right)  $ has its trace, $x,$ lying in $M.$ By a simple
compactness argument, there exists $k\in\mathbb{N},$ open subsets $U_{i}$ of
$E,$ and local defining functions, $F_{i}:U_{i}\rightarrow\mathbb{R}^{N-d},$
as in Definition \ref{def.3.1} for $1\leq i\leq k$ such that $\left\{
U_{i}\right\}  _{i=1}^{k}$ is an open cover of $x\left(  \left[  0,T\right]
\right)  .$ Furthermore, since $x$ is uniformly continuous, we can find
$\delta=\delta\left(  \mathbf{X}\right)  >0,$ such that for all $s$ and $t$ in
the interval $\left[  0,T\right]  $ with $\left\vert s-t\right\vert <\delta$
the path segment%
\begin{equation}
\left\{  x_{u}:u\in\left[  s,t\right]  \right\}  \subset U_{i}\text{ }
\label{equ.3.14}%
\end{equation}
for some $i\in\left\{  1,....,k\right\}  .$
\end{remark}

The next result describes the constraints on $x_{s,t}$ which arise when
$\mathbf{X}\in WG_{p}\left(  M\right)  $ -- also see Example \ref{exa.3.30} below.

\begin{lemma}
\label{lem.3.28}If $\mathbf{X}\in WG_{p}\left(  M\right)  $ then
\begin{equation}
Q_{x_{s}}x_{s,t}\simeq Q_{x_{s}}\left(  \partial_{P_{x_{s}}a}Pb\right)
_{a\otimes b=\mathbb{X}_{s,t}}. \label{equ.3.15}%
\end{equation}

\end{lemma}

\begin{proof}
Let $\tilde{\alpha}\left(  \xi_{x}\right)  =Q_{F}\left(  x\right)  \xi$ so
that $\tilde{\alpha}\in\Omega^{1}\left(  U,E\right)  .$ Then $\tilde{\alpha
}|_{T\left(  M\cap U\right)  }\equiv0$ and therefore by Definition
\ref{def.3.17} and Corollary \ref{cor.3.22},%
\begin{equation}
0=\left[  \int\tilde{\alpha}\left(  d\mathbf{X}\right)  \right]  _{s,t}%
^{1}\simeq\tilde{\alpha}_{x_{s}}\left(  x_{st}\right)  +\tilde{\alpha}_{x_{s}%
}^{\prime}\mathbb{X}_{st}\simeq Q_{x_{s}}x_{st}+\tilde{\alpha}_{x_{s}}%
^{\prime}\left[  P_{x_{s}}\otimes P_{x_{s}}\mathbb{X}_{st}\right]  .
\label{equ.3.16}%
\end{equation}
Solving Eq. (\ref{equ.3.16}) for $Q_{x_{s}}x_{st}$ completes the proof after
using the identity,
\[
\tilde{\alpha}_{x_{s}}^{\prime}\left[  P_{x_{s}}a\otimes P_{x_{s}}b\right]
=dQ\left(  P_{x_{s}}a\right)  P_{x_{s}}b=-Q_{x_{s}}dP\left(  P_{x_{s}%
}a\right)  P_{x_{s}}b\text{ }\forall~a,b\in E,
\]
wherein the last inequality made use of Lemma \ref{lem.3.11} and the fact that
$P^{2}=P.$ It is easily seen that this agrees with (\ref{equ.3.15}).
\end{proof}

We conclude this section with a theorem that provides a more explicit
description of the integral of one forms along $\mathbf{X}\in WG_{p}\left(
\left[  0,T\right]  ,M\right)  $ which require no extensions of the one form
to the ambient space.

\begin{proposition}
[Integrating one forms without extensions]\label{pro.3.29}If $\mathbf{X}\in
WG_{p}\left(  \left[  0,T\right]  ,M\right)  $ and $\alpha\in\Omega^{1}\left(
M,W\right)  ,$ then
\begin{equation}
\left[  \int\alpha\left(  d\mathbf{X}\right)  \right]  _{s,t}^{1}\simeq
\alpha_{x_{s}}\left(  P_{x_{s}}x_{st}\right)  +\left(  \nabla\alpha\right)
\left(  \left[  P_{x_{s}}\otimes P_{x_{s}}\right]  \mathbb{X}_{s,t}\right)
\label{equ.3.17}%
\end{equation}
and
\begin{equation}
\left[  \int\alpha\left(  d\mathbf{X}\right)  \right]  _{s,t}^{2}\simeq
\alpha_{x_{s}}\otimes\alpha_{x_{s}}\left[  P_{x_{s}}\otimes P_{x_{s}}\right]
\mathbb{X}_{s,t}, \label{equ.3.18}%
\end{equation}
where $\nabla\alpha$ is the Levi-Civita covariant derivative of $\alpha$ as in
Definition \ref{def.3.10}.
\end{proposition}

\begin{proof}
By Definition \ref{def.3.24},%
\begin{equation}
\int_{s}^{t}\alpha\left(  d\mathbf{X}\right)  =\int_{s}^{t}\tilde{\alpha
}\left(  d\mathbf{X}\right)  \simeq\left(  \tilde{\alpha}_{x_{s}}\left(
x_{st}\right)  +\tilde{\alpha}_{x_{s}}^{\prime}\mathbb{X}_{st},\left[
\tilde{\alpha}_{x_{s}}\otimes\tilde{\alpha}_{x_{s}}\right]  \mathbb{X}%
_{s,t}\right)  \label{equ.3.19}%
\end{equation}
where $\tilde{\alpha}$ is any extension of $\alpha$ to an open neighborhood of
$M$ in $E=\mathbb{R}^{N}.$ By Corollary \ref{cor.3.22}, $\mathbb{X}%
_{s,t}\simeq\left[  P_{x_{s}}\otimes P_{x_{s}}\right]  \mathbb{X}_{s,t}$ and
hence we may replace Eq. (\ref{equ.3.19}) by%
\begin{align}
\int_{s}^{t}\alpha\left(  d\mathbf{X}\right)   &  \simeq\left(  \tilde{\alpha
}_{x_{s}}\left(  x_{s,t}\right)  +\tilde{\alpha}_{x_{s}}^{\prime}\left[
P_{x_{s}}\otimes P_{x_{s}}\right]  \mathbb{X}_{s,t},\left[  \tilde{\alpha
}_{x_{s}}\otimes\tilde{\alpha}_{x_{s}}\right]  \left[  P_{x_{s}}\otimes
P_{x_{s}}\right]  \mathbb{X}_{s,t}\right) \label{equ.3.20}\\
&  \simeq\left(  \tilde{\alpha}_{x_{s}}\left(  x_{s,t}\right)  +\tilde{\alpha
}_{x_{s}}^{\prime}\left[  P_{x_{s}}\otimes P_{x_{s}}\right]  \mathbb{X}%
_{s,t},\alpha_{x_{s}}\otimes\alpha_{x_{s}}\left[  P_{x_{s}}\otimes P_{x_{s}%
}\right]  \mathbb{X}_{s,t}\right)  . \label{equ.3.21}%
\end{align}

Let us now use Remark \ref{rem.3.2} to locally extend $P$ to a neighborhood of
$M$ so that $P=P\circ\pi.$ By replacing $\tilde{\alpha}$ by $\tilde{\alpha}P$
if necessary we may assume $\tilde{\alpha}=\tilde{\alpha}P.$ Under this
assumption, Eq. (\ref{equ.3.21}) becomes,%
\begin{equation}
\int_{s}^{t}\alpha\left(  d\mathbf{X}\right)  \simeq\left(  \alpha_{x_{s}%
}\left(  P_{x_{s}}x_{s,t}\right)  +\tilde{\alpha}_{x_{s}}^{\prime}\left[
P_{x_{s}}\otimes P_{x_{s}}\right]  \mathbb{X}_{s,t},\alpha_{x_{s}}%
\otimes\alpha_{x_{s}}\left[  P_{x_{s}}\otimes P_{x_{s}}\right]  \mathbb{X}%
_{s,t}\right)  . \label{equ.3.22}%
\end{equation}
From item 4. of Proposition \ref{pro.3.12},
\[
\tilde{\alpha}_{m}^{\prime}\left[  v\otimes w\right]  =\left(  \nabla
_{v}\alpha\right)  \left(  w\right)  \text{ for all }v,w\in T_{m}M
\]
which combined with Eq. (\ref{equ.3.22}) proves Eqs. (\ref{equ.3.17}) and
(\ref{equ.3.18}).
\end{proof}

\subsection{Characterising Weakly Geometric Rough Paths on $M$\label{sub.3.3}}

The goal of this subsection is to show $\mathbf{X=}\left(  x,\mathbb{X}%
\right)  \in WG_{p}\left(  E\right)  $ is in $WG_{p}\left(  M\right)  $ iff
$x_{t}\in M$ for all $0\leq t\leq T$ and that either of the equivalent Eqs.
(\ref{equ.3.9}) or (\ref{equ.3.10}) holds locally. This will be carried out in
Proposition \ref{pro.3.35} below. The next example shows the result in Lemma
\ref{lem.3.28} is really about paths in $x_{t}\in M$ and not so much about its
augmentation to a rough path.

\begin{example}
\label{exa.3.30}Let $x_{t}$ be any path $M$ with $\left\vert x_{s,t}%
\right\vert \leq C\omega\left(  s,t\right)  ^{1/p}.$ Then%
\begin{align*}
0  &  =\left[  F\left(  x\right)  \right]  _{st}=F^{\prime}\left(
x_{s}\right)  x_{s,t}+\frac{1}{2}F^{\prime\prime}\left(  x_{s}\right)
x_{s,t}\otimes x_{s,t}+O\left(  \left\vert x_{s,t}\right\vert ^{3}\right) \\
&  \simeq F^{\prime}\left(  x_{s}\right)  x_{s,t}+\frac{1}{2}F^{\prime\prime
}\left(  x_{s}\right)  \left[  x_{s,t}\otimes x_{s,t}\right]  .
\end{align*}
Applying $A_{F}\left(  x_{s}\right)  \in\operatorname{Hom}\left(
\mathbb{R}^{N-d},E\right)  $ (see Remark \ref{rem.3.5}) to this equation then
shows%
\begin{equation}
Q\left(  x_{s}\right)  x_{s,t}\simeq-\frac{1}{2}A_{F}\left(  x_{s}\right)
F^{\prime\prime}\left(  x_{s}\right)  \left[  x_{s,t}\otimes x_{s,t}\right]  .
\label{equ.3.23}%
\end{equation}
From this equation it follows that $x_{s,t}=P\left(  x_{s}\right)
x_{s,t}+O\left(  \left\vert x_{s,t}\right\vert ^{2}\right)  $ and so we may
replace $x_{st}\otimes x_{st}$ in Eq. (\ref{equ.3.23}) by $P\left(
x_{s}\right)  x_{s,t}\otimes P\left(  x_{s}\right)  x_{s,t}$ which allows us
to rewrite Eq. (\ref{equ.3.23}) as%
\begin{equation}
Q\left(  x_{s}\right)  x_{s,t}\simeq-\frac{1}{2}A_{F}\left(  x_{s}\right)
F^{\prime\prime}\left(  x_{s}\right)  \left[  P\left(  x_{s}\right)
x_{s,t}\otimes P\left(  x_{s}\right)  x_{s,t}\right]  . \label{equ.3.24}%
\end{equation}
So if $x_{t}\in M$ for all $t,$ the component of $x_{s,t}$ orthogonal to
$\tau_{x_{s}}M$ is determined modulo terms of order $\left\vert x_{s,t}%
\right\vert ^{3}$ by knowing the component of $x_{s,t}$ tangential to $M$ at
$x_{s}.$
\end{example}

\begin{lemma}
\label{lem.3.31}Suppose $\mathbf{X}=\left(  x,\mathbb{X}\right)  \in
WG_{p}\left(  E\right)  $ such that $x_{t}\in M$ for all $t\in\left[
0,T\right]  ,$ then
\begin{equation}
\left(  I\otimes Q\left(  x_{s}\right)  \right)  \left[  \mathbb{X}%
_{s,t}\right]  ^{s}\simeq0. \label{equ.3.25}%
\end{equation}

\end{lemma}

\begin{proof}
Note that by the definition of $\simeq$ it is sufficient to check $\left(
\ref{equ.3.25}\right)  $ locally for all $0<s<t<T$ such that $\left\vert
t-s\right\vert <\delta$ and some $\delta>0.$ Let $\left\{  U_{i}%
:i=1,....,k\right\}  $ and $F_{i}$ as in Remark \ref{rem.3.27} be a cover of
the trace $x.$ By construction of the cover for all $0\leq s<t\leq T$ with
$\left\vert s-t\right\vert <\delta$ there exists $U_{i}$ such that $\left(
\ref{equ.3.14}\right)  $ holds. By (\ref{equ.3.1}) we may assume that for
$m\in U_{i}$ we have $Q\left(  m\right)  =A\left(  m\right)  F_{i}^{\prime
}\left(  m\right)  ,$ where $A\left(  m\right)  :=F_{i}^{\prime}(m)^{\ast
}(F_{i}^{\prime}(m)F_{i}^{\prime}(m)^{\ast})^{-1}.$ From Eq. (\ref{equ.2.6})
which holds by definition of $\mathbf{X}$ being in $WG_{p}\left(  E\right)  ,$
it follows that
\[
\left(  I\otimes Q\left(  x_{s}\right)  \right)  \left[  \mathbb{X}%
_{s,t}\right]  ^{s}=\frac{1}{2}\left(  I\otimes Q\left(  x_{s}\right)
\right)  x_{s,t}\otimes x_{s,t}=\frac{1}{2}x_{s,t}\otimes A\left(
x_{s}\right)  F_{i}^{\prime}\left(  x_{s}\right)  x_{s,t}.
\]
Applying item 1. of Lemma \ref{lem.3.18} to the right member of this equations
gives the estimate;%
\[
\left\vert \left(  I\otimes Q\left(  x_{s}\right)  \right)  \left[
\mathbb{X}_{s,t}\right]  ^{s}\right\vert \leq\frac{1}{2}\left\vert A\left(
x_{s}\right)  \right\vert \left\vert x_{s,t}\right\vert \left\vert
F_{i}^{\prime}\left(  x_{s}\right)  x_{s,t}\right\vert \leq C\left\vert
x_{s,t}\right\vert ^{3}\simeq0.
\]

\end{proof}

\begin{corollary}
\label{cor.3.32}If $\mathbf{X}$ is an element of $WG_{p}\left(  E\right)  $
such that the trace $x$ is in $M,$ then the following are equivalent:

\begin{enumerate}
\item $\left(  I_{E}\otimes Q\left(  x_{s}\right)  \right)  \left[
\mathbb{X}_{s,t}\right]  ^{a}\simeq0,$

\item $\left(  I_{E}\otimes Q\left(  x_{s}\right)  \right)  \left[
\mathbb{X}_{s,t}\right]  \simeq0,$ and

\item $\left(  I_{E}\otimes F^{\prime}\left(  x_{s}\right)  \right)  \left[
\mathbb{X}_{s,t}\right]  \simeq0$ over the interval $\left[  u,v\right]  ,$
whenever $F$ is a local defining function for $M$ on $U$ in the sense of
Definition \ref{def.3.1}, and the path segment of $x$ over $\left[
u,v\right]  $ satisfies%
\[
\left\{  x_{r}:r\in\left[  u,v\right]  \right\}  \subset U.
\]

\end{enumerate}
\end{corollary}

\begin{proof}
The equivalence of items 1. and 2. is an immediate corollary of Lemma
\ref{lem.3.31}. The equivalence of items 2. and 3. is the content of Remark
\ref{rem.3.21}.
\end{proof}

\begin{remark}
\label{rem.3.33}If $\mathbf{X\in}WG_{p}\left(  E\right)  $, then the condition
$\left(  I_{E}\otimes Q\left(  x_{s}\right)  \right)  \left[  \mathbb{X}%
_{s,t}\right]  \simeq0$ is equivalent to the condition that $\left(  Q\left(
x_{s}\right)  \otimes I_{E}\right)  \left[  \mathbb{X}_{s,t}\right]  \simeq0.$
To see this is the case we let $\mathcal{F}:E\otimes E\rightarrow E\otimes E$
denote the linear flip operator determined by $\mathcal{F}\left[  a\otimes
b\right]  =b\otimes a$ for all $a,b\in E.$ Then
\begin{align*}
\mathcal{F}\left(  Q\left(  x_{s}\right)  \otimes I_{E}\right)  \left[
\mathbb{X}_{s,t}\right]   &  =\left(  I_{E}\otimes Q\left(  x_{s}\right)
\right)  \left[  \mathcal{F}\mathbb{X}_{s,t}\right] \\
&  =\left(  I_{E}\otimes Q\left(  x_{s}\right)  \right)  \left[
\mathcal{F}\mathbb{X}_{s,t}^{s}+\mathcal{F}\mathbb{X}_{s,t}^{a}\right] \\
&  =\left(  I_{E}\otimes Q\left(  x_{s}\right)  \right)  \left[
\mathbb{X}_{s,t}^{s}-\mathbb{X}_{s,t}^{a}\right] \\
&  \simeq\left(  I_{E}\otimes Q\left(  x_{s}\right)  \right)  \left[
-\mathbb{X}_{s,t}^{s}-\mathbb{X}_{s,t}^{a}\right] \\
&  =-\left(  I_{E}\otimes Q\left(  x_{s}\right)  \right)  \left[
\mathbb{X}_{s,t}\right]
\end{align*}
where in the second to last line we have used $I_{E}\otimes Q\left(
x_{s}\right)  \mathbb{X}_{s,t}^{s}=\frac{1}{2}x_{s,t}\otimes Q\left(
x_{s}\right)  x_{s,t}\simeq0.$
\end{remark}

The next proposition shows Definition \ref{def.3.17} above and Definition
\ref{def.3.34} below for the notion of a weakly geometric rough path are equivalent.

\begin{definition}
[Projection Definition of Weakly Geometric Rough Paths]\label{def.3.34}We say
that $\mathbf{X=}\left(  x,\mathbb{X}\right)  $ is a \textbf{weakly geometric
}$p-$\textbf{ rough path }on the manifold $M$ if: $\mathbf{X}$ is
in\ $WG_{p}\left(  E\right)  $, its trace $x$ lies in $M,$ and $\mathbb{X}$
satisfies
\begin{equation}
\left(  I_{E}\otimes Q\left(  x_{s}\right)  \right)  \left[  \mathbb{X}%
_{s,t}\right]  ^{a}\simeq0\simeq\left(  I_{E}\otimes Q\left(  x_{s}\right)
\right)  \mathbb{X}_{s,t}, \label{equ.3.26}%
\end{equation}
wherein $Q$ is the orthogonal projection onto the normal bundle as in Notation
\ref{not.3.3} and $I_{E}$ is the identity map on $E.$
\end{definition}

\begin{proposition}
[The Projection Characterization of $WG_{p}\left(  M\right)  $]%
\label{pro.3.35}Let $\mathbf{X=}\left(  x,\mathbb{X}\right)  \in WG_{p}\left(
E\right)  $ then $\mathbf{X\in}WG_{p}\left(  M\right)  $ (Definition
\ref{def.3.17}) if and only the trace $x$ is in $M$ and any one of the
equivalent conditions in Corollary \ref{cor.3.32} hold.
\end{proposition}

\begin{proof}
$\left(  \implies\right)  $ This implication has already been demonstrated in
Corollary \ref{cor.3.20} and Remarks \ref{rem.3.21} and \ref{rem.3.27}.

$\left(  \Longleftarrow\right)  $ For the converse implication assume
$x_{t}\in M$ for all $t\in\left[  0,T\right]  $ and (again, locally)
\begin{equation}
\left[  I_{E}\otimes F^{\prime}\left(  x_{s}\right)  \right]  \mathbb{X}%
_{s,t}\simeq0\text{ and }\left[  I_{E}\otimes Q\left(  x_{s}\right)  \right]
\left[  \mathbb{X}_{s,t}\right]  \simeq0. \label{equ.3.27}%
\end{equation}
We have to show for any finite dimensional vector space $W$ that
\begin{equation}
\int\alpha\left(  d\mathbf{X}\right)  \equiv0\text{ }\forall~\alpha\in
\Omega^{1}\left(  E,W\right)  ~\ni~\alpha|_{TM}\equiv0. \label{equ.3.28}%
\end{equation}
The proof will proceed in several stages, considering first one-forms with
specific structures, and finally combining those results to deduce the general
claim. In what follows we let
\[
\mathbf{Y}=\left(  y,\mathbb{Y}\right)  :=\int\alpha\left(  d\mathbf{X}%
\right)
\]
and let $Q,$ $Q_{F},$ and $A_{F}$ be as in Notation \ref{not.3.3} and Remarks
\ref{rem.3.4} and \ref{rem.3.5}.

\textbf{Case 1. }We begin by supposing that $\alpha=\varphi dF=\varphi
F^{\prime}\in\Omega^{1}\left(  E,W\right)  $ for some $\varphi\in
C_{c}^{\infty}\left(  E,\operatorname{Hom}\left(  V,W\right)  \right)  $ with
$\mathrm{supp}\left(  \varphi\right)  \subset U$ $.$ By Eq. (\ref{equ.3.5}) of
Lemma \ref{lem.3.19} and Eq. (\ref{equ.2.7}) we learn that%
\begin{equation}
y_{s,t}\simeq\alpha\left(  x_{s}\right)  x_{s,t}+\alpha^{\prime}\left(
x_{s}\right)  \mathbb{X}_{s,t}\simeq\left[  \varphi^{\prime}\left(
x_{s}\right)  \cdot F^{\prime}\left(  x_{s}\right)  \right]  \mathbb{X}%
_{s,t}\simeq0 \label{equ.3.29}%
\end{equation}
where for the last approximation we have used the assumption on Eq.
(\ref{equ.3.27}). Similarly,%
\begin{equation}
\mathbb{Y}_{s,t}\simeq\alpha\left(  x_{s}\right)  \otimes\alpha\left(
x_{s}\right)  \mathbb{X}_{s,t}=\left[  \varphi\left(  x_{s}\right)
\otimes\varphi\left(  x_{s}\right)  \right]  \left[  F^{\prime}\left(
x_{s}\right)  \otimes I_{E}\right]  \left[  I_{E}\otimes F^{\prime}\left(
x_{s}\right)  \right]  \mathbb{X}_{s,t}\simeq0. \label{equ.3.30}%
\end{equation}
Equations (\ref{equ.3.29}) and (\ref{equ.3.30}) along with Lemma
\ref{lem.2.12} shows $y_{s,t}=0$ and $\mathbb{Y}_{s,t}\equiv0$ for all $s$ and
$t.$

\textbf{Case 2. }Now suppose $\alpha=\beta\circ Q_{F}$ where $\beta\in
\Omega^{1}\left(  E,W\right)  $ is any one form on $E.$ Then locally we have
\[
\alpha\left(  x\right)  \xi=\beta\left(  x\right)  Q_{F}\left(  x\right)
\xi=\beta\left(  x\right)  A_{F}\left(  x\right)  F^{\prime}\left(  x\right)
\xi=\varphi\left(  x\right)  F^{\prime}\left(  x\right)  \xi
\]
where $\varphi\left(  x\right)  :=\beta\left(  x\right)  A_{F}\left(
x\right)  .$ We conclude by using case 1 and a suitable application of Remark
\ref{rem.3.27}.

\textbf{Case 3. }Now assume that $\beta\in\Omega^{1}\left(  E,W\right)  $ is a
one-form such that $\beta\left(  m\right)  \equiv0$ for all $m\in M.$ If
$\sigma\left(  t\right)  $ is a path in $M,$ then $\beta\left(  \sigma\left(
t\right)  \right)  =0$ and therefore $0=\frac{d}{dt}\beta\left(  \sigma\left(
t\right)  \right)  =\beta^{\prime}\left(  \sigma\left(  t\right)  \right)
\dot{\sigma}\left(  t\right)  .$ Since $\sigma\left(  t\right)  \in M$ is
arbitrary, it follows that $\left(  \partial_{v_{m}}\beta\right)  =0$ for all
$v_{m}\in T_{m}M.$ Hence we conclude that
\[
\left(  \partial_{\xi}\beta\right)  _{m}=\left(  \partial_{P_{m}\xi}%
\beta+\partial_{Q_{m}\xi}\beta\right)  _{m}=\left(  \partial_{Q_{m}\xi}%
\beta\right)  _{m}\text{ }\forall~m\in M\text{ and }\xi\in E=\mathbb{R}^{N}%
\]
or in other words,%
\[
\beta^{\prime}\left(  m\right)  =\beta^{\prime}\left(  m\right)  \left[
Q\left(  m\right)  \otimes I_{E}\right]  \text{ for all }m\in M.
\]
With this in hand, using Lemma \ref{lem.3.18} and Eq. (\ref{equ.3.27}) again,
we find that
\begin{equation}
\left[  \int\beta\left(  d\mathbf{X}\right)  \right]  _{s,t}^{1}\cong%
\beta\left(  x_{s}\right)  x_{s,t}+\beta^{\prime}\left(  x_{s}\right)
\mathbb{X}_{s,t}=\beta^{\prime}\left(  x_{s}\right)  \left[  Q\left(
x_{s}\right)  \otimes I_{E}\right]  \mathbb{X}_{s,t}\simeq0. \label{equ.3.31}%
\end{equation}
As usual this together with the additivity of the trace shows $\left[
\int\beta\left(  d\mathbf{X}\right)  \right]  _{s,t}^{1}=0.$ Then, working as
above, the second-order process is given by%
\[
\left[  \int\beta\left(  d\mathbf{X}\right)  \right]  _{s,t}^{2}\simeq\left[
\beta\left(  x_{s}\right)  \otimes\beta\left(  x_{s}\right)  \right]
\mathbb{X}_{s,t}=\left[  0\otimes0\right]  \mathbb{X}_{s,t}=0.
\]

\textbf{Case 4. }Finally, if $\alpha$ is any one form on $E$ with the property
that $\alpha|_{TM}\equiv0,$ then $\alpha=\alpha\circ Q_{F}$ on\footnote{A
slightly subtle point here is that $\alpha=\alpha\circ Q$ on $M$ but not
necessarily on a neighborhood of $M.$ For this reason we can not directly use
case 2. here.} $M.$ We now let $\beta:=\alpha-\alpha\circ Q_{F}$ so that
$\beta\left(  m\right)  \equiv0$ for all $m\in M.$ Thus we have decomposed
$\alpha$ as $\alpha=\beta+\alpha\circ Q_{F}$ where $\beta\equiv0$ on $M$ and
therefore by cases 2. and 3.,%
\[
y:=\left[  \int\alpha\left(  d\mathbf{X}\right)  \right]  ^{1}=\left[
\int\beta\left(  d\mathbf{X}\right)  \right]  ^{1}+\left[  \int\left(
\alpha\circ Q_{F}\right)  \left(  d\mathbf{X}\right)  \right]  ^{1}=0+0=0.
\]
We further have, using
\[
\alpha_{m}=\beta_{m}+\alpha_{m}Q_{m}=\alpha_{m}Q_{m}\text{ }\forall~m\in M,
\]
and Eq. (\ref{equ.3.27}) that
\begin{align*}
\mathbb{Y}_{s,t}  &  :=\left[  \int\alpha\left(  d\mathbf{X}\right)  \right]
_{s,t}^{2}\simeq\left[  \alpha_{x_{s}}Q_{x_{s}}\otimes\alpha_{x_{s}}Q_{x_{s}%
}\right]  \mathbb{X}_{s,t}\\
&  =\left[  \alpha_{x_{s}}\otimes Q_{x_{s}}\right]  \left[  I_{E}\otimes
Q_{x_{s}}\right]  \left[  Q_{x_{s}}\otimes I_{E}\right]  \mathbb{X}%
_{s,t}\simeq0.
\end{align*}
An application of Lemma \ref{lem.2.12} then shows $\mathbb{Y}_{s,t}\equiv0.$
\end{proof}

The defining property in Eq. $\left(  \ref{equ.3.26}\right)  $ is local and we
therefore need a remark analogous to Lemma \ref{lem.A.1}, which allows us to
concatenate rough paths on manifolds.

\begin{remark}
[gluing]\label{rem.3.36}Suppose that $D=\left\{  0=t_{0}<t_{1}<\dots
<t_{n}=T\right\}  $ is any partition of $\left[  0,T\right]  $. Let $\delta
>0$, and suppose that the overlapping intervals $J_{k}$ for $1\leq k\leq n$
are defined by
\[
J_{k}=\left[  t_{k-1},\min(t_{k}+\delta,T)\right]
\]
Assume, for each $k,$ we are given $\mathbf{X}\left(  k\right)  \in
WG_{p}\left(  J_{k},M\right)  $ such that $\mathbf{X}\left(  k\right)
_{s,t}=\mathbf{X}\left(  j\right)  _{s,t}$ for $s,t$ $\in$ $J_{k}\cap J_{j}$
and any $i$ and $j.$ Then, fixing a starting point $x_{0}\in M,$ there exists
a unique $\mathbf{X}\in WG_{p}\left(  \left[  0,T\right]  ,M\right)  $ with
$x\left(  0\right)  =x_{0}$ which is consistent with the $\mathbf{X}\left(
k\right)  $s in the sense that for all $1\leq k\leq n,$%
\[
\mathbf{X}\left(  k\right)  _{s,t}=\mathbf{X}_{s,t}\text{ for all }s,t\in
J_{k}.
\]

\end{remark}

\subsection{Push forwards and independence of the choice of
embedding\label{sub.3.4}}

Analogous to the Banach space setting (see \ref{sub.A.2}) we may consider the
pushfoward of rough paths on manifolds under sufficiently smooth maps.

\begin{definition}
[Pushed-forward rough paths ]\label{def.3.37}Let $M$ and $N$ respectively be
smooth embedded submanifolds of the Euclidean spaces $E$ and $E^{\prime}.$
Suppose that $\varphi:M\rightarrow N$ is smooth and let $d\varphi\in\Omega
^{1}\left(  M,E^{\prime}\right)  $, i.e. we regard $d\varphi$ as an
$E^{\prime}$ -- valued one form. Then if $\mathbf{X}$ is an element of
$WG_{p}\left(  M\right)  ,$ we define the pushed-forward rough path
$\varphi_{\ast}\left(  \mathbf{X}\right)  $ in $E^{\prime}$ by setting
\[
\varphi_{\ast}\left(  \mathbf{X}\right)  :=\int d\varphi\left(  d\mathbf{X}%
\right)  =\int\varphi^{\prime}\left(  x\right)  d\mathbf{X},
\]
and taking the starting point to be $\varphi\left(  x_{0}\right)  \mathbf{.}$
\end{definition}

\begin{proposition}
[Pushing forward rough paths]\label{pro.3.38}Let $\mathbf{X}\in WG_{p}\left(
M\right)  .$ The rough path $\varphi_{\ast}\mathbf{X}$ in Definition
\ref{def.3.37} satisfies;

\begin{enumerate}
\item $\left[  \varphi_{\ast}\mathbf{X}\right]  _{s,t}^{1}=\varphi\left(
x_{t}\right)  -\varphi\left(  x_{s}\right)  \in E^{\prime}$ for all
$s,t\in\left[  0,T\right]  .$

\item $\varphi_{\ast}\mathbf{X}$ is an element of $WG_{p}\left(  N\right)  .$

\item If $L$ is another smooth submanifold which is embedded in the Euclidean
space $E^{\prime\prime}$ and if $\psi:N\rightarrow L$ is a smooth map, then
\[
\psi_{\ast}\left[  \varphi_{\ast}\left(  \mathbf{X}\right)  \right]  =\left[
\psi\circ\varphi\right]  _{\ast}\left(  \mathbf{X}\right)  .
\]

\item If $\beta\in\Omega^{1}\left(  N,V\right)  ,$ then
\[
\int\beta\left(  d\left[  \varphi_{\ast}\left(  \mathbf{X}\right)  \right]
\right)  =\int\left(  \varphi^{\ast}\beta\right)  \left(  d\mathbf{X}\right)
.
\]

\end{enumerate}
\end{proposition}

\begin{proof}
We take each item in turn.

\begin{enumerate}
\item If $\varphi:M\rightarrow N$ is a smooth map between embedded
submanifolds it may be viewed (at least locally) as the restriction of a
smooth map from $\Phi:E\rightarrow E^{\prime}.$ It then follows that $d\Phi$
is an extension of $d\varphi$ to a neighborhood of $M$ and therefore by
Definition \ref{def.3.24}, $\varphi_{\ast}\left(  \mathbf{X}\right)
=\Phi_{\ast}\left(  \mathbf{X}\right)  ,$ and hence from Lemma \ref{lem.A.4}
we have that
\[
\left[  \varphi_{\ast}\left(  \mathbf{X}\right)  \right]  _{s,t}^{1}=\left[
\Phi_{\ast}\left(  \mathbf{X}\right)  \right]  _{s,t}^{1}=\Phi\left(
x_{t}\right)  -\Phi\left(  x_{s}\right)  =\varphi\left(  x_{t}\right)
-\varphi\left(  x_{s}\right)  .
\]

\item Since $\varphi_{\ast}\left(  \mathbf{X}\right)  =\Phi_{\ast}\left(
\mathbf{X}\right)  ,$ it follows that $\varphi_{\ast}\left(  \mathbf{X}%
\right)  \in WG_{p}\left(  E^{\prime}\right)  .$ Moreover if $\alpha\in
\Omega^{1}\left(  E^{\prime},W\right)  $ is such that $\alpha|_{TN}\equiv0$
then by Theorem \ref{the.A.5}
\[
\int\alpha\left(  d\left[  \varphi_{\ast}\left(  \mathbf{X}\right)  \right]
\right)  =\int\alpha\left(  d\left[  \Phi_{\ast}\left(  \mathbf{X}\right)
\right]  \right)  =\int\left[  \alpha\circ\Phi^{\prime}\right]  d\mathbf{X}=0
\]
as $\alpha\circ\Phi^{\prime}=\Phi^{\ast}\alpha\in\Omega^{1}\left(  E,W\right)
$ which vanishes on $TM.$ We deduce from Definition \ref{def.3.17} and 1. that
$\varphi_{\ast}\left(  \mathbf{X}\right)  \in WG_{p}\left(  N\right)  .$

\item Follows by a similar argument to 2. using Corollary \ref{cor.A.6}.

\item This is a consequence of Theorem \ref{the.A.5} (once again using that
$\varphi_{\ast}\left(  \mathbf{X}\right)  =\Phi_{\ast}\left(  \mathbf{X}%
\right)  ,$ and the fact that $\Phi^{\ast}\beta$ restricts to $\varphi^{\ast
}\beta$)$.$
\end{enumerate}
\end{proof}

\begin{example}
\label{exa.3.39}Suppose that $\varphi:M\rightarrow M$ is the identity map,
then $\varphi=\Phi|_{E}$ where $\Phi:E\rightarrow E$ is the identity map and
therefore,
\[
\varphi_{\ast}\left(  \mathbf{X}\right)  =\int d\varphi\left(  d\mathbf{X}%
\right)  =\int d\Phi\left(  d\mathbf{X}\right)  =\mathbf{X}.
\]

\end{example}

The preceding example is a special case of the more general fact that
diffeomorphisms give rise to bijections between the respective sets of weakly
geometric rough paths on two embedded manifolds. The following corollary is
immediate from Proposition \ref{pro.3.38}.

\begin{corollary}
\label{cor.3.40}Let $M,N$ be embedded manifolds and $\varphi:M\rightarrow N$ a
diffeomorphism. Then the function $\varphi_{\ast}$ is a bijection between
$WG_{p}\left(  M\right)  $ and $WG_{p}\left(  N\right)  .$
\end{corollary}

Suppose now $\mathcal{M}$ is an abstract manifold embedded as $M$ and
$\tilde{M}$ in two vector spaces $E$ and $\widetilde{E}.$ Then there exist
smooth maps $f:\mathcal{M}$ $\rightarrow E$ and $\tilde{f}$ :$\mathcal{M}$
$\rightarrow\tilde{E}$ diffeomorphic onto their image such that $f\left(
\mathcal{M}\right)  =M$ and $\tilde{f}\left(  \mathcal{M}\right)  =\tilde{M}.$
The following corollary shows that we have a natural identification between
the rough paths on $M$ and $\tilde{M}.$ The map we construct is natural in the
sense that it respects the integration of one forms (characterizing the rough
paths, cf. Corollary \ref{cor.3.25}).

\begin{corollary}
\label{cor.3.41}Let $\mathcal{M}$, $M$, $\tilde{M}$ as above. Then the
pushforward $\left(  \tilde{f}\circ f^{-1}\right)  _{\ast}$ \ is a bijective
map from $WG_{p}\left(  M\right)  $ to $WG_{p}\left(  \tilde{M}\right)  $ such
that for any finite dimensional vector space valued on form $\alpha\in
\Omega^{1}\left(  \mathcal{M},W\right)  $ and any $\mathbf{X\in}WG_{p}\left(
M\right)  $
\[
\int\left(  \left(  f^{-1}\right)  ^{\ast}\alpha\right)  \left(
d\mathbf{X}\right)  =\int\left(  \left(  \tilde{f}^{-1}\right)  ^{\ast}%
\alpha\right)  \left(  d\left(  \left[  \tilde{f}\circ f^{-1}\right]  _{\ast
}\mathbf{X}\right)  \right)  .
\]

\end{corollary}

\begin{definition}
[Abstract weakly geometric rough paths]\label{def.3.42}Let $\mathcal{M}$ be an
abstract manifold and suppose that $f:\mathcal{M}\rightarrow M\subset E$ and
$\tilde{f}:\mathcal{M}\rightarrow\tilde{M}\subset\tilde{E}$ are two embedding
of $\mathcal{M}$ into Euclidean spaces $E$ and $\tilde{E}$ respectively. We
say that $\left(  f,\mathbf{X}\right)  $ and $\left(  \tilde{f},\mathbf{\tilde
{X}}\right)  ,$ where $\mathbf{X\in}WG_{p}\left(  M\right)  $ and
$\mathbf{\tilde{X}}\in WG_{p}\left(  \tilde{M}\right)  ,$ are
equivalent\footnote{This is an equivalence relation because of item 3. of
Proposition \ref{pro.3.38}.} if
\[
\mathbf{\tilde{X}}=\left(  \tilde{f}\circ f^{-1}\right)  _{\ast}\left(
\mathbf{X}\right)  .
\]
The equivalence class associated to $\left(  f,\mathbf{X}\right)  $ will be
denoted by $\left[  \left(  f,\mathbf{X}\right)  \right]  .$ The
\textbf{weakly geometric rough paths on }$\mathcal{M}$ is the collection of
these equivalence classes;%
\[
WG_{p}\left(  \mathcal{M}\right)  :=\left\{  \left[  \left(  f,\mathbf{X}%
\right)  \right]  :\mathbf{X\in}WG_{p}\left(  M\right)  \right\}  .
\]

\end{definition}

If $\alpha\in\Omega^{1}\left(  \mathcal{M},W\right)  $ and $\left[  \left(
f,\mathbf{X}\right)  \right]  \in WG_{p}\left(  \mathcal{M}\right)  $ then we
define%
\[
Z_{\left[  \left(  f,\mathbf{X}\right)  \right]  }\left(  \alpha\right)
:=\int\alpha\left(  d\left[  \left(  f,\mathbf{X}\right)  \right]  \right)
:=\int\left(  \left(  f^{-1}\right)  ^{\ast}\alpha\right)  \left(
d\mathbf{X}\right)  \in WG_{p}\left(  W\right)  .
\]
Because of Corollary \ref{cor.3.40}, $Z_{\left[  \left(  f,\mathbf{X}\right)
\right]  }$ is well defined and because of Corollary \ref{cor.3.25}, knowledge
of $f$ and $Z_{\left[  \left(  f,\mathbf{X}\right)  \right]  }$ uniquely
determines $\mathbf{X}.$ The functionals $Z_{\left[  \left(  f,\mathbf{X}%
\right)  \right]  }$ are closely related to the notion of manifold valued
rough paths as introduced in \cite{Cass2012a}. An alternative, more explicit,
proof of the independence of the embedding for the rough paths will be given
in Cass, Driver, Litterer \cite{CDL} where another intrinsic notion of rough
paths will be developed.

\section{RDEs on manifolds and consequences\label{sec.4}}

In this section we consider rough differential equations constrained to $M,$
see Definition \ref{def.4.1} below. Theorem \ref{the.4.2} gives the basic
existence uniqueness results for constrained RDEs (rough differential
equations). The extrinsic Definition \ref{def.4.1} is shown in Theorem
\ref{the.4.5} to be equivalent to a pair of intrinsic notions of solutions for
constrained RDEs. In Example \ref{exa.4.12}, we use constrained RDEs to give
examples of weakly geometric rough paths on $M$ and then in Theorem
\ref{the.4.18} we show that all $\mathbf{X\in}WG_{p}\left(  M\right)  $ arise
as in Example \ref{exa.4.12}. The relationships between $WG_{p}\left(
M\right)  $ and $G_{p}\left(  M\right)  $ is spelled out in Theorem
\ref{the.4.17} and a summary of all of our characterizations of $WG_{p}\left(
M\right)  $ is then given in Theorem \ref{the.4.18}. As an illustration of our
results, in subsection \ref{sub.4.3} we study RDEs on a Lie group $G$ whose
dynamics are determined by right invariant vector fields on $G.$ This added
right invariance assumption guarantees that the resulting RDEs have global
solutions, see Theorem \ref{the.4.20}.

\subsection{Rough differential equations on $M$\label{sub.4.1}}

\begin{definition}
[Constrained RDE]\label{def.4.1}Let $x_{0}\in M,$ $Y:\mathbb{R}^{n}%
\rightarrow\Gamma\left(  TM\right)  $ be a linear function, and $\mathbf{Z\in
}WG_{p}\left(  \mathbb{R}^{n}\right)  $ be given. We say $\mathbf{X\in}%
WG_{p}\left(  M\right)  $ solves the RDE%
\begin{equation}
d\mathbf{X}_{t}=Y_{d\mathbf{Z}_{t}}\left(  x_{t}\right)  \text{ with }x\left(
0\right)  =x_{0}\in M \label{equ.4.1}%
\end{equation}
provided,%
\begin{align}
x_{s,t}  &  \simeq Y_{z_{s,t}}\left(  x_{s}\right)  +\left(  \partial_{Y_{a}%
}Y_{b}\right)  \left(  x_{s}\right)  |_{a\otimes b=\mathbb{Z}_{st}}\text{
and}\label{equ.4.2}\\
\mathbb{X}_{s,t}  &  \simeq Y_{a}\left(  x_{s}\right)  \otimes Y_{b}\left(
x_{s}\right)  |_{a\otimes b=\mathbb{Z}_{s,t}}. \label{equ.4.3}%
\end{align}

\end{definition}

Notice that $Y_{a}\left(  x_{s}\right)  \in T_{x_{s}}M$ and therefore there
exists a smooth curve $\sigma\left(  t\right)  \in M$ such that $\dot{\sigma
}\left(  0\right)  =Y_{a}\left(  x_{s}\right)  $ and we then compute $\left(
\partial_{Y_{a}}Y_{b}\right)  \left(  x_{s}\right)  $ using%
\[
\left(  \partial_{Y_{a}}Y_{b}\right)  \left(  x_{s}\right)  =\frac{d}{dt}%
|_{0}Y_{b}\left(  \sigma\left(  t\right)  \right)  \in E.
\]
This comment shows that the above definition makes sense but it is not yet
clear that there is a (local in time) solution to the RDE (\ref{equ.4.1}). If
$\mathbf{X\in}WG_{p}\left(  M\right)  $ solves Eq. (\ref{equ.4.1}), $U$ is an
open (in $E)$ neighborhood of $M,$ and $\tilde{Y}:\mathbb{R}^{n}%
\rightarrow\Gamma\left(  TU\right)  $ is a linear map such that $\tilde{Y}%
_{a}=Y_{a}$ on $M,$ then $\mathbf{X}$ solves the standard Euclidean space
RDE\
\begin{equation}
d\mathbf{X}_{t}=\tilde{Y}_{d\mathbf{Z}_{t}}\left(  x_{t}\right)  \text{ with
}x\left(  0\right)  =x_{0}\in M. \label{equ.4.4}%
\end{equation}
From these considerations we see that if there exists $\mathbf{X\in}%
WG_{p}\left(  M\right)  $ solving Eq. (\ref{equ.4.1}) then this solution may
be described as the unique solution $\mathbf{X\in}WG_{p}\left(  E\right)  $ to
Eq. (\ref{equ.4.4}). We will use this remark in our proof of the existence
Theorem \ref{the.4.2} below to Eq. (\ref{equ.4.1}). Once this is accomplished
we develop in Theorem \ref{the.4.5} alternative intrinsic characterizations of
solutions to the RDE\ in Eq. (\ref{equ.4.1}).

\begin{theorem}
\label{the.4.2}There is a unique solution $\mathbf{X}\in WG_{p}\left(
M\right)  $ (possibly up to explosion time) of the RDE (\ref{equ.4.1}).
Moreover, either $\mathbf{X}$ exists on all of $\left[  0,T\right]  $ or there
exists a $\tau\in\lbrack0,T]$ such that $\mathbf{X}$ exists on $[0,\tau)$ and
$\overline{\left\{  x\left(  t\right)  :0\leq t<\tau\right\}  }^{M}$ is not
compact in $M.$
\end{theorem}

\begin{proof}
Let $U$ be a neighborhood of $x_{0}\in M$ and $F:U\rightarrow\mathbb{R}^{k}$
be a local defining function of $M$ as in Definition \ref{def.3.1}. We then
let $\tilde{Y}_{a}:=P_{F}\left[  Y_{a}\circ\pi\right]  $ where $\pi$ is as in
Remark \ref{rem.3.2} and $P_{F}$ is the projection map in Notation
\ref{not.3.3}. Let $\mathbf{X=}\left(  x,\mathbb{X}\right)  \in WG_{p}\left(
E\right)  $ be the RDE\ solution to Eq. (\ref{equ.4.4}) defined up to the
first exit time $\tau$ from $U$ where we let $\tau=\infty$ if $x_{t}\in M$ for
all $0\leq t\leq T.$ We are now going to show $x\left(  t\right)  \in M\cap U$
for $0\leq t<\tau.$

Notice by construction that $\tilde{Y}_{b}=Y_{b}$ on $M\cap U$ and $F^{\prime
}\tilde{Y}_{b}=0$ on $U.$ Differentiating this last equation along $\xi
=\tilde{Y}_{a}$ then further implies,%
\begin{equation}
F^{\prime\prime}\tilde{Y}_{a}\otimes\tilde{Y}_{b}+F^{\prime}\partial
_{\tilde{Y}_{a}}\tilde{Y}_{b}=0. \label{equ.4.5}%
\end{equation}
Recall that $\mathbf{X}$ solves Eq. (\ref{equ.4.4}) iff%
\begin{equation}
x_{s,t}\simeq\tilde{Y}_{z_{s,t}}\left(  x_{s}\right)  +\left(  \partial
_{\tilde{Y}_{a}\left(  x_{s}\right)  }\tilde{Y}_{b}\right)  \left(
x_{s}\right)  |_{a\otimes b=\mathbb{Z}_{s,t}}\text{ and \ }\mathbb{X}%
_{s,t}\simeq\left[  \tilde{Y}\left(  x_{s}\right)  \otimes\tilde{Y}\left(
x_{s}\right)  \right]  \mathbb{Z}_{s,t}. \label{equ.4.6}%
\end{equation}
Using the first approximate identity in Eq. (\ref{equ.4.6}) along with
$F^{\prime}\tilde{Y}=0$ shows%
\begin{align}
\left[  F\left(  x_{\cdot}\right)  \right]  _{s,t}:=  &  F\left(
x_{t}\right)  -F\left(  x_{s}\right)  =F\left(  x_{s}+x_{s,t}\right)
-F\left(  x_{s}\right) \nonumber\\
\simeq &  F^{\prime}\left(  x_{s}\right)  x_{s,t}+\frac{1}{2}F^{\prime\prime
}\left(  x_{s}\right)  x_{s,t}\otimes x_{s,t}\nonumber\\
\simeq &  F^{\prime}\left(  x_{s}\right)  \left[  \tilde{Y}_{z_{s,t}}\left(
x_{s}\right)  +\left(  \partial_{\tilde{Y}_{a}\left(  x_{s}\right)  }\tilde
{Y}_{b}\right)  \left(  x_{s}\right)  |_{a\otimes b=\mathbb{Z}_{s,t}}\right]
+\frac{1}{2}F^{\prime\prime}\left(  x_{s}\right)  \left[  \tilde{Y}_{z_{s,t}%
}\left(  x_{s}\right)  \otimes\tilde{Y}_{z_{s,t}}\left(  x_{s}\right)  \right]
\nonumber\\
=  &  F^{\prime}\left(  x_{s}\right)  \left(  \partial_{\tilde{Y}_{a}\left(
x_{s}\right)  }\tilde{Y}_{b}\right)  \left(  x_{s}\right)  |_{a\otimes
b=\mathbb{Z}_{s,t}}+\frac{1}{2}F^{\prime\prime}\left(  x_{s}\right)  \left[
\tilde{Y}_{z_{s,t}}\left(  x_{s}\right)  \otimes\tilde{Y}_{z_{s,t}}\left(
x_{s}\right)  \right]  . \label{equ.4.7}%
\end{align}
Since $F^{\prime\prime}\left(  x_{s}\right)  $ is symmetric and $\mathbf{Z}$
is a geometric rough path it follows that Eq. (\ref{equ.4.5}) and
\begin{equation}
\frac{1}{2}F^{\prime\prime}\left(  x_{s}\right)  \left[  \tilde{Y}_{z_{s,t}%
}\left(  x_{s}\right)  \otimes\tilde{Y}_{z_{s,t}}\left(  x_{s}\right)
\right]  =F^{\prime\prime}\left(  x_{s}\right)  \left[  \tilde{Y}_{a}\left(
x_{s}\right)  \otimes\tilde{Y}_{b}\left(  x_{s}\right)  \right]  |_{a\otimes
b=\mathbb{Z}_{st}}. \label{equ.4.8}%
\end{equation}
Combining Eqs. (\ref{equ.4.5}), (\ref{equ.4.7}), and (\ref{equ.4.8}) shows
$\left[  F\left(  x_{\cdot}\right)  \right]  _{s,t}\simeq0$ which implies
$F\left(  x_{t}\right)  $ is constant in $t\in\lbrack0,\tau).$ Since $F\left(
x_{0}\right)  =0$ it follows that $F\left(  x_{t}\right)  =0$ for $t<\tau,$
i.e. $x\left(  t\right)  \in M$ for $t<\tau.$ Also notice that
\[
I\otimes Q\left(  x_{s}\right)  \mathbb{X}_{s,t}\simeq I\otimes Q\left(
x_{s}\right)  \left[  \tilde{Y}\left(  x_{s}\right)  \otimes\tilde{Y}\left(
x_{s}\right)  \right]  \mathbb{Z}_{s,t}\simeq0,
\]
and therefore $\mathbf{X}\in WG_{p}\left(  [0,\tau),M\right)  $ and we have
proved local existence to Eq. (\ref{equ.4.4}). This shows local existence to
Eq. (\ref{equ.4.1}).

Suppose that we have found $\mathbf{X}\in WG_{p}\left(  [0,\tau),M\right)  $
solving Eq. (\ref{equ.4.1}) on $[0,\tau)$ for some $\tau\leq T.$ If there
exists a compact subset $K\subset M$ such that $\left\{  x\left(  t\right)
:t<\tau\right\}  \subset K,$ then there exists $t_{n}\in\lbrack0,\tau)$ such
that $t_{n}\uparrow\tau$ and $x_{\infty}:=\lim_{n\rightarrow\infty}x\left(
t_{n}\right)  $ exists in $K\subset M.$ We now let $U$ be a precompact
neighborhood of $x_{\infty}\in M$ and $F:U\rightarrow\mathbb{R}^{k}$ be a
local defining function of $M$ as in Definition \ref{def.3.1} and as above let
$\tilde{Y}_{a}:=P_{F}\left[  Y_{a}\circ\pi\right]  $ on $U.$ Moreover we may
assume $\tilde{Y}$ is compactly supported. By Corollary \ref{cor.2.17} there
exists an $\varepsilon>0$ and a neighborhood $V\subset U$ of $x_{\infty}$ such
that for any $s\in\lbrack\tau-\varepsilon,\tau]$ and $y\in V$ there exists
$\mathbf{\hat{X}\in}WG_{p}\left(  [s,\tau+\varepsilon],E\right)  $ with trace
in $U$ solving
\[
d\mathbf{\hat{X}}_{t}\mathbf{=}\tilde{Y}_{d\mathbf{Z}_{t}}\left(
x_{t}\right)  \text{ with }x_{s}=y\in V.
\]
We then choose $n$ sufficiently large so that $t_{n}\in\lbrack\tau
-\varepsilon,\tau]$ and let $\mathbf{\hat{X}\in}WG_{p}\left(  [t_{n}%
,\tau+\varepsilon],E\right)  $ solve the previous equation with $y=x\left(
t_{n}\right)  .$ We may now apply the concatenation Lemma \ref{lem.A.2} to
glue $\mathbf{X}$ and $\mathbf{\hat{X}}$ together to show there exists a
solution to Eq. (\ref{equ.4.1}) on $[0,\tau+\varepsilon].$

Let us now consider the case where Eq. (\ref{equ.4.1}) does not admit a global
solution defined on $\left[  0,T\right]  .$ In this case, let%
\[
\tau=\sup\left\{  T_{0}\in\left(  0,T\right)  :~\exists~\mathbf{\tilde{X}}\in
WG_{p}\left(  \left[  0,T_{0}\right]  ,M\right)  \text{ solving }%
(\ref{equ.4.1})\right\}  \in(0,T]
\]
and for $0\leq s\leq t<\tau$ let $\mathbf{X}_{s,t}:=\mathbf{\tilde{X}}_{s,t}$
where $\mathbf{\tilde{X}}\in WG_{p}\left(  \left[  0,T_{0}\right]  ,M\right)
$ solves Eq. $(\ref{equ.4.1})$ on $\left[  0,T_{0}\right]  $ for some
$T_{0}\in\left(  t,\tau\right)  .$ By the uniqueness part of Theorem
\ref{the.2.16}, $\left\{  \mathbf{X}_{s,t}:0\leq s\leq t<\tau\right\}  $ is
well defined and satisfies Eq. $(\ref{equ.4.1})$ on $[0,\tau).$ If (for the
sake of contradiction) $\overline{\left\{  x\left(  t\right)  :0\leq
t<\tau\right\}  }^{M}$ were compact, then the procedure above allows us to
produce a solution $\mathbf{\hat{X}}$ to Eq. $(\ref{equ.4.1})$ which is valid
on $\left[  0,\min\left(  \tau+\epsilon,T\right)  \right]  $ which would
violate either the definition of $\tau$ or the assumption that no global
solution to Eq. (\ref{equ.4.1}) exists on $\left[  0,T\right]  .$ Hence we
must conclude that $\overline{\left\{  x\left(  t\right)  :0\leq
t<\tau\right\}  }^{M}$ is not compact.
\end{proof}

We now prepare an equivalent intrinsic characterizations of an RDE solution.
The following proposition is a consequence of the universality property of the
full tensor algebra,%
\[
\mathcal{T}\left(  \mathbb{R}^{n}\right)  =\mathbb{R}\oplus_{k=1}^{\infty
}\left[  \mathbb{R}^{n}\right]  ^{\otimes k},
\]
over $\mathbb{R}^{n}.$

\begin{proposition}
\label{pro.4.3}Let $\mathcal{L}\left(  M\right)  $ denote the collection of
all linear differential operators on $C^{\infty}\left(  M\right)  .$ If
$Y:\mathbb{R}^{n}\rightarrow\Gamma\left(  TM\right)  $ is a linear map, then
$Y$ extends uniquely to an algebra homomorphism, $\mathcal{Y}:\mathcal{T}%
\left(  \mathbb{R}^{n}\right)  \rightarrow\mathcal{L}\left(  M\right)  $ such
that $\mathcal{Y}_{1}:=Id\in\mathcal{L}\left(  M\right)  ,$ where
$1\in\mathcal{T}\left(  \mathbb{R}^{n}\right)  .$
\end{proposition}

\begin{example}
\label{exa.4.4}If $A\in\mathbb{R}^{n}\otimes\mathbb{R}^{n},$ then
$\mathcal{Y}_{A}=Y_{a}Y_{b}|_{a\otimes b=A}$ wherein we are using the
conventions introduced in Notation \ref{not.2.2}.
\end{example}

\begin{theorem}
\label{the.4.5}Let $Y:\mathbb{R}^{n}\rightarrow\Gamma\left(  TM\right)  $ be a
linear map and $\mathbf{X\in}WG_{p}\left(  M\right)  .$ Then the following are equivalent:

\begin{enumerate}
\item $\mathbf{X}$ solves the RDE in Eq. (\ref{equ.4.1}).

\item For any finite dimensional vector space $W$ and any $\alpha\in\Omega
^{1}\left(  M,W\right)  ,$
\begin{equation}
\left[  \int\alpha\left(  d\mathbf{X}\right)  \right]  _{s,t}^{1}\simeq
\alpha_{x_{s}}\left(  Y_{z_{s,t}}\left(  x_{s}\right)  \right)  +\left[
Y_{a}\left(  x_{s}\right)  \alpha\left(  Y_{b}\right)  \right]  |_{a\otimes
b=\mathbb{Z}_{s,t}}, \label{equ.4.9}%
\end{equation}
and%
\begin{equation}
\left[  \int\alpha\left(  d\mathbf{X}\right)  \right]  _{s,t}^{2}\simeq\left[
\alpha_{x_{s}}Y_{a}\left(  x_{s}\right)  \otimes\alpha_{x_{s}}Y_{b}\left(
x_{s}\right)  \right]  |_{a\otimes b=\mathbb{Z}_{s,t}}. \label{equ.4.10}%
\end{equation}

\item For any finite dimensional vector space $W$ and any $f\in C^{\infty
}\left(  M,W\right)  ,$%
\begin{equation}
f\left(  x_{t}\right)  -f\left(  x_{s}\right)  \simeq\left(  \mathcal{Y}%
_{\mathbf{Z}_{s,t}}f\right)  \left(  x_{s}\right)  \label{equ.4.11}%
\end{equation}
and%
\begin{equation}
\left(  df\otimes df\right)  \left(  \left[  P_{x_{s}}\otimes P_{x_{s}%
}\right]  \mathbb{X}_{s,t}\right)  \simeq Y_{a}f\left(  x_{s}\right)  \otimes
Y_{b}f\left(  x_{s}\right)  |_{a\otimes b=\mathbb{Z}_{s,t}}. \label{equ.4.12}%
\end{equation}

\end{enumerate}
\end{theorem}

\begin{proof}
We will show $1.\implies2.\implies3.\implies1.$

$\left(  1.\implies2.\right)  $ From Eqs. (\ref{equ.3.17}), (\ref{equ.4.2}),
and (\ref{equ.4.3}),%
\begin{align*}
\left[  \int\alpha\left(  d\mathbf{X}\right)  \right]  _{s,t}^{1}\simeq &
\alpha_{x_{s}}\left(  P_{x_{s}}x_{st}\right)  +\left(  \nabla\alpha\right)
\left(  \left[  P_{x_{s}}\otimes P_{x_{s}}\right]  \mathbb{X}_{s,t}\right) \\
\simeq &  \alpha_{x_{s}}\left(  P_{x_{s}}\left[  Y_{z_{s,t}}\left(
x_{s}\right)  +\left(  \partial_{Y_{a}}Y_{b}\right)  \left(  x_{s}\right)
|_{a\otimes b=\mathbb{Z}_{st}}\right]  \right) \\
&  \qquad+\left(  \nabla\alpha\right)  \left(  \left[  P_{x_{s}}\otimes
P_{x_{s}}\right]  Y_{a}\left(  x_{s}\right)  \otimes Y_{b}\left(
x_{s}\right)  |_{a\otimes b=\mathbb{Z}_{s,t}}\right) \\
\simeq &  \alpha_{x_{s}}\left(  Y_{z_{s,t}}\left(  x_{s}\right)  +P_{x_{s}%
}\left(  \partial_{Y_{a}}Y_{b}\right)  \left(  x_{s}\right)  |_{a\otimes
b=\mathbb{Z}_{st}}\right)  +\left(  \nabla\alpha\right)  \left(  Y_{a}\left(
x_{s}\right)  \otimes Y_{b}\left(  x_{s}\right)  |_{a\otimes b=\mathbb{Z}%
_{s,t}}\right) \\
\simeq &  \alpha_{x_{s}}\left(  Y_{z_{s,t}}\left(  x_{s}\right)
+\nabla_{Y_{a}\left(  x_{s}\right)  }Y_{b}|_{a\otimes b=\mathbb{Z}_{st}%
}\right)  +\left(  \nabla_{Y_{a}\left(  x_{s}\right)  }\alpha\right)  \left(
Y_{b}\left(  x_{s}\right)  \right)  |_{a\otimes b=\mathbb{Z}_{s,t}}.
\end{align*}
Combining this approximate identity with the product rule for covariant
derivatives in item 2. of Proposition \ref{pro.3.12} gives Eq. (\ref{equ.4.9}%
). Equation (\ref{equ.4.10}) follows easily from Eqs. (\ref{equ.3.18}) and
(\ref{equ.4.3});%
\begin{align*}
\left[  \int\alpha\left(  d\mathbf{X}\right)  \right]  _{s,t}^{2}  &
\simeq\alpha_{x_{s}}\otimes\alpha_{x_{s}}\left[  P_{x_{s}}\otimes P_{x_{s}%
}\right]  \mathbb{X}_{st}\\
&  \simeq\alpha_{x_{s}}\otimes\alpha_{x_{s}}\left[  P_{x_{s}}\otimes P_{x_{s}%
}\right]  \left[  Y_{a}\left(  x_{s}\right)  \otimes Y_{b}\left(
x_{s}\right)  \right]  |_{a\otimes b=\mathbb{Z}_{s,t}}\\
&  =\alpha_{x_{s}}\otimes\alpha_{x_{s}}\left[  Y_{a}\left(  x_{s}\right)
\otimes Y_{b}\left(  x_{s}\right)  \right]  |_{a\otimes b=\mathbb{Z}_{s,t}},
\end{align*}
wherein we have used $P_{x_{s}}Y_{\left(  \cdot\right)  }\left(  x_{s}\right)
=Y_{\left(  \cdot\right)  a}\left(  x_{s}\right)  $ in the last equality.

$\left(  2.\implies3.\right)  $ Applying item 2. with $\alpha=df$ shows%
\begin{align*}
\left[  \int\alpha\left(  d\mathbf{X}\right)  \right]  _{s,t}^{1}  &  \simeq
df_{x_{s}}\left(  Y_{z_{s,t}}\left(  x_{s}\right)  \right)  +\left[
Y_{a}\left(  x_{s}\right)  df_{x_{s}}\left(  Y_{b}\right)  \right]
|_{a\otimes b=\mathbb{Z}_{s,t}}\\
&  \simeq df_{x_{s}}\left(  Y_{z_{s,t}}\left(  x_{s}\right)  \right)  +\left[
Y_{b}Y_{b}f\left(  x_{s}\right)  \right]  |_{a\otimes b=\mathbb{Z}_{s,t}}\\
&  =\left(  \mathcal{Y}_{z_{s,t}+\mathbb{Z}_{s,t}}f\right)  \left(
x_{s}\right)
\end{align*}
and
\[
\left[  \int\alpha\left(  d\mathbf{X}\right)  \right]  _{s,t}^{2}\simeq\left[
df_{x_{s}}Y_{a}\left(  x_{s}\right)  \otimes df_{x_{s}}Y_{b}\left(
x_{s}\right)  \right]  |_{a\otimes b=\mathbb{Z}_{s,t}}.
\]
This shows item 3. holds once we recall that%
\[
\left[  \int df\left(  d\mathbf{X}\right)  \right]  _{s,t}^{1}=f_{\ast}\left(
\mathbf{X}\right)  _{s,t}^{1}=f\left(  x_{t}\right)  -f\left(  x_{s}\right)
\]
and
\[
\left[  \int df\left(  d\mathbf{X}\right)  \right]  _{s,t}^{2}\simeq
df_{x_{s}}\otimes df_{x_{s}}\left[  P_{x_{s}}\otimes P_{x_{s}}\right]
\mathbb{X}_{st}.
\]

$\left(  3.\implies1.\right)  $ Let $W=E$ and $f:M\rightarrow E$ be the
restrictions of the identity map on $E,$ i.e. $f\left(  x\right)  =x$ for all
$x\in M.$ For this $f,$ we have $df\left(  v_{x}\right)  =v$ for all $v_{x}\in
TM,$%
\[
\left(  Y_{b}f\right)  \left(  x\right)  =Y_{b}\left(  x\right)  \text{ and
}\left(  Y_{a}Y_{b}f\right)  \left(  x\right)  =\left(  \partial_{Y_{a}}%
Y_{b}\right)  \left(  x\right)
\]
and so Eq. (\ref{equ.4.11}) becomes,%
\[
x_{s,t}=f\left(  x_{t}\right)  -f\left(  x_{s}\right)  \simeq\left(
\mathcal{Y}_{\mathbf{Z}_{s,t}}f\right)  \left(  x_{s}\right)  =Y_{z_{s,t}%
}\left(  x_{s}\right)  +\left(  \partial_{Y_{a}}Y_{b}\right)  \left(
x_{s}\right)  |_{a\otimes b=\mathbb{Z}_{s,t}}%
\]
which is precisely Eq. (\ref{equ.4.2}). Similarly Eq. (\ref{equ.4.12}) becomes%
\begin{align*}
\left[  P_{x_{s}}\otimes P_{x_{s}}\right]  \mathbb{X}_{s,t}  &  =\left(
df\otimes df\right)  \left(  \left[  P_{x_{s}}\otimes P_{x_{s}}\right]
\mathbb{X}_{s,t}\right) \\
&  \simeq Y_{a}f\left(  x_{s}\right)  \otimes Y_{b}f\left(  x_{s}\right)
|_{a\otimes b=\mathbb{Z}_{s,t}}=Y_{a}\left(  x_{s}\right)  \otimes Y\left(
x_{s}\right)  |_{a\otimes b=\mathbb{Z}_{s,t}}%
\end{align*}
which is equivalent to Eq. (\ref{equ.4.3}) by Corollary \ref{cor.3.22}.
\end{proof}

\begin{remark}
\label{rem.4.6}If we restrict $W$ to be $\mathbb{R}$ in Theorem \ref{the.4.5}
we may still conclude from either of items 2. or 3. of that theorem that
$\mathbf{X}$ satisfies Eq. (\ref{equ.4.2}), i.e. the level one condition for
the RDE solution (\ref{equ.4.1}). Indeed if $f=\ell|_{M}:M\rightarrow
\mathbb{R}$ where $\ell\in E^{\ast}$ is any linear functional on $E,$ then%
\[
\left(  Y_{a}f\right)  \left(  x\right)  =\ell Y_{a}\left(  x\right)  \text{
and }\left(  Y_{a}Y_{b}f\right)  \left(  x\right)  =\left(  Y_{a}\ell
Y_{b}\right)  \left(  x\right)  =\ell\left(  \partial_{Y_{a}}Y_{b}\right)
\left(  x\right)  .
\]
So for $f=\ell|_{M},$ Eq. (\ref{equ.4.11}) becomes%
\[
\ell\left(  x_{s,t}\right)  =f\left(  x_{t}\right)  -f\left(  x_{s}\right)
\simeq\ell Y_{z_{s,t}}\left(  x_{s}\right)  +\ell\left(  \partial_{Y_{a}}%
Y_{b}\right)  \left(  x_{s}\right)  |_{a\otimes b=\mathbb{Z}_{s,t}}.
\]
As this true for all $\ell\in E^{\ast}$ we may conclude Eq. (\ref{equ.4.2}) holds.
\end{remark}

\begin{remark}
\label{rem.4.7}We can not get Eq. (\ref{equ.4.3}) from Eq. (\ref{equ.4.10})
without allowing for $\dim W>1.$ Indeed, if $\alpha\in\Omega^{1}\left(
M,\mathbb{R}\right)  ,$ then%
\[
\alpha_{x_{s}}\otimes\alpha_{x_{s}}\left(  \left[  P_{x_{s}}\otimes P_{x_{s}%
}\right]  \mathbb{X}_{s,t}\right)  \simeq\alpha_{x_{s}}\otimes\alpha_{x_{s}%
}\left[  Y_{a}\left(  x_{s}\right)  \otimes Y_{b}\left(  x_{s}\right)
\right]  |_{a\otimes b=\mathbb{Z}_{s,t}}%
\]
from which we may only conclude that
\begin{equation}
\left[  P_{x_{s}}\otimes P_{x_{s}}\right]  \mathbb{X}_{s,t}^{s}\simeq\left[
Y_{a}\left(  x_{s}\right)  \otimes Y_{b}\left(  x_{s}\right)  \right]
|_{a\otimes b=\mathbb{Z}_{s,t}^{s}}. \label{equ.4.13}%
\end{equation}
This is because $\alpha\otimes\alpha\left(  \xi\otimes\eta-\eta\otimes
\xi\right)  =\alpha\left(  \xi\right)  \alpha\left(  \eta\right)
-\alpha\left(  \eta\right)  \alpha\left(  \xi\right)  =0$ since scalar
multiplication is commutative. Here we have used that $\mathbb{R}%
\otimes_{\mathbb{R}}\mathbb{R\cong R}.$ The reader should further observe that
information contained in Eq. (\ref{equ.4.13}) is already a consequence of Eq.
(\ref{equ.4.2}) and the assumption that $\mathbf{Z}$ and $\mathbf{X}$ are
weakly geometric rough paths.
\end{remark}

\begin{definition}
[Intrinsic RDEs on Manifolds]\label{def.4.8}Given a linear map $Y:\mathbb{R}%
^{n}\rightarrow\Gamma\left(  TM\right)  ,$ we say that a geometric rough path
$\mathbf{X\in}WG_{p}\left(  M\right)  $ solves the RDE%
\begin{equation}
d\mathbf{X}_{t}=Y_{d\mathbf{Z}_{t}}\left(  x_{t}\right)  \text{ with }x\left(
0\right)  =x_{0}\in M \label{equ.4.14}%
\end{equation}
if and only if equations (\ref{equ.4.11}) and (\ref{equ.4.12})hold for all
$f\in C^{\infty}\left(  M,W\right)  $ and every finite dimensional vector
space $W.$
\end{definition}

\begin{notation}
[Intrinsic RDEs]\label{not.4.9}To emphasize when we are working with the
intrinsic definition of an RDE we sometimes write%
\[
d\mathbf{X}_{t}=\mathcal{Y}_{d\mathbf{Z}_{t}}\left(  x_{t}\right)  \text{ with
}x\left(  0\right)  =x_{0}\in M
\]
in place of (\ref{equ.4.14}) where now $\mathbf{Z}_{s,t}=z_{s,t}%
+\mathbb{Z}_{s,t}$ and we interpret%
\[
\left[  \mathcal{Y}_{d\mathbf{Z}}\left(  x\right)  \right]  _{s,t}^{2}%
\simeq\left[  Y\left(  x_{s}\right)  \otimes Y\left(  x_{s}\right)  \right]
\mathbb{Z}_{s,t}\in\mathbb{R}^{N}\otimes\mathbb{R}^{N}.
\]

\end{notation}

We end this subsection with a result describing (in special cases) the push
forward of solutions to RDEs.

\begin{definition}
\label{def.4.10}Suppose that $\pi:M\rightarrow N$ is a smooth map between two
smooth manifolds. Also suppose that $Y^{M}:\mathbb{R}^{n}\rightarrow
\Gamma\left(  TM\right)  $ and $Y^{N}:\mathbb{R}^{n}\rightarrow\Gamma\left(
TN\right)  $ are two linear maps. We say $Y^{M}$ and $Y^{N}$ are $\pi$\textbf{
-- related dynamical systems} if
\[
\pi_{\ast}Y_{a}^{M}=Y_{a}^{N}\circ\pi\text{ for all }a\in\mathbb{R}^{n}.
\]

\end{definition}

\begin{theorem}
\label{the.4.11}Suppose $\pi:M\rightarrow N$ is a smooth map between
manifolds. Let $Y^{M}:\mathbb{R}^{n}\rightarrow\Gamma\left(  TM\right)  $ and
$Y^{N}:\mathbb{R}^{n}\rightarrow\Gamma\left(  TN\right)  $ be two $\pi$ --
related dynamical systems. Further suppose that $\mathbf{Z\in}WG_{p}\left(
\mathbb{R}^{n}\right)  $ and $\mathbf{X}=\left(  x,\mathbb{X}\right)  $ solves
the RDE,%
\[
d\mathbf{X}_{t}=\mathcal{Y}_{d\mathbf{Z}_{t}}^{M}\left(  x_{t}\right)  \text{
with }x_{0}\in M\text{ given.}%
\]
Then $\mathbf{X}^{N}:=\pi_{\ast}\left(  \mathbf{X}\right)  =\left(
x^{N},\mathbb{X}^{N}\right)  $ solves the RDE,%
\[
d\mathbf{X}_{t}^{N}=\mathcal{Y}_{d\mathbf{Z}_{t}}^{N}\left(  x_{t}^{N}\right)
\text{ with }x_{0}^{N}:=\pi\left(  x_{0}\right)  \in N\text{ given.}%
\]

\end{theorem}

\begin{proof}
Fix a finite dimensional vector space $W$ and let $f\in C^{\infty}\left(
N,W\right)  .$ Applying item 3. of Theorem \ref{the.4.5} to the function
$f\circ\pi\in C^{\infty}\left(  M,W\right)  $ shows,%
\begin{align*}
f\left(  x_{t}^{N}\right)  -f\left(  x_{s}^{N}\right)   &  =f\circ\pi\left(
x_{t}\right)  -f\circ\pi\left(  x_{s}\right) \\
&  \simeq\left(  \mathcal{Y}_{\mathbf{Z}_{st}}^{M}\left(  f\circ\pi\right)
\right)  \left(  x_{s}\right)  =\left(  \mathcal{Y}_{\mathbf{Z}_{st}}%
^{N}f\right)  \circ\pi\left(  x_{s}\right)  =\left(  \mathcal{Y}%
_{\mathbf{Z}_{st}}^{N}f\right)  \left(  x_{s}^{N}\right)
\end{align*}
and
\begin{align*}
\mathbb{X}_{s,t}^{N}  &  \simeq\pi_{\ast}\otimes\pi_{\ast}P\left(
x_{s}\right)  \otimes P\left(  x_{s}\right)  \mathbb{X}_{s,t}\\
&  \simeq\pi_{\ast}\otimes\pi_{\ast}P\left(  x_{s}\right)  \otimes P\left(
x_{s}\right)  Y_{\left(  \cdot\right)  }^{M}\left(  x_{s}\right)  \otimes
Y_{\left(  \cdot\right)  }^{M}\left(  x_{s}\right)  \mathbb{Z}_{s,t}\\
&  =\pi_{\ast}Y_{\left(  \cdot\right)  }^{M}\left(  x_{s}\right)  \otimes
\pi_{\ast}Y_{\left(  \cdot\right)  }^{M}\left(  x_{s}\right)  \mathbb{Z}%
_{s,t}\\
&  =Y_{\left(  \cdot\right)  }^{N}\left(  x_{s}^{N}\right)  \otimes Y_{\left(
\cdot\right)  }^{N}\left(  x_{s}^{N}\right)  \mathbb{Z}_{s,t}.
\end{align*}

\end{proof}

\subsection{Fundamental properties of rough paths on manifolds\label{sub.4.2}}

Armed with well-defined notions of integration and RDEs, we now derive some of
the fundamental properties of geometric and weakly geometric rough paths on
manifolds. We also exhibit some natural examples of elements in $WG_{p}\left(
M\right)  $ which are constructed by \textquotedblleft projecting the
increments\textquotedblright\ of geometric rough paths on $E$ to the tangent
space of $M.$

\begin{example}
[Projection Construction of Geometric Rough Paths]\label{exa.4.12}Let
$\mathbf{Z}$ be a weakly geometric $p-$ rough path on $E$ for some
$p\in\lbrack2,3),$ then there exists a unique rough path solution $\mathbf{X}$
(possibly only up to an explosion time) to the RDE%
\begin{equation}
d\mathbf{X}_{t}=V_{d\mathbf{Z}_{t}}\left(  x_{t}\right)  =P_{x_{t}}%
d\mathbf{Z}_{t}\text{ with }x_{0}\in M. \label{equ.4.15}%
\end{equation}
Moreover it will follow from Theorem \ref{the.4.2} and Theorem \ref{the.4.17}
that $\mathbf{X}\in G_{p^{\prime}}\left(  M\right)  \cap WG_{p}\left(
M\right)  $ for all $p^{\prime}>p.$
\end{example}

The following proposition shows that, in fact, all weakly geometric rough
paths on $M$ may be constructed by this method.

\begin{proposition}
[Consistency]\label{pro.4.13}If $\mathbf{Z}\in WG_{p}\left(  M\right)  \subset
WG_{p}\left(  E\right)  ,$ then the unique solution to Eq. (\ref{equ.4.15})
with $x_{0}=z_{0}$ is $\mathbf{X\equiv Z}.$ [So in this setting the solution
to Eq. (\ref{equ.4.15}) exists on all of $\left[  0,T\right]  .$]
\end{proposition}

\begin{proof}
The proof amounts to showing that $\mathbf{X}=\mathbf{Z}$ solves Eq.
(\ref{equ.4.15}), i.e. that
\begin{align}
z_{s,t}  &  \simeq V_{z_{s,t}}\left(  z_{s}\right)  +\left(  \partial_{V_{a}%
}V_{b}\right)  \left(  z_{s}\right)  |_{a\otimes b=\mathbb{Z}_{s,t}}\text{ and
}\label{equ.4.16}\\
\mathbb{Z}_{s,t}  &  \simeq V_{a}\left(  z_{s}\right)  \otimes V_{b}\left(
z_{s}\right)  |_{a\otimes b=\mathbb{Z}_{s,t}}=P\left(  z_{s}\right)  \otimes
P\left(  z_{s}\right)  \mathbb{Z}_{s,t}. \label{equ.4.17}%
\end{align}
Equation (\ref{equ.4.17}) is a consequence of Corollary \ref{cor.3.20}. The
right side of Eq. (\ref{equ.4.16}) is approximated as%
\begin{align*}
P\left(  z_{s}\right)  z_{s,t}+dP\left(  \left[  P\left(  z_{s}\right)
a\right]  _{z_{s}}\right)  b|_{a\otimes b=\mathbb{Z}_{s,t}}  &  \simeq
P\left(  z_{s}\right)  z_{s,t}+dP\left(  \left[  P\left(  z_{s}\right)
a\right]  _{z_{s}}\right)  P\left(  z_{s}\right)  b|_{a\otimes b=\mathbb{Z}%
_{s,t}}\\
&  \simeq P\left(  z_{s}\right)  z_{s,t}+Q\left(  z_{s}\right)  z_{s,t}%
=z_{s,t},
\end{align*}
wherein we have used Lemma \ref{lem.3.28} for the second approximate equality above.
\end{proof}

We now address the relation between geometric and weakly geometric rough paths
on manifolds. To do this we first require a couple of elementary lemmas.

\begin{lemma}
\label{lem.4.14}Suppose $U$ is an open neighborhood of $M$ and $\mathbb{R}%
^{n}\ni a\rightarrow\tilde{Y}_{a}\in\Gamma\left(  TU\right)  $ is a linear map
such that $\tilde{Y}_{a}\left(  m\right)  \in T_{m}M$ for all $m\in M.$
Further suppose that $z:\left[  0,T\right]  \rightarrow\mathbb{R}^{n}$ is a
smooth function and $x:\left[  0,T\right]  \rightarrow U$ is a smooth solution
to
\begin{equation}
\dot{x}\left(  t\right)  =\tilde{Y}_{\dot{z}\left(  t\right)  }\left(
x\left(  t\right)  \right)  \text{ with }x\left(  0\right)  =x_{0}\in M.
\label{equ.4.18}%
\end{equation}
If there is an open neighborhood, $\mathcal{V},$ in $E$ such that $x\left(
\left[  0,T\right]  \right)  \subset\mathcal{V}$ and $\overline{\mathcal{V\cap
}M}^{E}\subset M,$ then $x\left(  t\right)  \in M$ for all $t\in\left[
0,T\right]  $ and $x\left(  t\right)  $ satisfies, $\dot{x}\left(  t\right)
=Y_{\dot{z}\left(  t\right)  }\left(  x\left(  t\right)  \right)  $ with
$x\left(  0\right)  =x_{0}\in M.$
\end{lemma}

\begin{proof}
By replacing $\mathcal{V}$ by $\mathcal{V\cap U}$ we may assume that
$\mathcal{V}\subset U.$ For the sake of contradiction, suppose that $x\left(
\left[  0,T\right]  \right)  $ is not contained in $M$ and let $\tau
=\inf\left\{  t\in\left[  0,T\right]  :x\left(  t\right)  \notin M\right\}  $
be the first exit time of $x_{\left(  \cdot\right)  }$ from $M.$ Since
$x\left(  0\right)  =x_{0}\in M$ and $x\left(  t\right)  \in M$ for all $0\leq
t<\tau$ if $\tau>0$ we may conclude that $x\left(  \tau\right)  \in
\overline{\mathcal{V\cap}M}^{E}\subset M.$ As $x\left(  \left[  0,T\right]
\right)  $ is not contained in $M$ we may now conclude that $\tau<T.$

By the local existence theorem for the ODEs, there exists an $\varepsilon>0$
and a solution $y:\left[  \tau,\tau+\varepsilon\right]  \rightarrow M$ solving%
\[
\dot{y}\left(  t\right)  =Y_{\dot{z}\left(  t\right)  }\left(  y\left(
t\right)  \right)  \text{ with }y\left(  \tau\right)  =x\left(  \tau\right)
.
\]
The function $\tilde{x}:\left[  0,\tau+\varepsilon\right]  \rightarrow M$
defined by
\[
\tilde{x}\left(  t\right)  :=\left\{
\begin{array}
[c]{ccc}%
x\left(  t\right)  & \text{if} & 0\leq t\leq\tau\\
y\left(  t\right)  & \text{if} & \tau\leq t\leq\tau+\varepsilon
\end{array}
\right.
\]
then solves Eq. (\ref{equ.4.18}) on $\left[  0,\tau+\varepsilon\right]  $ and
hence by uniqueness, $x\left(  t\right)  =\tilde{x}\left(  t\right)  $ for
$0\leq t\leq\tau+\varepsilon.$ This however shows $x\left(  t\right)  \in M$
for $0\leq t\leq\tau+\varepsilon$ which contradicts the definition of $\tau.$
\end{proof}

\begin{lemma}
\label{lem.4.15}If $K$ is a compact subset of $M,$ there exists an open
neighborhood $\mathcal{V}$ in $E$ containing $K$ such that $\overline
{M\cap\mathcal{V}}^{E}\subset M.$
\end{lemma}

\begin{proof}
First suppose $K=\left\{  x\right\}  \subset M.$ Let $F:U\rightarrow
\mathbb{R}^{N-d}$ be a local defining function for $M$ so that $x\in U$ and
$U\cap M=\left\{  F=0\right\}  .$ Let $\mathcal{V}_{x}$ be a precompact open
neighborhood of $x$ in $E$ so that $\mathcal{\bar{V}}_{x}\subset U.$ Since
$\overline{M\cap\mathcal{V}_{x}}^{E}\subset\overline{\mathcal{V}_{x}}%
^{E}\subset U$ and $F\equiv0$ on $M\cap\mathcal{V}_{x}$ it follows by
continuity that $F=0$ on $\overline{M\cap\mathcal{V}_{x}}^{E}$ from which it
follows that $\overline{M\cap\mathcal{V}_{x}}^{E}\subset\left\{  F=0\right\}
\subset M.$

If $K$ is a general compact subset of $M,$ to each $x\in K$ there exists a
precompact open neighborhood of $\mathcal{V}_{x}$ in $E$ with $x\in
\mathcal{V}_{x}$ and $\overline{M\cap\mathcal{V}_{x}}^{E}\subset M.$ Since $K$
is compact there is a finite subset, $\Lambda\subset K,$ such that
$\mathcal{V}:=\cup_{x\in\Lambda}\mathcal{V}_{x}$ contains $K.$ This is the
desired open set in $E$ since
\[
\overline{M\cap\mathcal{V}}^{E}=\overline{\cup_{x\in\Lambda}M\cap
\mathcal{V}_{x}}^{E}=\cup_{x\in\Lambda}\overline{M\cap\mathcal{V}_{x}}%
^{E}\subset M.
\]

\end{proof}

\begin{lemma}
\label{lem.4.16}Let $\mathbb{R}^{n}\ni a\rightarrow Y_{a}\in\Gamma\left(
TM\right)  $ be a linear map, $\mathbf{Z\in}G_{p}\left(  \mathbb{R}%
^{n}\right)  ,$ and suppose $z^{k}:\left[  0,T\right]  \rightarrow
\mathbb{R}^{n}$ are smooth functions such that $S_{2}\left(  z^{k}\right)
\rightarrow\mathbf{Z}$ in rough $p-$variation metric, see Eq. (\ref{equ.A.1}).
Assume $\mathbf{X\in}WG_{p}\left(  M\right)  $ satisfies the RDE
$d\mathbf{X}=Y_{d\mathbf{Z}}\left(  x\right)  $ with starting point
$x_{0}=z_{0}.$ Then, for $k$ sufficiently large, there exists smooth functions
$x^{k}:\left[  0,T\right]  \rightarrow M$ (note: taking values in $M$)
satisfying
\begin{equation}
\dot{x}^{k}\left(  t\right)  =Y_{\dot{z}^{k}\left(  t\right)  }\left(
x\left(  t\right)  \right)  \text{ with }x^{k}\left(  0\right)  =z_{0},
\label{equ.4.19}%
\end{equation}
and moreover such that $S_{2}\left(  x^{k}\right)  $ converges to $\mathbf{X}$
in $WG_{p}\left(  M\right)  $. Consequently $\mathbf{X\in}G_{p}\left(
M\right)  .$
\end{lemma}

\begin{proof}
By Remark \ref{rem.3.2} and a partition of unity argument we may find an open
neighborhood $U$ of $M$ in $E$ and a linear map $\mathbb{R}^{n}\ni
a\rightarrow\hat{Y}_{a}\in\Gamma\left(  TU\right)  $ such that $\hat{Y}%
_{a}=Y_{a}$ on $M.$ By Lemma \ref{lem.4.15} there exists a precompact open
neighborhood $\mathcal{V}$ in $E$ containing $K=x\left(  \left[  0,T\right]
\right)  $ such that $\overline{M\cap\mathcal{V}}^{E}\subset M.$ By replacing
$\mathcal{V}$ by $\mathcal{V\cap U}$ we may assume that $\mathcal{V}\subset
U.$ We can then find a linear map $\mathbb{R}^{n}\ni a\rightarrow\tilde{Y}%
_{a}\in\Gamma\left(  TU\right)  $, such that $\tilde{Y}_{a}=\hat{Y}_{a}$ on
$\mathcal{V}$ and the vector fields $\tilde{Y}_{a}$ have compact support. As
$x\left(  \left[  0,T\right]  \right)  \subset\mathcal{V}$ and $\mathbf{X}$
solves $d\mathbf{X}=Y_{d\mathbf{Z}}\left(  x\right)  ,$ it follows that
$\mathbf{X}$ also solves $d\mathbf{X}=\tilde{Y}_{d\mathbf{Z}}\left(  x\right)
.$ By Lemma \ref{lem.4.14} we know the equations $\dot{x}^{k}\left(  t\right)
=\tilde{Y}_{\dot{z}^{k}\left(  t\right)  }\left(  x\left(  t\right)  \right)
,$ $x^{k}\left(  0\right)  =z_{0}$ have (global) solutions $x^{k}\left(
t\right)  \in M$ for all $0\leq t\leq T.$ In addition, it follows by the
universal limit theorem (Theorem 5.3 of \cite{LCL}) that solutions to the
differential equations,
\[
d\mathbf{X}^{k}=\tilde{Y}_{dS_{2}\left(  z_{k}\right)  }\left(  x\right)
\text{ with }x^{k}\left(  0\right)  =z_{0}%
\]
satisfy $S_{2}\left(  x^{k}\right)  \rightarrow\mathbf{X}$ in $p$ -- variation
as $k\rightarrow\infty$ and hence $x^{k}\rightarrow x$ uniformly. Therefore,
for sufficiently large $k,$ it follows that $x^{k}\left(  t\right)
\in\mathcal{V}$ for all $0\leq t\leq T$ and hence $x^{k}\left(  t\right)  \in
M$ (Lemma \ref{lem.4.14}). Since $Y_{a}=\tilde{Y}_{a}$ on $\mathcal{V\cap M}$
we conclude that $x_{k}$ solve $\left(  \ref{equ.4.19}\right)  $ as required.
\end{proof}

\begin{theorem}
\label{the.4.17}For all $p^{\prime}>p\geq1$ we have $G_{p}\left(  M\right)
\subseteq WG_{p}\left(  M\right)  \subseteq G_{p^{\prime}}\left(  M\right)  .$
\end{theorem}

\begin{proof}
We have already demonstrated the first containment in Corollary \ref{cor.3.32}%
. Suppose~now $\mathbf{Z}\in WG_{p}\left(  M\right)  $, then in particular
$\mathbf{Z}\in WG_{p}\left(  E\right)  $ and hence by classical results (see
Corollary 8.24 of \cite{FV}) $\mathbf{Z}$ belongs to $G_{p^{\prime}}\left(
E\right)  .$ By Proposition \ref{pro.4.13}, $\mathbf{Z}$ solves the RDE,%
\begin{equation}
d\mathbf{Z}=V_{d\mathbf{Z}}\left(  z\right)  =P\left(  z\right)
d\mathbf{Z},\text{ with }z_{0}\in M. \label{equ.4.20}%
\end{equation}
Consequently by Lemma \ref{lem.4.16}, $\mathbf{Z}\in WG_{p}\left(  M\right)
.$
\end{proof}

We conclude the section with the following theorem summarizes three equivalent
characterizations of weakly geometric rough paths on manifolds. \ We
reemphasize that $WG_{p}\left(  M\right)  $ are precisely those rough paths in
$WG_{p}\left(  E\right)  $ that consistently integrate finite dimensional
vector space valued one forms $\alpha\in\Omega^{1}\left(  M,W\right)  .$

\begin{theorem}
[Characterisation of $WG_{p}\left(  M\right)  $]\label{the.4.18}If
$\mathbf{Z\in}WG_{p}\left(  E\right)  ,$ then the following are equivalent:

\begin{enumerate}
\item $\mathbf{Z\in}WG_{p}\left(  M\right)  .$

\item The trace $z$ of $\mathbf{Z\in}WG_{p}\left(  E\right)  $ is in $M$\ and
further satisfies, for all finite dimensional vector spaces $W,$%
\[
\int\hat{\alpha}\left(  d\mathbf{Z}\right)  =\int\tilde{\alpha}\left(
d\mathbf{Z}\right)
\]
for any $\hat{\alpha},\tilde{\alpha}\in\Omega^{1}\left(  E,W\right)  $ such
that $\hat{\alpha}=\alpha$ on $TM.$

\item The trace $z$ of $\mathbf{Z\in}WG_{p}\left(  E\right)  $ is in $M$ and
$Q_{x_{s}}\otimes I\mathbb{X}_{st}\simeq0$ for $0\leq s\leq t\leq T.$

\item The starting point, $z_{0},$ is in $M$ and $\mathbf{Z}$ solves the
projection equation $\left(  \ref{equ.4.15}\right)  .$
\end{enumerate}
\end{theorem}

\begin{proof}
Lemma \ref{lem.3.23} shows 1. implies 2. and taking $\hat{\alpha}=0$ in item
2. shows $\mathbf{Z}$ satisfies Definition \ref{def.3.17} and so items 1. and
2. are equivalent. The equivalence of items 1. and 3. is the content of
Proposition \ref{pro.3.35}. The equivalence of items 1. and 4. follows from
Example \ref{exa.4.12} and Proposition \ref{pro.4.13}.
\end{proof}

\subsection{Right invariant RDE's on Lie groups\label{sub.4.3}}

To illustrate some of the results above we are going to consider RDEs on a Lie
group $G$ relative to right invariant vector fields. We assume, as is always
possible, that $G$ is embedded in some Euclidean space $\mathbb{R}^{N}.$
Although we will be using the results above we will not need to know any
information about the embedding other than it exists.

\begin{definition}
\label{def.4.19}To each Lie group $G$ with Lie algebra $\mathfrak{g}%
:=\operatorname*{Lie}\left(  G\right)  ,$ let $Y^{G}:\mathfrak{g}%
\rightarrow\Gamma\left(  TG\right)  $ be the linear map defined by,
\begin{equation}
\left(  Y_{\xi}^{G}\right)  \left(  g\right)  :=-\hat{\xi}\left(  g\right)
:=-\frac{d}{dt}|_{0}e^{t\xi}g, \label{equ.4.21}%
\end{equation}
i.e. $Y_{\xi}^{G}$ is the right invariant vector field on $G$ such that
$Y_{\xi}^{G}\left(  e\right)  =-\xi.$
\end{definition}

\begin{theorem}
[Global Solutions to Right Invariant RDEs]\label{the.4.20}To each
$\mathbf{A}=\left(  a,\mathbb{A}\right)  \in WG_{p}\left(  \mathfrak{g,\omega
}\right)  $ there exists a (unique) \textbf{global }solution $\mathbf{G}%
=\left(  g,\mathbb{G}\right)  $ to the RDE,%
\begin{equation}
d\mathbf{G}=\mathcal{Y}_{d\mathbf{A}}^{G}\left(  g\right)  \text{ with }%
g_{0}=e\in G, \label{equ.4.22}%
\end{equation}

\end{theorem}

\begin{proof}
According to Theorem \ref{the.4.2}, Eq. (\ref{equ.4.22}) either has a solution
on all of $\left[  0,T\right]  $ in which case we are done or there is a
$\tau\in(0,T]$ such that the solution $\mathbf{G}$ exists on $[0,\tau)$ while
$\overline{\left\{  g_{t}:0\leq t<\tau\right\}  }^{G}$ is not compact. To
finish the proof we need only rule out the second case.

By Corollary \ref{cor.2.17}, we may find and $\varepsilon>0$ such that for any
$t_{0}\in\left[  0,T\right]  $ there is a solution $\mathbf{H}=\left(
h,\mathbb{H}\right)  $ on $\left[  t_{0},\min\left(  t_{0}+\epsilon,T\right)
\right]  $ to the RDE, $d\mathbf{H}=\mathcal{Y}_{d\mathbf{A}}^{G}\left(
h\right)  $ with $h_{t_{0}}=e\in G.$ For $u\in G$ let $R_{u}:G\rightarrow G$
be the diffeomorphism of $G$ given by $R_{u}x=xu$ for all $x\in G.$ By its
very definition, we have $R_{u\ast}Y_{\xi}^{u}=Y_{\xi}^{u}\circ R_{u}$ and so
by an application of Theorem \ref{the.4.11} it follows that $\mathbf{K=}%
\left(  k,\mathbb{K}\right)  :=\left(  R_{u}\right)  _{\ast}\left(
\mathbf{H}\right)  $ solves $d\mathbf{K}=\mathcal{Y}_{d\mathbf{A}}^{G}\left(
k\right)  $ with $k_{t_{0}}=u$ on $\left[  t_{0},\min\left(  t_{0}%
+\epsilon,T\right)  \right]  .$

Choose $t_{0}\in\left(  \max\left\{  0,\tau-\varepsilon/2\right\}
,\tau\right)  $ and apply the above result with $u=g_{t_{0}}$ in order to
produce a weakly geometric rough path, $\mathbf{K=}\left(  k,\mathbb{K}%
\right)  ,$ on $\left[  t_{0},\min\left(  \tau+\epsilon/2,T\right)  \right]  $
solving $d\mathbf{K}=\mathcal{Y}_{d\mathbf{A}}^{G}\left(  k\right)  $ with
$k_{t_{0}}=g_{t_{0}}.$ An application of Lemma \ref{lem.A.2} (easily adapted
to RDE on manifolds) shows that $\mathbf{G}$ restricted to $\left[
0,t_{0}\right]  $ and $\mathbf{K}$ on $\left[  t_{0},\min\left(  \tau
+\epsilon/2,T\right)  \right]  $ may be concatenated into a weakly geometric
rough path $\mathbf{\tilde{G}}$ which solves Eq. (\ref{equ.4.22}) on $\left[
0,\min\left(  \tau+\epsilon/2,T\right)  \right]  .$ This then violates the
definition of $\tau$ and shows that Eq. (\ref{equ.4.22}) can not explode.
\end{proof}

\begin{theorem}
[Pushing forward solutions by Lie homorphisms]\label{the.4.21} Suppose that
$\rho:G\rightarrow H$ is a Lie group homomorphism and for $\mathbf{A}=\left(
a,\mathbb{A}\right)  \in WG_{p}\left(  \mathfrak{g,\omega}\right)  $ let
\begin{equation}
\mathbf{A}^{\rho}=\left(  d\rho\right)  _{\ast}\left(  \mathbf{A}\right)
=\left(  \left[  d\rho\right]  a,\left[  d\rho\otimes d\rho\right]
\mathbb{A}\right)  \in WG_{p}\left(  \mathfrak{h,\omega}\right)  .
\label{equ.4.23}%
\end{equation}
If $\mathbf{G}=\left(  g,\mathbb{G}\right)  \in WG_{p}\left(  G\right)  $ is
the unique global solution to the RDE (\ref{equ.4.22}) and $\mathbf{H}%
=\rho_{\ast}\left(  \mathbf{G}\right)  ,$ then%
\begin{equation}
d\mathbf{H}=\mathcal{Y}_{d\mathbf{A}^{\rho}}^{H}\left(  h\right)  \text{ with
}h_{0}=e_{H}\in H. \label{equ.4.24}%
\end{equation}
Moreover, if $\mathbb{G}_{s,t}^{T}:=P^{G}\left(  g_{s}\right)  \otimes
P^{G}\left(  g_{s}\right)  \mathbb{G}_{s,t}$ and $\mathbb{H}_{s,t}^{T}%
:=P^{H}\left(  h_{s}\right)  \otimes P^{H}\left(  h_{s}\right)  \mathbb{H}%
_{s,t}$ denote the tangential components of $\mathbb{G}$ and $\mathbb{H}$
respectively, then $\mathbf{H}$ may also be characterised by;%
\begin{equation}
h_{t}=\rho\left(  g_{t}\right)  \text{ and }\mathbb{H}_{s,t}^{T}\simeq\left[
\rho_{\ast}\otimes\rho_{\ast}\right]  \mathbb{G}_{s,t}^{T}. \label{equ.4.25}%
\end{equation}

\end{theorem}

\begin{proof}
For $\xi\in\mathfrak{g},$ let $W_{\xi}\in\Gamma\left(  TH\right)  $ be defined
by
\[
W_{\xi}\left(  h\right)  =Y_{d\rho\left(  \xi\right)  }^{H}\left(  h\right)
=\frac{d}{dt}|_{0}he^{-td\rho\left(  \xi\right)  }.
\]
A simple computation then shows $\rho_{\ast}Y_{\xi}^{G}=W_{\xi}\circ\rho$ and
therefore by Theorem \ref{the.4.11}, $\mathbf{H}\in WG_{p}\left(  H\right)  $
satisfies the RDE,%
\begin{equation}
d\mathbf{H}=\mathcal{W}_{d\mathbf{A}}\left(  h\right)  \text{ with }h_{0}%
=\rho\left(  e_{G}\right)  =e_{H}\in H. \label{equ.4.26}%
\end{equation}
Using $W_{a_{s,t}}=Y_{d\rho\left(  a_{s,t}\right)  }^{H}$ and%
\[
W_{a}W_{b}|_{a\otimes b=\mathbb{A}_{s,t}}=Y_{d\rho\left(  a\right)  }%
^{H}Y_{d\rho\left(  b\right)  }^{H}|_{a\otimes b=\mathbb{A}_{s,t}}=Y_{\alpha
}^{H}Y_{\beta}^{H}|_{\alpha\otimes\beta=d\rho\otimes d\rho\left[
\mathbb{A}_{s,t}\right]  }%
\]
along with Theorem \ref{the.4.5} one shows $\mathbf{H}$ also solves Eq.
(\ref{equ.4.24}). From Proposition \ref{pro.3.38} we know $h_{t}=\rho\left(
g_{t}\right)  $ and from Equation \ref{equ.4.2} and Corollary \ref{cor.3.22},
\[
\mathbb{H}_{s,t}^{T}\simeq W_{\left(  \cdot\right)  }\left(  h_{s}\right)
\otimes W_{\left(  \cdot\right)  }\left(  h_{s}\right)  \mathbb{A}%
_{s,t}=\left[  \rho_{\ast}Y_{\left(  \cdot\right)  }^{G}\left(  g_{s}\right)
\otimes\rho_{\ast}Y_{\left(  \cdot\right)  }^{G}\left(  g_{s}\right)  \right]
\mathbb{A}_{s,t}\simeq\left[  \rho_{\ast}\otimes\rho_{\ast}\right]
\mathbb{G}_{st}^{T}.
\]

\end{proof}

\section{Parallel Translation\label{sec.5}}

In subsection \ref{sub.5.1}, we recall the definition of parallel translation
along smooth curves in $M$ along with some of its basic properties. In order
to transfer these results to the rough path setting it is useful to introduce
the orthogonal frame bundle $\left(  O\left(  M\right)  \right)  $ over $M$
which is done in subsection \ref{sub.5.2}. The \textquotedblleft
lifting\textquotedblright\ of paths in $M$ to \textquotedblleft
horizontal\textquotedblright\ paths in $O\left(  M\right)  $ and the
relationship of these horizontal lifts to parallel translation is also
reviewed here. After this warm-up, we defined \textit{parallel translation}
along $\mathbf{X}\in WG_{p}\left(  M\right)  $ as an element $\mathbf{U}\in
WG_{p}\left(  O\left(  M\right)  \right)  $ solving a prescribed RDE on
$O\left(  M\right)  $ driven by\textbf{ }$\mathbf{X},$ see Definition
\ref{def.5.13} of subsection \ref{sub.5.3}. An alternative characterization of
the level one components of $\mathbf{U}$ is then given in Proposition
\ref{pro.5.15} which is then used to show that the RDE defining $\mathbf{U}$
exists on the full time interval, $\left[  0,T\right]  .$ It is then shown in
Theorems \ref{the.5.16} and \ref{the.5.17} that two natural classes of RDE's
on $O\left(  M\right)  $ give rise to an element $\mathbf{U}\in WG_{p}\left(
O\left(  M\right)  \right)  $ each of which is parallel translation along
$\mathbf{X}:=\pi_{\ast}\left(  \mathbf{U}\right)  $ where $\pi:O\left(
M\right)  \rightarrow M$ is the natural projection map on $O\left(  M\right)
.$

\subsection{Smooth Parallel Translation\label{sub.5.1}}

\begin{definition}
\label{def.5.1}Given smooth paths $x\left(  t\right)  \in M$ and $v\left(
t\right)  \in E$ such that $v\left(  t\right)  _{x\left(  t\right)  }\in
T_{x\left(  t\right)  }M$ for all $t,$ the\textbf{ covariant derivative }of
$v\left(  \cdot\right)  _{x\left(  \cdot\right)  }$ is defined as
\[
\frac{\nabla v\left(  t\right)  _{x\left(  t\right)  }}{dt}:=\left[  P\left(
x\left(  t\right)  \right)  \dot{v}\left(  t\right)  \right]  _{x\left(
t\right)  }=\left[  \dot{v}\left(  t\right)  +dQ\left(  \dot{x}\left(
t\right)  \right)  v\left(  t\right)  \right]  _{x\left(  t\right)  },
\]
wherein the last equality follows by differentiating the identity, $P\left(
x\left(  t\right)  \right)  v\left(  t\right)  =v\left(  t\right)  ,$ and
using $dQ=-dP.$ A path $v\left(  t\right)  _{x\left(  t\right)  }\in TM$ is
said to be \textbf{parallel }if $\frac{\nabla}{dt}\left[  v\left(  t\right)
_{x\left(  t\right)  }\right]  =0$ for all $t,$ i.e. $v\left(  t\right)  $
solves the differential equation,%
\begin{equation}
\dot{v}\left(  t\right)  +dQ\left(  \dot{x}\left(  t\right)  \right)  v\left(
t\right)  =0. \label{equ.5.1}%
\end{equation}

\end{definition}

If $v\left(  t\right)  $ solves Eq. (\ref{equ.5.1}) with $v\left(  0\right)
\in T_{x\left(  0\right)  }M$ then a simple calculation using Eq.
(\ref{equ.5.1}) and Lemma \ref{lem.3.11} shows%
\[
\frac{d}{dt}\left[  Q\left(  x\left(  t\right)  \right)  v\left(  t\right)
\right]  =dQ\left(  \dot{x}\left(  t\right)  \right)  \left[  Q\left(
x\left(  t\right)  \right)  v\left(  t\right)  \right]  \text{ with }Q\left(
x\left(  0\right)  \right)  v\left(  0\right)  =0
\]
which forces $Q\left(  x\left(  t\right)  \right)  v\left(  t\right)  =0$ by
the uniqueness theorem of linear ordinary differential equations. Moreover
using $PdQP=0$ (Lemma \ref{lem.3.11}),%
\begin{align*}
\frac{d}{dt}\left\Vert v\left(  t\right)  \right\Vert _{E}^{2}  &
=2\left\langle v\left(  t\right)  ,\dot{v}\left(  t\right)  \right\rangle
=-2\left\langle v\left(  t\right)  ,dQ\left(  \dot{x}\left(  t\right)
\right)  v\left(  t\right)  \right\rangle \\
&  =-2\left\langle P\left(  x\left(  t\right)  \right)  v\left(  t\right)
,dQ\left(  \dot{x}\left(  t\right)  \right)  P\left(  x\left(  t\right)
\right)  v\left(  t\right)  \right\rangle =0,
\end{align*}
which shows $\left\Vert v\left(  t\right)  \right\Vert _{E}=\left\Vert
v\left(  0\right)  \right\Vert _{E}.$

\begin{notation}
\label{not.5.2}Given two inner product spaces, $V$ and $W$, let $\mathrm{Iso}%
\left(  V,W\right)  $ denote the collection of isometries from $V$ to $W.$
\end{notation}

From the previous discussion, if $V$ is an inner product space and $g_{0}%
\in\mathrm{Iso}\left(  V,\tau_{x\left(  0\right)  }M\right)  ,$ then the
function $g\left(  t\right)  \in\operatorname{Hom}\left(  V,E\right)  $
solving,%
\begin{equation}
\dot{g}\left(  t\right)  +dQ\left(  \dot{x}\left(  t\right)  \right)  g\left(
t\right)  =0\text{ with }g\left(  0\right)  =g_{0}, \label{equ.5.2}%
\end{equation}
satisfies $g\left(  t\right)  \in\mathrm{Iso}\left(  V,\tau_{x\left(
t\right)  }M\right)  $ for $0\leq t\leq T.$

\begin{definition}
[Smooth Parallel Translation]\label{def.5.3}\textbf{Parallel translation}
along the smooth path $x\left(  \cdot\right)  \in M$ is the collection of
isometries, $//_{t}\left(  x\right)  :T_{x\left(  0\right)  }M\rightarrow
T_{x\left(  t\right)  }M,$ defined by%
\begin{equation}
//_{t}\left(  x\right)  v_{x\left(  0\right)  }=\left[  g\left(  t\right)
v\right]  _{x\left(  t\right)  } \label{equ.5.3}%
\end{equation}
where $g\left(  t\right)  $ solves Eq. (\ref{equ.5.2}) with $g_{0}%
=Id_{\tau_{x\left(  0\right)  }M}\in\operatorname{Hom}\left(  \tau_{x\left(
0\right)  }M,E\right)  .$
\end{definition}

\subsection{The Frame Bundle, $O\left(  M\right)  $\label{sub.5.2}}

\begin{definition}
\label{def.5.4}The \textbf{orthogonal frame bundle}, $O\left(  M\right)  ,$ is
the subset of $E\times\operatorname{Hom}\left(  \mathbb{R}^{d},E\right)  $
defined by,%
\begin{equation}
O\left(  M\right)  =\left\{  \left(  m,g\right)  :m\in M\text{ and }%
g\in\mathrm{Iso}\left(  \mathbb{R}^{d},\tau_{m}M\right)  \right\}  .
\label{equ.5.4}%
\end{equation}
Further, let $\pi:O\left(  M\right)  \rightarrow M$ be the restriction to
$O\left(  M\right)  $ of projection of $E\times\operatorname{Hom}\left(
\mathbb{R}^{d},E\right)  $ onto its first factor and set%
\begin{equation}
O_{m}\left(  M\right)  :=\pi^{-1}\left(  \left\{  m\right\}  \right)
=\left\{  m\right\}  \times\mathrm{Iso}\left(  \mathbb{R}^{d},\tau
_{m}M\right)  . \label{equ.5.5}%
\end{equation}

\end{definition}

\begin{theorem}
[Embedding the Frame Bundle]\label{the.5.5}The orthogonal frame bundle,
$O\left(  M\right)  ,$ is an embedded submanifold of $E\times
\operatorname{Hom}\left(  \mathbb{R}^{d},E\right)  .$ In fact, if
$F:U\rightarrow\mathbb{R}^{N-d}$ is a local defining function for $M,$ then
\[
G:U\times\operatorname{Hom}\left(  \mathbb{R}^{d},E\right)  \rightarrow
\mathbb{R}^{k}\times\operatorname{Hom}\left(  \mathbb{R}^{d},\mathbb{R}%
^{k}\right)  \times\mathcal{S}_{d}%
\]
defined by%
\begin{equation}
G\left(  x,g\right)  :=\left(  F\left(  x\right)  ,Q\left(  x\right)
g,g^{\ast}g-I_{d}\right)  , \label{equ.5.6}%
\end{equation}
where $\mathcal{S}_{d}$ denotes the linear subspace of $\operatorname*{End}%
\left(  \mathbb{R}^{d}\right)  $ consisting of symmetric $d\times d$ matrices
is a local defining function for $O\left(  M\right)  .$ Moreover, if $\left(
m,g\right)  \in O\left(  M\right)  ,$ then
\begin{equation}
T_{\left(  m,g\right)  }O\left(  M\right)  =\left\{  \left(  \xi,h\right)
_{\left(  m,g\right)  }:\xi\in\tau_{m}M,\text{ }Q\left(  m\right)
h=-dQ\left(  \xi_{m}\right)  g\text{ and }g^{\ast}h\in so\left(  d\right)
\right\}  , \label{equ.5.7}%
\end{equation}
where $so\left(  d\right)  $ is the vector space of $d\times d$ real skew
symmetric matrices.
\end{theorem}

The proof of this standard theorem is given in Appendix \ref{app.B} for the
readers convenience. From Eq. (\ref{equ.5.7}), if $\left(  \xi,h\right)
_{\left(  m,g\right)  }\in T_{\left(  m,g\right)  }O\left(  M\right)  ,$ then
\[
h=Q\left(  m\right)  h+P\left(  m\right)  h=-dQ\left(  \xi_{m}\right)
g+P\left(  m\right)  h
\]
which leads to the decomposition of $T_{\left(  m,g\right)  }O\left(
M\right)  $ into its \emph{horizontal} and \emph{vertical} components,
\begin{equation}
\left(  \xi,h\right)  _{\left(  m,g\right)  }=\left(  \xi,-dQ\left(  \xi
_{m}\right)  g\right)  _{\left(  m,g\right)  }+\left(  0,P\left(  m\right)
h\right)  _{\left(  m,g\right)  }. \label{equ.5.8}%
\end{equation}

\begin{definition}
\label{def.5.6}The \textbf{vertical sub-bundle,} $T^{v}O\left(  M\right)  ,$
of $TO\left(  M\right)  $ is defined by;%
\begin{equation}
T_{\left(  m,g\right)  }^{v}O\left(  M\right)  =\operatorname*{Nul}\left(
\pi_{\ast\left(  m,g\right)  }\right)  =\left\{  \left(  0,h\right)  _{\left(
m,g\right)  }:\text{ }Q\left(  m\right)  h=0\text{ and }g^{\ast}h\in so\left(
d\right)  \right\}  . \label{equ.5.9}%
\end{equation}
The \textbf{horizontal sub-bundle, }$T^{\nabla}O\left(  M\right)  ,$
associated to the Levi-Civita covariant derivative, $\nabla,$ is defined by%
\begin{equation}
T_{\left(  m,g\right)  }^{\nabla}O\left(  M\right)  =\left\{  \left(
\xi,-dQ\left(  \xi\right)  g\right)  _{\left(  m,g\right)  }:\xi\in\tau
_{m}M\right\}  . \label{equ.5.10}%
\end{equation}

\end{definition}

According to Eq. (\ref{equ.5.8}),
\[
T_{\left(  m,g\right)  }O\left(  M\right)  =T_{\left(  m,g\right)  }%
^{v}O\left(  M\right)  \oplus T_{\left(  m,g\right)  }^{\nabla}O\left(
M\right)  \text{ for all }\left(  m,g\right)  \in O\left(  M\right)  .
\]

\begin{example}
[Horizontal Lifts]A smooth path $u\left(  t\right)  =\left(  x\left(
t\right)  ,g\left(  t\right)  \right)  \in O\left(  M\right)  $ is
\textbf{horizontal }if $\dot{u}\left(  t\right)  \in T_{u\left(  t\right)
}^{\nabla}O\left(  M\right)  $ which happens iff $g\left(  t\right)  $ solves
Eq. (\ref{equ.5.2}). Given a smooth path, $x\left(  \cdot\right)  ,$ in $M$
and $\left(  x\left(  0\right)  ,g_{0}\right)  \in O_{x\left(  0\right)
}\left(  M\right)  $, there is a unique horizontal path $u\left(  t\right)
\in O\left(  M\right)  $ (called the \textbf{horizontal lift }of $x)$ such
that $u\left(  0\right)  =\left(  x\left(  0\right)  ,g_{0}\right)  .$ The
relationship of parallel translation to horizontal lifts is given by
\[
//_{t}\left(  x\right)  v_{x\left(  0\right)  }=\left[  g\left(  t\right)
g_{0}^{-1}v\right]  _{x\left(  t\right)  }\text{ for all }v\in\tau_{x\left(
0\right)  }M.
\]

\end{example}

\begin{definition}
[Horizontal Lifts of Vector Fields]\label{def.5.8}If $W\in\Gamma\left(
TM\right)  $ and $u=\left(  m,g\right)  \in O\left(  M\right)  ,$ let%
\begin{equation}
W^{\nabla}\left(  m,g\right)  =\left(  W\left(  m\right)  ,-dQ\left(  W\left(
m\right)  \right)  g\right)  . \label{equ.5.11}%
\end{equation}
We may also describe $W^{\nabla}$ by%
\begin{equation}
W^{\nabla}\left(  u\right)  :=\frac{d}{dt}|_{0}//_{t}\left(  \sigma\right)
u\text{ where }\dot{\sigma}\left(  0\right)  =W\left(  \pi\left(  u\right)
\right)  \label{equ.5.12}%
\end{equation}
or alternatively as the unique horizontal vector field, $W^{\nabla}\in
\Gamma\left(  T^{\nabla}O\left(  M\right)  \right)  ,$ such that $\pi_{\ast
}W^{\nabla}=W\circ\pi.$
\end{definition}

\begin{lemma}
\label{lem.5.9}If $u\left(  t\right)  $ is the horizontal lift of a smooth
path $x\left(  \cdot\right)  $ in $M$ starting at $\left(  x\left(  0\right)
,g_{0}\right)  ,$ then $u\left(  t\right)  $ is the unique solution to the
ordinary differential equation,%
\begin{equation}
\dot{u}\left(  t\right)  =V_{\dot{x}\left(  t\right)  }^{\nabla}\left(
u\left(  t\right)  \right)  \text{ with }u\left(  0\right)  =\left(  x\left(
0\right)  ,g_{0}\right)  . \label{equ.5.13}%
\end{equation}
where $V_{z}\left(  m\right)  =P\left(  m\right)  z$ for all $z\in E$ and
$m\in M$ as in Example \ref{exa.3.7}.
\end{lemma}

\begin{proof}
A path $u\left(  t\right)  =\left(  x\left(  t\right)  ,g\left(  t\right)
\right)  \in O\left(  M\right)  $ solves Eq. (\ref{equ.5.13}) iff
\[
\left(  \dot{x}\left(  t\right)  ,\dot{g}\left(  t\right)  \right)  _{u\left(
t\right)  }=V_{\dot{x}\left(  t\right)  }^{\nabla}\left(  u\left(  t\right)
\right)  =\left(  V_{\dot{x}\left(  t\right)  }\left(  x\left(  t\right)
\right)  ,-dQ\left(  V_{\dot{x}\left(  t\right)  }\left(  x\left(  t\right)
\right)  \right)  g\left(  t\right)  \right)  =\left(  \dot{x}\left(
t\right)  ,-dQ\left(  \dot{x}\left(  t\right)  \right)  g\left(  t\right)
\right)  ,
\]
i.e. iff $g\left(  t\right)  $ solves Eq. (\ref{equ.5.2}).
\end{proof}

To end this subsection let us recall that the horizontal/vertical sub-bundle
decomposition of $TO\left(  M\right)  $ in Definition \ref{def.5.6} gives rise
to two \textquotedblleft canonical\textquotedblright\ vector fields and one
forms on $O\left(  M\right)  .$

\begin{definition}
\label{def.5.10}Let $u=\left(  m,g\right)  \in O\left(  M\right)  .$ The
\textbf{canonical vertical vector field} on $O\left(  M\right)  $ associated
to $A\in so\left(  d\right)  $ is defined by%
\begin{equation}
\mathcal{V}_{A}\left(  u\right)  :=\frac{d}{dt}|_{0}ue^{tA}=\left(
0,uA\right)  _{\left(  m,g\right)  }\in T_{u}^{v}O\left(  M\right)
\label{equ.5.14}%
\end{equation}
while the \textbf{horizontal vector field }associated to $a\in\mathbb{R}^{d}$
(determined by $\nabla)$ is defined by
\begin{equation}
B_{a}\left(  u\right)  =B_{a}^{\nabla}\left(  u\right)  =\left(  ga,-dQ\left(
ga\right)  g\right)  _{\left(  m,g\right)  }\in T_{u}^{\nabla}O\left(
M\right)  . \label{equ.5.15}%
\end{equation}

\end{definition}

\begin{definition}
\label{def.5.11}Let $u=\left(  m,g\right)  \in O\left(  M\right)  .$ The
\textbf{canonical }$\mathbb{R}^{d}$ -- valued\textbf{ one-form}, $\theta,$ on
$O\left(  M\right)  $ is defined by%
\begin{equation}
\theta\left(  \left(  \xi,h\right)  _{\left(  m,g\right)  }\right)
:=g^{-1}\xi=g^{\ast}\xi\text{ for all }\left(  \xi,h\right)  _{\left(
m,g\right)  }\in T_{u}O\left(  M\right)  . \label{equ.5.16}%
\end{equation}
The \textbf{connection one-form} on $O\left(  M\right)  $ determined by the
covariant derivative $\nabla$ is given by
\begin{equation}
\omega^{\nabla}\left(  \left(  \xi,h\right)  _{\left(  m,g\right)  }\right)
=g^{-1}\left[  h+dQ\left(  \xi_{m}\right)  g\right]  \in so\left(  d\right)  ,
\label{equ.5.17}%
\end{equation}
where $u\left(  t\right)  =\left(  \sigma\left(  t\right)  ,g\left(  t\right)
\right)  $ is any smooth curve in $O\left(  M\right)  $ such that $\dot
{u}\left(  0\right)  =\left(  \xi,h\right)  _{\left(  m,g\right)  }.$
\end{definition}

\begin{remark}
\label{rem.5.12}Since $g^{-1}=g^{\ast}$ and
\[
g^{\ast}dQ\left(  \xi_{m}\right)  g=g^{\ast}P\left(  m\right)  dQ\left(
\xi_{m}\right)  P\left(  m\right)  g=0.
\]
we may express $\omega^{\nabla}$ more simply as,%
\begin{equation}
\omega^{\nabla}\left(  \left(  \xi,h\right)  _{\left(  m,g\right)  }\right)
=g^{\ast}h. \label{equ.5.18}%
\end{equation}
Also, if $u\left(  t\right)  :=\left(  x\left(  t\right)  ,g\left(  t\right)
\right)  $ is a smooth path in $O\left(  M\right)  $ then%
\[
\frac{\nabla}{dt}\left[  \left(  x\left(  t\right)  ,g\left(  t\right)
a\right)  \right]  :=\left(  x\left(  t\right)  ,g\left(  t\right)
\omega^{\nabla}\left(  \dot{u}\left(  t\right)  \right)  a\right)  \text{ for
all }a\in\mathbb{R}^{d}%
\]
from which it follows that $u\left(  t\right)  $ is horizontal iff
$\frac{\nabla}{dt}\left[  \left(  x\left(  t\right)  ,g\left(  t\right)
a\right)  \right]  =0$ for all $a\in\mathbb{R}^{d}.$
\end{remark}

\subsection{Rough Parallel Translation on $O\left(  M\right)  $\label{sub.5.3}%
}

As in Proposition \ref{pro.3.12} we may choose to write $\Gamma$ for $dQ.$ The
following definition is motivated by Lemma \ref{lem.5.9} above.

\begin{definition}
[Parallel Translation on $M$]\label{def.5.13}Given $\mathbf{X}\in
WG_{p}\left(  M\right)  $ and $u_{0}\in O_{x_{0}}\left(  M\right)  ,$ we say
$\mathbf{U\in}WG_{p}\left(  O\left(  M\right)  \right)  $ is \textbf{parallel
translation }along $\mathbf{X}$ starting at $u_{0}$ if $\mathbf{U}$ solves the
RDE,
\begin{equation}
d\mathbf{U}=V_{d\mathbf{X}}^{\nabla}\left(  u\right)  \text{ with }u\left(
0\right)  =u_{0}, \label{equ.5.19}%
\end{equation}
where $V_{z}\left(  x\right)  :=P_{x}z$ as in Example \ref{exa.3.7} and
$V_{z}^{\nabla}$ is its horizontal lift as in Definition \ref{def.5.8}. [In
Proposition \ref{pro.5.15} below it will be shown that Eq. (\ref{equ.5.19})
has global solutions, i.e. $\mathbf{U}$ exists on $\left[  0,T\right]  .$]
\end{definition}

\begin{lemma}
\label{lem.5.14}If $\mathbf{U}$ is parallel translation along $\mathbf{X}$ as
in Definition \ref{def.5.13}, then $\pi_{\ast}\left(  \mathbf{U}\right)
=\mathbf{X}.$
\end{lemma}

\begin{proof}
From Definition \ref{def.5.8} we know that $V^{\nabla}$ and $V$ are $\pi$ --
related dynamical systems and therefore by Theorem \ref{the.4.11},
$\mathbf{\hat{X}}:=\pi_{\ast}\left(  \mathbf{U}\right)  $ solves the RDE,
\[
d\mathbf{\hat{X}}=V_{d\mathbf{X}}\left(  \hat{x}\right)  \text{ with }\hat
{x}_{0}=\pi\left(  u_{0}\right)  =x_{0}.
\]
On the other hand by the consistence Proposition \ref{pro.4.13} we know
$\mathbf{X}$ satisfies the same RDE and so by uniqueness of solutions to RDEs
we conclude that $\mathbf{X=\hat{X}}=\pi_{\ast}\left(  \mathbf{U}\right)  .$
\end{proof}

\begin{proposition}
\label{pro.5.15}Suppose that $\mathbf{X}\in WG_{p}\left(  M\right)  ,$
$\mathbf{A}:=\int\Gamma\left(  d\mathbf{X}\right)  ,$ where $\Gamma:=dQ$ and
$\mathbf{U}=\left(  u=\left(  x_{t},g_{t}\right)  ,\mathbb{U}\right)  \in
WG_{p}\left(  O\left(  M\right)  \right)  $ is parallel translation\textbf{
}along $\mathbf{X}$ starting at $u_{0}=\left(  x_{0},g_{0}\right)  .$ Then $g$
satisfies the level one component of the RDE,
\begin{equation}
d\mathbf{g}=\left(  -d\mathbf{A}\right)  g=Y_{d\mathbf{A}}^{G}\left(
g\right)  . \label{equ.5.20}%
\end{equation}
In particular, the RDE in Eq. (\ref{equ.5.19}) exists for all time that
$\mathbf{X}$ is defined.]
\end{proposition}

\begin{proof}
Using $d\mathbf{X}=V_{d\mathbf{X}}\left(  x\right)  $ along with item 2. of
Theorem \ref{the.4.5} implies%
\begin{align*}
a_{s,t}  &  =\left[  \int\Gamma\left(  d\mathbf{X}\right)  \right]  _{s,t}%
^{1}\simeq\Gamma\left(  V_{x_{s,t}}\left(  x_{s}\right)  \right)  +\left(
V_{a}\Gamma\left(  V_{b}\right)  \right)  \left(  x_{s}\right)  |_{a\otimes
b=\mathbb{X}_{s,t}}\text{ and}\\
\mathbb{A}_{st}  &  =\left[  \int\Gamma\left(  d\mathbf{X}\right)  \right]
_{s,t}^{2}\simeq\Gamma\left(  V_{a}\left(  x_{s}\right)  \right)
\otimes\Gamma\left(  V_{b}\left(  x_{s}\right)  \right)  |_{a\otimes
b=\mathbb{X}_{s,t}}%
\end{align*}
Now let $f:O\left(  M\right)  \rightarrow\operatorname*{End}\left(
\mathbb{R}^{d},E\right)  $ be the projection map, $f\left(  x,g\right)  =g.$
From Theorem \ref{the.4.5},
\[
g_{st}=\left[  f\left(  u\right)  \right]  _{s,t}\simeq\left(  V_{x_{s,t}%
}^{\nabla}f\right)  \left(  u_{s}\right)  +\left(  \mathcal{V}_{\mathbb{X}%
_{st}}^{\nabla}f\right)  \left(  u_{s}\right)  .
\]
Combing this equation with the identities,%
\begin{align*}
\left(  V_{b}^{\nabla}f\right)  \left(  x,g\right)   &  =-\Gamma\left(
V_{b}\left(  x\right)  \right)  g\text{ and }\\
\left(  V_{a}^{\nabla}V_{b}^{\nabla}f\right)  \left(  x,g\right)   &
=\Gamma\left(  V_{b}\left(  x\right)  \right)  \Gamma\left(  V_{a}\left(
x\right)  \right)  g-\left(  V_{a}\Gamma\left(  V_{b}\right)  \right)  \left(
x\right)  g,
\end{align*}
shows%
\begin{align*}
g_{st}  &  \simeq-\Gamma\left(  V_{x_{s,t}}\left(  x_{s}\right)  \right)
g_{s}+\left[  \Gamma\left(  V_{b}\left(  x_{s}\right)  \right)  \Gamma\left(
V_{a}\left(  x_{s}\right)  \right)  g_{s}-\left(  V_{a}\Gamma\left(
V_{b}\right)  \right)  \left(  x_{s}\right)  g_{s}\right]  |_{a\otimes
b=\mathbb{X}_{st}}\\
&  \simeq-a_{s,t}g_{s}+\left[  \mathbb{A}_{s,t}\right]  g_{s}.
\end{align*}
where $\left[  A\otimes B\right]  :=BA.$ Similarly if we let $\mathcal{I}%
\left(  g\right)  =g,$ the RDE in Eq. (\ref{equ.5.20}) is equivalent to
\begin{align*}
g_{s,t}  &  =\left(  Y_{a_{s,t}}^{G}\mathcal{I}\right)  \left(  g_{s}\right)
+\left(  Y_{a}^{G}Y_{b}^{G}\mathcal{I}\right)  \left(  g_{s}\right)
|_{a\otimes b=\mathbb{A}_{s,t}}\\
&  =-a_{s,t}g_{s}+bag_{s}|_{|_{a\otimes b=\mathbb{A}_{s,t}}}=-a_{s,t}%
g_{s}+\left[  \mathbb{A}_{s,t}\right]  g_{s}.
\end{align*}

From the theory of linear RDE \cite{LCL} or by a minor modification of the
results in Theorem \ref{the.4.20} we know that $\mathbf{G}$ solving Eq.
(\ref{equ.5.20}) does not explode. Therefore we may then conclude that
$u_{t}=\left(  x_{t},g_{t}\right)  $ has no explosion. Combining this result
with Lemma \ref{lem.2.18} then shows that the RDE of Eq. (\ref{equ.5.19}) also
does not explode.
\end{proof}

\begin{theorem}
\label{the.5.16}Let $\mathbb{R}^{n}\ni z\rightarrow Y_{z}\in\Gamma\left(
TM\right)  $ be a linear map, $\mathbf{Z\in}WG_{p}\left(  \mathbb{R}%
^{n}\right)  ,$ and $u_{0}\in O\left(  M\right)  $ be given. If $\mathbf{X}\in
WG_{p}\left(  M\right)  $ and $\mathbf{U\in}WG_{p}\left(  O\left(  M\right)
\right)  $ solve the RDEs
\begin{align}
d\mathbf{X}  &  =Y_{d\mathbf{Z}}\left(  x\right)  \text{ with }x_{0}%
:=\pi\left(  u_{0}\right)  \in M\text{ and}\label{equ.5.21}\\
d\mathbf{U}  &  =Y_{d\mathbf{Z}}^{\nabla}\left(  u\right)  \text{ with
}u\left(  0\right)  =u_{0}, \label{equ.5.22}%
\end{align}
then $\mathbf{U}$ is a parallel translation along $\mathbf{X},$ i.e.
$\mathbf{X}=\pi_{\ast}\left(  \mathbf{U}\right)  $ and $\mathbf{U}$ satisfies
Eq. (\ref{equ.5.19}).
\end{theorem}

\begin{proof}
Since $Y^{\nabla}$ and $Y$ are $\pi$ related it follows from Theorem
\ref{the.4.11} that $\mathbf{X}=\pi_{\ast}\left(  \mathbf{U}\right)  .$ Using
Theorem \ref{the.4.5} and Remark \ref{rem.4.6}, Eq. (\ref{equ.5.22}) at the
first level is equivalent to
\begin{equation}
F\left(  u\right)  _{s,t}\simeq\left(  Y_{z_{s,t}}^{\nabla}F\right)  \left(
u_{s}\right)  +\left(  \mathcal{Y}_{\mathbb{Z}_{s,t}}^{\nabla}F\right)
\left(  u_{s}\right)  \label{equ.5.23}%
\end{equation}
while $\mathbf{U}$ solving Eq. (\ref{equ.5.19}) at the first level is
equivalent to
\begin{equation}
F\left(  u\right)  _{s,t}\simeq\left(  V_{x_{s,t}}^{\nabla}F\right)  \left(
u_{s}\right)  +\left(  \mathcal{V}_{\mathbb{X}_{s,t}}^{\nabla}F\right)
\left(  u_{s}\right)  , \label{equ.5.24}%
\end{equation}
where in each case $F$ is assumed to be an arbitrary smooth function on
$O\left(  M\right)  .$ Thus to complete the proof we must show Eq.
(\ref{equ.5.23}) implies Eq. (\ref{equ.5.24}) and show the second-order
condition
\begin{equation}
\mathbb{U}_{s,t}\simeq\left[  V_{\cdot}^{\nabla}\left(  u_{s}\right)  \otimes
V_{\cdot}^{\nabla}\left(  u_{s}\right)  \right]  \mathbb{X}_{s,t}.
\label{equ.5.25}%
\end{equation}

First recall that
\begin{align*}
Y_{z}^{\nabla}\left(  x,g\right)   &  =\left(  Y_{z}\left(  x\right)
,-\Gamma\left(  Y_{z}\left(  x\right)  \right)  g\right)  \text{ and}\\
V_{\xi}^{\nabla}\left(  x,g\right)   &  =\left(  V_{\xi}\left(  x\right)
,-\Gamma\left(  V_{\xi}\left(  x\right)  \right)  g\right)  =\left(  P\left(
x\right)  \xi,-\Gamma\left(  P\left(  x\right)  \xi\right)  g\right)
\end{align*}
so that $Y_{z}^{\nabla}=V_{Y_{z}}^{\nabla}$ from which (\ref{equ.5.25}) can be
deduced immediately, and also%
\begin{align*}
\mathcal{Y}_{a\otimes b}^{\nabla}F  &  =Y_{a}^{\nabla}Y_{b}^{\nabla}%
F=V_{Y_{a}}^{\nabla}V_{Y_{b}}^{\nabla}F\\
&  =V_{Y_{a}}^{\nabla}V_{\beta}^{\nabla}F|_{\beta=Y_{b}}+V_{\left(  Y_{a}%
Y_{b}\right)  }^{\nabla}F.
\end{align*}
Putting this together with Eq. (\ref{equ.5.23}) shows,
\[
F\left(  u\right)  _{s,t}\simeq dF\left(  V_{Y_{z_{s,t}}\left(  x_{s}\right)
+\nabla_{Y_{a}\left(  x_{s}\right)  }Y_{b}|_{a\otimes b=\mathbb{Z}_{s,t}}%
}^{\nabla}\left(  u_{s}\right)  \right)  +\left(  \mathcal{V}_{\alpha
\otimes\beta}^{\nabla}F\right)  \left(  u_{s}\right)  |_{\alpha\otimes
\beta=\left[  Y_{\left(  \cdot\right)  }\left(  x_{s}\right)  \otimes
Y_{\left(  \cdot\right)  }\left(  x_{s}\right)  \right]  \mathbb{Z}_{s,t}}.
\]
Thus to finish the proof we must show%
\begin{equation}
x_{s,t}\simeq Y_{z_{s,t}}\left(  x_{s}\right)  +Y_{a}\left(  x_{s}\right)
Y_{b}|_{a\otimes b=\mathbb{Z}_{s,t}}. \label{equ.5.26}%
\end{equation}
But we already know that $\mathbf{X}$ solves Eq. (\ref{equ.5.21}) which
applied to the identity function $\mathcal{I}$ on $\mathbb{R}^{N}$ shows%
\begin{align*}
x_{s,t}  &  =\mathcal{I}\left(  x\right)  _{s,t}\simeq\left(  Y_{z_{s,t}%
}\mathcal{I}\right)  \left(  x_{s}\right)  +\left(  Y_{\left(  \cdot\right)
}Y_{\left(  \cdot\right)  }\mathcal{I}\right)  \left(  x_{s}\right)
\mathbb{Z}_{s,t}\\
&  =Y_{z_{s,t}}\left(  x_{s}\right)  +\left(  Y_{\left(  \cdot\right)
}Y_{\left(  \cdot\right)  }\mathcal{I}\right)  \left(  x_{s}\right)
\mathbb{Z}_{s,t}%
\end{align*}
which is precisely Eq. (\ref{equ.5.26}).
\end{proof}

We will actually be more interested in the following variant of Theorem
\ref{the.5.16}.

\begin{theorem}
\label{the.5.17}Suppose that $\mathbf{Z}=\left(  z,\mathbb{Z}\right)  \in
WG_{p}\left(  \left[  0,T\right]  ,\mathbb{R}^{d},0\right)  $ and
$\mathbf{U\in}WG_{p}\left(  O\left(  M\right)  \right)  $ solves%
\begin{equation}
d\mathbf{U}=B_{d\mathbf{Z}_{t}}^{\nabla}\left(  u_{t}\right)  \text{ with
}u_{0}=u_{o}\text{ given.} \label{equ.5.27}%
\end{equation}
Then $\mathbf{U}$ is a parallel translation along $\mathbf{X}=\pi_{\ast
}\left(  \mathbf{U}\right)  ,$ i.e. $\mathbf{U}$ satisfies Eq. (\ref{equ.5.19}).
\end{theorem}

\begin{proof}
Working as above, Eq. (\ref{equ.5.27}) is equivalent to
\begin{equation}
F\left(  u\right)  _{s,t}\simeq\left(  B_{z_{s,t}}^{\nabla}F\right)  \left(
u_{s}\right)  +\left(  \mathcal{B}_{\mathbb{Z}_{st}}^{\nabla}F\right)  \left(
u_{s}\right)  \label{equ.5.28}%
\end{equation}
while $\mathbf{U}$ solving Eq. (\ref{equ.5.19}) is equivalent to Eq.
(\ref{equ.5.24}) where in each case $F$ is assumed to be an arbitrary smooth
function on $O\left(  M\right)  $. Thus to complete the proof we must show Eq.
(\ref{equ.5.28}) implies Eq. (\ref{equ.5.24}) and the correspondence of the
second-order pieces by the approximate identity%
\begin{equation}
\mathbb{U}_{s,t}\simeq\left[  V_{\cdot}^{\nabla}\left(  u_{s}\right)  \otimes
V_{\cdot}^{\nabla}\left(  u_{s}\right)  \right]  \mathbb{X}_{s,t}
\label{equ.5.29}%
\end{equation}

First recall that
\begin{align*}
B_{z}^{\nabla}\left(  x,g\right)   &  =\left(  \left(  gz\right)  _{x}%
,-\Gamma\left(  \left(  gz\right)  _{x}\right)  g\right)  \text{ and}\\
V_{\xi}^{\nabla}\left(  x,g\right)   &  =\left(  V_{\xi}\left(  x\right)
,-\Gamma\left(  V_{\xi}\left(  x\right)  \right)  g\right)  =\left(  P\left(
x\right)  \xi,-\Gamma\left(  P\left(  x\right)  \xi\right)  g\right)
\end{align*}
so that $B_{z}^{\nabla}\left(  x,g\right)  =V_{gz}^{\nabla}\left(  x,g\right)
.$ (\ref{equ.5.29}) is then immediate from the calculation%
\[
\mathbb{U}_{s,t}\simeq\left[  B_{\cdot}^{\nabla}\left(  u_{s}\right)  \otimes
B_{\cdot}^{\nabla}\left(  u_{s}\right)  \right]  \mathbb{Z}_{s,t}\simeq\left[
V_{\cdot}^{\nabla}\left(  u_{s}\right)  \otimes V_{\cdot}^{\nabla}\left(
u_{s}\right)  \right]  \mathbb{X}_{s,t}.
\]
Furthermore writing $u=\left(  x,g\right)  $ we have%
\begin{align*}
\left(  \mathcal{B}_{a\otimes b}^{\nabla}F\right)  \left(  u\right)   &
=\left(  B_{a}^{\nabla}B_{b}^{\nabla}F\right)  \left(  u\right)
=B_{a}^{\nabla}\left[  \left(  x,g\right)  \rightarrow\left(  V_{gb}^{\nabla
}F\right)  \left(  x,g\right)  \right] \\
&  =\left(  V_{-\Gamma\left(  \left(  ga\right)  _{x}\right)  gb}^{\nabla
}F\right)  \left(  x,g\right)  +\left(  \mathcal{V}_{ga\otimes gb}^{\nabla
}F\right)  \left(  u\right)
\end{align*}
and putting this together with Eq. (\ref{equ.5.28}) shows,
\begin{equation}
F\left(  u\right)  _{s,t}\simeq dF\left(  V_{g_{s}z_{s,t}-\Gamma\left(
\left(  ga\right)  _{x}\right)  gb|_{a\otimes b=\mathbb{Z}_{st}}}^{\nabla
}\left(  u_{s}\right)  \right)  +\left(  \mathcal{V}_{\alpha\otimes\beta
}^{\nabla}F\right)  \left(  u_{s}\right)  |_{\alpha\otimes\beta=\left[
g_{s}\otimes g_{s}\right]  \mathbb{Z}_{s,t}}. \label{equ.5.30}%
\end{equation}
Applying Eq. (\ref{equ.5.28}) to $F=\pi$ where $\pi\left(  x,g\right)  =x$
shows
\begin{equation}
x_{s,t}=\left[  \pi\left(  u\right)  \right]  _{s,t}\simeq\left(  B_{z_{s,t}%
}\pi\right)  \left(  u_{s}\right)  +\left(  \mathcal{B}_{\mathbb{Z}_{s,t}%
}^{\nabla}\pi\right)  \left(  u_{s}\right)  . \label{equ.5.31}%
\end{equation}
Using%
\begin{align*}
\left(  B_{b}\pi\right)  \left(  x,g\right)   &  =gb\text{ and }\\
\left(  B_{a}B_{b}\pi\right)  \left(  x,g\right)   &  =-\Gamma\left(  \left(
ga\right)  _{x}\right)  gb,
\end{align*}
in Eq. (\ref{equ.5.31}) gives%
\begin{equation}
x_{s,t}\simeq g_{s}z_{s,t}-\Gamma\left(  \left(  ga\right)  _{x}\right)
gb|_{a\otimes b=\mathbb{Z}_{s,t}} \label{equ.5.32}%
\end{equation}
which combined with Eq. (\ref{equ.5.30}) shows Eq. (\ref{equ.5.24}) does
indeed hold.
\end{proof}

\section{Rolling and Unrolling\label{sec.6}}

In this section we develop the rough path analogy of Cartan's rolling map. As
a consequence we will see that rough paths on a $d$ -dimensional manifold are
in one to one correspondence with rough paths on $d$ - dimensional Euclidean space.

\begin{definition}
\label{def.6.1}A manifold $M$ is said to \textbf{parallelizable }if there
exists a linear map, $Y:\mathbb{R}^{d}\rightarrow\Gamma\left(  TM\right)  $
such that the map
\[
Y_{\cdot}\left(  m\right)  :a\mapsto Y_{a}\left(  m\right)  \in T_{m}M
\]
is a linear isomorphism for all $m\in M.$ We refer to any choice of
$Y:\mathbb{R}^{d}\rightarrow\Gamma\left(  TM\right)  $ with this property as
a\textbf{ parallelism} of $M.$ Associated to a parallelism $Y$ is an
$\mathbb{R}^{d}$ -- valued one form on $M$ given by
\[
\theta^{Y}\left(  v_{m}\right)  :=Y\left(  m\right)  ^{-1}v.
\]

\end{definition}

It is easy to see that every vector space is parallelizable; we detail some
other not so trivial examples which will be useful later.

\begin{example}
\label{exa.6.2}Every Lie group $G$ is parallelizable. Indeed if we let $d=\dim
G$, so that the Lie-algebra $\mathfrak{g}:=\operatorname*{Lie}\left(
G\right)  \cong\mathbb{R}^{d},$ then $Y=Y^{G}$ of Eq. (\ref{equ.4.21}) defines
a parallelism on $G.$ In this case, the associated one form $\theta^{Y}$ is
known as the (right) Maurer--Cartan form on $G.$
\end{example}

\begin{example}
\label{exa.6.3}Let $M=M^{d}$ be any Riemannian manifold and $O\left(
M\right)  $ be the associated orthogonal frame bundle. Then $O\left(
M\right)  $ is parallelizable with $T_{u}O\left(  M\right)  \cong%
\mathbb{R}^{d}\times so\left(  d\right)  .$ In this case we can define a
parallelism by taking
\begin{equation}
Y^{O\left(  M\right)  }\left(  u\right)  \left(  a,A\right)  :=B_{a}\left(
u\right)  +\mathcal{V}_{A}\left(  u\right)  , \label{equ.6.1}%
\end{equation}
where $B_{a}$ and $\mathcal{V}_{A}$ were defined in Eqs. (\ref{equ.5.14}) and
(\ref{equ.5.15}). In this case the associated $\mathbb{R}^{d}\times so\left(
d\right)  $ -- valued one form $\theta^{Y^{O\left(  M\right)  }}:=\left(
\theta,\omega\right)  $ on $O\left(  M\right)  $ is determined by
\begin{equation}
\left(  \theta,\omega\right)  \left(  B_{a}\left(  u\right)  +\mathcal{V}%
_{A}\left(  u\right)  \right)  :=\left(  a,A\right)  \text{ for all }\left(
a,A\right)  \in\mathbb{R}^{d}\times so\left(  d\right)  \text{ and }u\in
O\left(  M\right)  \label{equ.6.2}%
\end{equation}
where $\theta$ and $\omega$ are as in Definition \ref{def.5.11}.
\end{example}

\subsection{Smooth Rolling and Unrolling\label{sub.6.1}}

The following \textquotedblleft rolling and unrolling\textquotedblright%
\ theorems in the smooth category are all relatively easy to prove and
therefore most proofs are omitted here. They are included as a warm-up to the
more difficult rough path versions which are appear in the next subsection.

\begin{theorem}
[Rolling and Unrolling I]\label{the.6.4}Let $M$ be a parallelizable manifold
and $Y:\mathbb{R}^{d}\rightarrow\Gamma\left(  TM\right)  $ be a parallelism
and $\theta^{Y}$ be the associated one form. Fix $o\in M.$ Then every $x\in
C_{o}^{1}\left(  \left[  0,T\right]  ,M\right)  $ determines a path in $z\in
C_{0}^{1}\left(  \left[  0,T\right]  ,\mathbb{R}^{d}\right)  $ by
\begin{equation}
C_{o}^{1}\left(  \left[  0,T\right]  ,M\right)  \ni x\rightarrow z:=\int%
_{0}^{\cdot}\theta^{Y}\left(  dx_{s}\right)  =\int_{0}^{\cdot}Y\left(
x_{s}\right)  ^{-1}\dot{x}_{s}ds \label{equ.6.3}%
\end{equation}
Conversely, given $z\in C_{0}^{1}\left(  \left[  0,T\right]  ,\mathbb{R}%
^{d}\right)  $ the solution to the differential equation%
\begin{equation}
\dot{x}_{t}=Y\left(  x_{t}\right)  \dot{z}_{t}\text{ with }x_{0}=o\in M,
\label{equ.6.4}%
\end{equation}
which may explode in finite time $\tau=\tau\left(  z\right)  <T,$ is such that
$x\in C_{o}^{1}\left(  \left[  0,\tau\right]  ,M\right)  $ and over
$[0,\tau)$
\begin{equation}
z=\int_{0}^{\cdot}\theta^{Y}\left(  dx_{s}\right)  . \label{equ.6.5}%
\end{equation}
\qquad\ The solution to (\ref{equ.6.4}) determines the inverse of the map
(\ref{equ.6.3}); that is, the solution to (\ref{equ.6.4}) satisfies
(\ref{equ.6.5}) and any $x\in$ $C_{o}^{1}\left(  [0,\tau),M\right)  $ agrees
with the solution $w$ to the differential equation%
\[
\dot{w}_{t}=Y\left(  w_{t}\right)  \dot{z}_{t}\text{ with }x_{0}=o\in M,
\]
until the explosion time of this equation.
\end{theorem}

\begin{corollary}
\label{cor.6.5}Fix $o\in M$ and $u_{o}$ an orthogonal frame at $o.$ Then every
$u\in C_{u_{o}}^{1}\left(  \left[  0,T\right]  ,O\left(  M\right)  \right)  $
determines an element of $C_{\left(  0,0\right)  }^{1}\left(  \left[
0,T\right]  ,\mathbb{R}^{d}\times so\left(  d\right)  \right)  $ by the map
\begin{equation}
C_{o}^{1}\left(  \left[  0,T\right]  ,O\left(  M\right)  \right)  \ni
u\rightarrow\int_{0}^{\cdot}\theta^{Y^{O\left(  M\right)  }}\left(  du\right)
=\int_{0}^{\cdot}\theta\left(  du\right)  +\int_{0}^{\cdot}\omega\left(
du\right)  \in C_{\left(  0,0\right)  }^{1}\left(  \left[  0,T\right]
,\mathbb{R}^{d}\times so\left(  d\right)  \right)  \label{equ.6.6}%
\end{equation}
Suppose that $\left(  a,A\right)  \in C_{\left(  0,0\right)  }^{1}\left(
\left[  0,T\right]  ,\mathbb{R}^{d}\times so\left(  d\right)  \right)  $ and
define $u$ to be the solution to the differential equation
\begin{equation}
\dot{u}_{t}=B_{\dot{a}_{t}}\left(  u_{t}\right)  +\mathcal{V}_{\dot{A}_{t}%
}\left(  u_{t}\right)  \text{ with }u_{0}=u_{o}\text{ given,} \label{equ.6.7}%
\end{equation}
which may explode in finite time $\tau:=\tau\left(  a,A\right)  <T.$ Then $u$
is in $C_{u_{o}}^{1}\left(  \left[  0,T\right]  ,O\left(  M\right)  \right)  $
and over $[0,\tau)$ we have
\begin{equation}
\int_{0}^{\cdot}\left(  \theta,\omega\right)  \left(  du\right)  =\left(
a_{\cdot},A_{\cdot}\right)  \label{equ.6.8}%
\end{equation}
\qquad\ 
\end{corollary}

\begin{theorem}
The solution to (\ref{equ.6.7}) determines the inverse to (\ref{equ.6.6})
until explosion; that is, the solution to (\ref{equ.6.7}) satisfies
(\ref{equ.6.8}), and any $u\in C_{u_{o}}^{1}\left(  \left[  0,T\right]
,O\left(  M\right)  \right)  $ agrees with $w\mathbf{,}$ the solution to the
differential equation
\[
\dot{w}_{t}=B_{\theta\left(  \dot{u}_{t}\right)  }\left(  w_{t}\right)
+\mathcal{V}_{\omega\left(  \dot{u}_{t}\right)  }\left(  w_{t}\right)  \text{
with }w_{0}=u_{o}\in O\left(  M\right)  ,
\]
until the explosion time of this equation.
\end{theorem}

\begin{definition}
\label{def.6.7}We say a path $u\in C^{1}\left(  \left[  0,T\right]  ,O\left(
M\right)  \right)  $ is \textbf{horizontal (or parallel) }provided
$\omega\left(  \dot{u}_{t}\right)  =0,$ i.e. provided $\int_{0}^{\cdot}%
\omega\left(  du\right)  \equiv0.$ We let $\mathcal{H}C^{1}\left(  \left[
0,T\right]  ,O\left(  M\right)  \right)  $ denote the horizontal path in
$C^{1}\left(  \left[  0,T\right]  ,O\left(  M\right)  \right)  $.
\end{definition}

\begin{theorem}
\label{the.6.8}Let $u\in C^{1}\left(  \left[  0,T\right]  ,O\left(  M\right)
\right)  $ and $x:=\pi\left(  u\right)  \in C^{1}\left(  \left[  0,T\right]
,M\right)  $ be its projection to $M.$ Then $u$ is horizontal iff
$u_{t}=//_{t}\left(  x\right)  u_{0}$ for all $0\leq t\leq T.$
\end{theorem}

\begin{proof}
If $u\in C^{1}\left(  \left[  0,T\right]  ,O\left(  M\right)  \right)  $ we
have from Eq. (\ref{equ.5.17}) that $\omega\left(  \dot{u}_{t}\right)
=u_{t}^{-1}\frac{\nabla u_{t}}{dt}$ which is zero iff $\frac{\nabla u_{t}}%
{dt}=0$ iff $u_{t}$ is parallel iff $u_{t}=//_{t}\left(  x\right)  u_{0}$ for
all $0\leq t\leq T.$
\end{proof}

\begin{corollary}
\label{cor.6.9}Let $M$ be a Riemannian manifold with $o\in M$ and $u_{o}\in
O\left(  M\right)  $ given. Then the map,%
\begin{equation}
\mathcal{H}C_{u_{o}}^{1}\left(  \left[  0,T\right]  ,O\left(  M\right)
\right)  \ni u\rightarrow\pi\circ u\in C_{o}^{1}\left(  \left[  0,T\right]
,M\right)  \label{equ.6.9}%
\end{equation}
is a bijection with inverse map given by,%
\begin{equation}
C_{o}^{1}\left(  \left[  0,T\right]  ,M\right)  \ni x\rightarrow u_{t}%
:=//_{t}\left(  x\right)  u_{o}\in\mathcal{H}C_{u_{o}}^{1}\left(  \left[
0,T\right]  ,O\left(  M\right)  \right)  . \label{equ.6.10}%
\end{equation}

\end{corollary}

\begin{corollary}
\label{cor.6.10}Let $M$ be a Riemannian manifold with $o\in M$ given. Then
there exists a one to one correspondence between $C_{o}^{1}\left(  \left[
0,T\right]  ,M\right)  $ and $C_{0}^{1}\left(  \left[  0,T\right]
,\mathbb{R}^{d}\right)  $ determined by,%
\[%
\begin{array}
[c]{ccccc}%
C_{o}^{1}\left(  \left[  0,T\right]  ,M\right)  & \rightarrow & \mathcal{H}%
C_{u_{o}}^{1}\left(  \left[  0,T\right]  ,O\left(  M\right)  \right)  &
\rightarrow & C_{0}^{1}\left(  \left[  0,T\right]  ,\mathbb{R}^{d}\right) \\
x & \rightarrow & //_{\cdot}\left(  x\right)  u_{o} & \rightarrow & \int%
_{0}^{\cdot}\theta\left(  d\left[  //_{s}\left(  x\right)  u_{o}\right]
\right)  =u_{o}^{-1}\int_{0}^{\cdot}//_{s}\left(  x\right)  ^{-1}dx_{s}%
\end{array}
\]

\end{corollary}

\subsection{Rough Rolling and Unrolling\label{sub.6.2}}

\begin{theorem}
[Rough Rolling and Unrolling I]\label{the.6.11}Let $M$ be a parallelizable
manifold and $Y:\mathbb{R}^{d}\rightarrow\Gamma\left(  TM\right)  $ be a
parallelism and $\theta^{Y}$ be the associated one form. Fix $o\in M.$ Then
every $\mathbf{X\in}WG_{p}\left(  \left[  0,T\right]  ,M,o\right)  $
determines an element of $WG_{p}\left(  \left[  0,T\right]  ,\mathbb{R}%
^{d},0\right)  $ by the map%
\begin{equation}
WG_{p}\left(  \left[  0,T\right]  ,M,o\right)  \ni\mathbf{X}\rightarrow
\mathbf{Z}:=\int_{0}^{\cdot}\theta^{Y}\left(  d\mathbf{X}\right)  \in
WG_{p}\left(  \left[  0,T\right]  ,\mathbb{R}^{d},0\right)  . \label{equ.6.11}%
\end{equation}
Suppose that $\mathbf{Z=}\left(  z,\mathbb{Z}\right)  \in WG_{p}\left(
\left[  0,T\right]  ,\mathbb{R}^{d},0\right)  $ and let $\mathbf{X}$ denote
the solution to the RDE%
\begin{equation}
d\mathbf{X}=Y_{d\mathbf{Z}_{t}}\left(  x_{t}\right)  \text{ with }x_{0}=o\in
M, \label{equ.6.12}%
\end{equation}
with possible explosion time $\tau:=\tau\left(  Y,\mathbf{Z}\right)  <T.$ Then
$\mathbf{X}$ is in $WG_{p}\left(  [0,\tau),M,o\right)  $ and over $[0,\tau)$
we have%
\begin{equation}
\int_{0}^{\cdot}\theta^{Y}\left(  d\mathbf{X}\right)  =\mathbf{Z.}
\label{equ.6.13}%
\end{equation}
\qquad\ 

The solution to (\ref{equ.6.12}) determines the inverse to (\ref{equ.6.11})
until explosion; that is, both (\ref{equ.6.13}) holds, and any $\mathbf{X\in
}WG_{p}\left(  M,o\right)  $ agrees with $\mathbf{W}$ the solution to the RDE%
\begin{equation}
d\mathbf{W}_{t}=Y\left(  w_{t}\right)  d\left[  \int_{0}^{t}\theta^{Y}\left(
d\mathbf{X}\right)  \right]  \text{ with }w_{0}=o\in M \label{equ.6.14}%
\end{equation}
until the explosion time of this equation.
\end{theorem}

\begin{proof}
Suppose that $\mathbf{Z}\in WG_{p}\left(  \mathbb{R}^{d},0\right)  $ and let
$\mathbf{X}$ solve Eq. (\ref{equ.6.12}). Since $\theta^{Y}\left(
Y_{a}\right)  =a$ for all $a\in\mathbb{R}^{d},$ it follows that $Y_{a}\left[
\theta^{Y}\left(  Y_{b}\right)  \right]  =Y_{a}\left[  b\right]  =0$ and hence
from item 2. of Theorem \ref{the.4.5}
\[
\left[  \int\theta^{Y}\left(  d\mathbf{X}\right)  \right]  _{s,t}^{1}%
\simeq\theta_{x_{s}}^{Y}\left(  Y\left(  x_{s}\right)  z_{s,t}\right)
=z_{s,t}%
\]
and
\[
\left[  \int\theta^{Y}\left(  d\mathbf{X}\right)  \right]  _{s,t}^{2}%
\simeq\left[  \theta_{x_{s}}Y\left(  x_{s}\right)  \otimes\theta_{x_{s}%
}Y\left(  x_{s}\right)  \right]  \mathbb{Z}_{s,t}=\mathbb{Z}_{s,t}.
\]

Conversely, suppose that $\mathbf{X\in}WG_{p}\left(  M,o\right)  $ and now
\textit{define} $\mathbf{Z=}\left(  z,\mathbb{Z}\right)  $ by $\mathbf{Z}%
=\int_{0}^{\cdot}\theta^{Y}\left(  d\mathbf{X}\right)  .$ We need to show,
making the usual caveat about explosion, that $\mathbf{X}$ is the solution to
(\ref{equ.6.14}). To this end, we first note
\[
\mathbb{Z}_{s,t}\simeq\left[  \theta_{x_{s}}\otimes\theta_{x_{s}}\right]
\left[  P_{x_{s}}\otimes P_{x_{s}}\right]  \mathbb{X}_{s,t}\text{ and
\ }z_{s,t}\simeq\theta_{x_{s}}P_{x_{s}}x_{s,t}+\nabla\theta\left[  P_{x_{s}%
}\otimes P_{x_{s}}\right]  \mathbb{X}_{s,t}.
\]
Since $Y\theta=Id_{TM}$ it follows from the last two equations that
\begin{align*}
Y\left(  x_{s}\right)  \otimes Y\left(  x_{s}\right)  \mathbb{Z}_{s,t}  &
\simeq Y\left(  x_{s}\right)  \otimes Y\left(  x_{s}\right)  \left[
\theta_{x_{s}}\otimes\theta_{x_{s}}\right]  \left[  P_{x_{s}}\otimes P_{x_{s}%
}\right]  \mathbb{X}_{s,t}\\
&  \simeq\left[  P_{x_{s}}\otimes P_{x_{s}}\right]  \mathbb{X}_{s,t}%
\end{align*}
and
\[
Y\left(  x_{s}\right)  z_{s,t}\simeq P_{x_{s}}x_{s,t}+Y\left(  x_{s}\right)
\nabla\theta\left[  P_{x_{s}}\otimes P_{x_{s}}\right]  \mathbb{X}_{s,t}\simeq
P_{x_{s}}x_{s,t}+Y\left(  x_{s}\right)  \nabla\theta\left[  Y\left(
x_{s}\right)  \otimes Y\left(  x_{s}\right)  \mathbb{Z}_{s,t}\right]
\]
or, equivalently, that
\begin{equation}
\mathbb{X}_{s,t}\simeq\left[  P_{x_{s}}\otimes P_{x_{s}}\right]
\mathbb{X}_{st}\simeq Y\left(  x_{s}\right)  \otimes Y\left(  x_{s}\right)
\mathbb{Z}_{s,t} \label{equ.6.15}%
\end{equation}
and%
\begin{equation}
P_{x_{s}}x_{s,t}\simeq Y\left(  x_{s}\right)  z_{s,t}-Y\left(  x_{s}\right)
\nabla\theta\left[  Y\left(  x_{s}\right)  \otimes Y\left(  x_{s}\right)
\mathbb{Z}_{s,t}\right]  . \label{equ.6.16}%
\end{equation}
Again using the fact that $Y\theta=Id_{TM}$ we see that
\[
0=\nabla Id_{TM}=\nabla\left[  Y\theta\right]  =\left(  \nabla Y\right)
\theta+Y\nabla\theta,
\]
which combined with Eq. (\ref{equ.6.16}) and the fact that $\theta Y_{a}=a$
for all $a\in\mathbb{R}^{d}$ implies%
\begin{align}
P_{x_{s}}x_{s,t}  &  \simeq Y\left(  x_{s}\right)  z_{s,t}+\left(
\nabla_{\left(  \cdot\right)  }Y\right)  \theta\left(  \cdot\right)  \left[
Y\left(  x_{s}\right)  \otimes Y\left(  x_{s}\right)  \mathbb{Z}_{s,t}\right]
\nonumber\\
&  =Y\left(  x_{s}\right)  z_{s,t}+\left(  \nabla_{Y\left(  \cdot\right)
}Y_{\left(  \cdot\right)  }\right)  \mathbb{Z}_{s,t}\nonumber\\
&  =Y_{z_{s,t}}\left(  x_{s}\right)  +P_{x_{s}}\left(  \partial_{Y_{a}}%
Y_{b}\right)  \left(  x_{s}\right)  |_{a\otimes b=\mathbb{Z}_{s,t}}.
\label{equ.6.17}%
\end{align}
It only remains to show%
\begin{equation}
Q_{x_{s}}x_{s,t}\simeq Q_{x_{s}}\left(  \partial_{Y_{a}}Y_{b}\right)  \left(
x_{s}\right)  |_{a\otimes b=\mathbb{Z}_{s,t}} \label{equ.6.18}%
\end{equation}
since adding Eqs. (\ref{equ.6.17}) and (\ref{equ.6.18}) gives Eq.
(\ref{equ.4.2}) while Eq. (\ref{equ.6.15}) is the same as Eq. (\ref{equ.4.3})
and these equations are equivalent to $\mathbf{X}\in WG_{p}\left(  M,o\right)
$ solving Eq. (\ref{equ.6.12}). However from Eq. (\ref{equ.3.15}) of Lemma
\ref{lem.3.28},%
\[
Q_{x_{s}}x_{s,t}\simeq Q_{x_{s}}\left(  \partial_{P_{x_{s}}\left(
\cdot\right)  }P\right)  \mathbb{X}_{s,t}\simeq Q_{x_{s}}\left(
\partial_{P_{x_{s}}\left(  \cdot\right)  }P\right)  \left[  Y\left(
x_{s}\right)  \otimes Y\left(  x_{s}\right)  \right]  \mathbb{Z}_{s,t}=\left[
Q_{x_{s}}\left(  \partial_{Y\left(  x_{s}\right)  }P\right)  Y\left(
x_{s}\right)  \right]  \mathbb{Z}_{s,t}.
\]
This gives Eq. (\ref{equ.6.18}) since $Q\left(  \partial_{Y}P\right)
Y=Q\partial_{Y}Y$ which is proved by applying $Q$ to the identity,%
\[
\partial_{Y}Y=\partial_{Y}\left[  PY\right]  =\left(  \partial_{Y}P\right)
Y+P\partial_{Y}Y.
\]

\end{proof}

\subsection{Rolling via the frame bundle\label{sub.6.3}}

We can specialize this result to $O\left(  M\right)  $. Making use of the
notation in Example \ref{exa.6.3}. we obtain the following.

\begin{corollary}
\label{cor.6.12}Fix $o\in M$ and $u_{o}$ an orthogonal frame at $o.$ Then
every $\mathbf{U\in}WG_{p}\left(  \left[  0,T\right]  ,O\left(  M\right)
,u_{o}\right)  $ determines an element of $WG_{p}\left(  \left[  0,T\right]
,\mathbb{R}^{d}\times so\left(  d\right)  ,\left(  0,0\right)  \right)  $ by
the map%
\begin{equation}%
\begin{array}
[c]{ccc}%
WG_{p}\left(  \left[  0,T\right]  ,O\left(  M\right)  ,u_{o}\right)  &
\longrightarrow & WG_{p}\left(  \left[  0,T\right]  ,\mathbb{R}^{d}\times
so\left(  d\right)  ,\left(  0,0\right)  \right) \\
\mathbf{U} & \longrightarrow & \int_{0}^{\cdot}\theta^{Y^{O\left(  M\right)
}}\left(  d\mathbf{U}_{t}\right)  :=\int_{0}^{\cdot}\left(  \theta
,\omega\right)  \left(  d\mathbf{U}_{t}\right)
\end{array}
\label{equ.6.19}%
\end{equation}
Suppose that $\mathbf{Z\in}WG_{p}\left(  \mathbb{R}^{d}\times so\left(
d\right)  ,\left(  0,0\right)  \right)  ,$ and let $\mathbf{a},\mathbf{A}$
\ denote the projections of $\mathbf{Z}$ to the elements of $WG_{p}\left(
\mathbb{R}^{d},0\right)  $ and $WG_{p}\left(  so\left(  d\right)  ,0\right)  $
respectively$.$ Define $\mathbf{U}$ to be the solution to the RDE
\begin{equation}
d\mathbf{U}=Y_{d\mathbf{Z}_{t}}^{O\left(  M\right)  }\left(  u_{t}\right)
\text{ with }u_{0}=u_{o}\text{ given,} \label{equ.6.20}%
\end{equation}
which may explode in finite time $\tau:=\tau\left(  \mathbf{Z}\right)  <T.$
Then $\mathbf{U}$ is in $WG_{p}\left(  [0,\tau),O\left(  M\right)
,u_{o}\right)  $ and over $[0,\tau)$ we have%
\begin{equation}
\int_{0}^{\cdot}\theta^{Y^{O\left(  M\right)  }}\left(  d\mathbf{U}\right)
=\mathbf{Z.} \label{equ.6.21}%
\end{equation}
\qquad\ 
\end{corollary}

\begin{theorem}
\label{the.6.13}The solution to (\ref{equ.6.20}) determines the inverse to the
map (\ref{equ.6.19}) until explosion; that is, the solution to (\ref{equ.6.20}%
) satisfies (\ref{equ.6.21}), and any $\mathbf{U\in}WG_{p}\left(  O\left(
M\right)  ,u_{o}\right)  $ agrees with $\mathbf{W,}$ the solution to the RDE%
\[
d\mathbf{W}=Y^{O\left(  M\right)  }\left(  w_{t}\right)  d\left[  \int_{0}%
^{t}\theta^{Y^{O\left(  M\right)  }}\left(  d\mathbf{U}_{s}\right)  \right]
\text{ with }w_{0}=u_{o}\in O\left(  M\right)  ,
\]
until the explosion time of this equation.
\end{theorem}

\begin{definition}
\label{def.6.14}We say a rough path $\mathbf{U}\in WG_{p}\left(  \left[
0,T\right]  ,O\left(  M\right)  ,u_{o}\right)  $ is \textbf{horizontal}
provided
\begin{equation}
\int\left(  \theta,\omega\right)  \left(  d\mathbf{U}_{t}\right)  =\int%
\theta^{Y^{O\left(  M\right)  }}\left(  d\mathbf{U}\right)  \in WG_{p}\left(
\left[  0,T\right]  ,\mathbb{R}^{d}\times\left\{  0_{so\left(  d\right)
}\right\}  ,\left(  0,0\right)  \right)  , \label{equ.6.22}%
\end{equation}
where $\omega$ is the connection one-form defined in (\ref{equ.5.17}) and
$\theta$ is the canonical one form on $O\left(  M\right)  $ defined in Eq.
(\ref{equ.5.16}). We use $\mathcal{H}WG_{p}\left(  \left[  0,T\right]
,O\left(  M\right)  ,u_{o}\right)  $ to denote the horizontal rough paths
$WG_{p}\left(  \left[  0,T\right]  ,O\left(  M\right)  ,u_{o}\right)  .$
\end{definition}

\begin{remark}
\label{rem.6.15}Another way to state Eq. (\ref{equ.6.22}) is that
$\mathbf{U}\in WG_{p}\left(  \left[  0,T\right]  ,O\left(  M\right)
,u_{o}\right)  $ is \textbf{horizontal }provided,%
\[
\int\left(  \theta,\omega\right)  \left(  d\mathbf{U}\right)  =\int\left(
\theta,\mathbf{0}\right)  \left(  d\mathbf{U}_{t}\right)
\]
where $\mathbf{0}\in\Omega^{1}\left(  O\left(  M\right)  ,so\left(  d\right)
\right)  $ is the identically zero one form on $O\left(  M\right)  $ with
values in $so\left(  d\right)  .$ Consequently $\mathbf{U}\in WG_{p}\left(
\left[  0,T\right]  ,O\left(  M\right)  ,u_{o}\right)  $ is horizontal implies
$\int\omega\left(  d\mathbf{U}\right)  =0.$ On the other hand it is not enough
to assume $\int\omega\left(  d\mathbf{U}\right)  =0$ in order to conclude
$\mathbf{U}$ is horizontal because the condition $\int\omega\left(
d\mathbf{U}\right)  =0$ does not rule out $\left[  \int\left(  \theta
,\omega\right)  \left(  d\mathbf{U}\right)  \right]  ^{2}$ having cross term
components, i.e. components in $\mathbb{R}^{d}\otimes so\left(  d\right)
\oplus so\left(  d\right)  \otimes\mathbb{R}^{d}.$
\end{remark}

\begin{proposition}
[Parallel implies horizontal]\label{pro.6.16}If $\mathbf{U}\in WG_{p}\left(
\left[  0,T\right]  ,O\left(  M\right)  ,u_{o}\right)  $ is parallel
translation along $\mathbf{X}:=\pi_{\ast}\left(  \mathbf{U}\right)  \in
WG_{p}\left(  M,o\right)  $, then $\mathbf{U}$ is horizontal.
\end{proposition}

\begin{proof}
Recall $\Gamma=dQ$ and that $\mathbf{U}$ solves (see Definition \ref{def.5.13}%
), $d\mathbf{U}=V_{d\mathbf{X}}^{\nabla}\left(  u\right)  $ where
$V_{a}^{\nabla}\left(  m,g\right)  =\left(  V_{a}\left(  m\right)
,-\Gamma\left(  V_{a}\left(  m\right)  \right)  g\right)  $ and $V_{a}\left(
m\right)  =P_{m}a$ for all $a\in E.$ Using these formulas we find for
$u=\left(  m,g\right)  \in O\left(  M\right)  $ and $a,b\in E$ that,%
\[
\theta^{Y^{O\left(  M\right)  }}\left(  V_{b}^{\nabla}\left(  u\right)
\right)  =\left(  \theta,\omega\right)  \left(  V_{b}^{\nabla}\left(
u\right)  \right)  =\left(  g^{\ast}V_{b}\left(  m\right)  ,0\right)
\]
and%
\begin{align*}
V_{a}^{\nabla}\left(  u\right)  \left[  \theta^{Y^{O\left(  M\right)  }%
}\left(  V_{b}^{\nabla}\right)  \right]   &  =V_{a}^{\nabla}\left(  u\right)
\left[  \left(  x,h\right)  \rightarrow\left(  h^{\ast}V_{b}\left(  x\right)
,0\right)  \right] \\
&  =\left(  g^{\ast}\left(  \partial_{V_{a}}V_{b}\right)  \left(  m\right)
-g^{\ast}\Gamma\left(  V_{a}\left(  m\right)  \right)  V_{b}\left(  m\right)
,0\right)  =\left(  g^{\ast}\nabla_{V_{a}\left(  m\right)  }V_{b},0\right)  ,
\end{align*}
wherein in the last line we have used $Pg=g$ so that $g^{\ast}=g^{\ast}P$ and
hence%
\[
g^{\ast}\Gamma\left(  V_{a}\left(  m\right)  \right)  V_{b}\left(  m\right)
=g^{\ast}P_{m}dQ\left(  V_{a}\left(  m\right)  \right)  P_{m}V_{b}\left(
m\right)  =0.
\]
From these identities and item 2. of Theorem \ref{the.4.5} we conclude,%
\[
\left[  \int\theta^{Y^{O\left(  M\right)  }}\left(  d\mathbf{U}\right)
\right]  _{s,t}^{1}\simeq\left(  g_{s}^{\ast}V_{x_{st}}\left(  x_{s}\right)
,0\right)  +\left(  g_{s}^{\ast}\nabla_{V_{a}\left(  x_{s}\right)  }%
V_{b},0\right)  |_{a\otimes b=\mathbb{X}_{s,t}}%
\]
and%
\[
\left[  \int\theta^{Y^{O\left(  M\right)  }}\left(  d\mathbf{U}\right)
\right]  _{s,t}^{2}\simeq\theta^{Y^{O\left(  M\right)  }}\left(  V_{a}%
^{\nabla}\left(  u_{s}\right)  \right)  \otimes\theta^{Y^{O\left(  M\right)
}}\left(  V_{b}^{\nabla}\left(  u_{s}\right)  \right)  |_{a\otimes
b=\mathbb{X}_{s,t}}=\left(  g_{s}^{\ast}V_{a}\left(  x_{s}\right)  ,0\right)
\otimes\left(  g_{s}^{\ast}V_{b}\left(  x_{s}\right)  ,0\right)
\]
from which it follows that
\[
\int\theta^{Y^{O\left(  M\right)  }}\left(  d\mathbf{U}\right)  \in
WG_{p}\left(  \left[  0,T\right]  ,\mathbb{R}^{d}\times\left\{  0_{so\left(
d\right)  }\right\}  ,\left(  0,0\right)  \right)  .
\]

\end{proof}

\begin{theorem}
\label{the.6.17}Let $\mathbf{U}\in WG_{p}\left(  \left[  0,T\right]  ,O\left(
M\right)  ,u_{o}\right)  $ and $\mathbf{X}:=\pi_{\ast}\left(  \mathbf{U}%
\right)  \in WG_{p}\left(  M,o\right)  $ be its push-forward under the
projection $\pi:O\left(  M\right)  \rightarrow M.$ Then the following are equivalent;

\begin{enumerate}
\item $\mathbf{U}$ is horizontal.

\item there exist $\mathbf{a}\in WG_{p}\left(  \left[  0,T\right]
,\mathbb{R}^{d},0\right)  $ such that
\begin{equation}
d\mathbf{U}=B_{d\mathbf{a}_{t}}\left(  u_{t}\right)  \text{ with }u_{0}%
=u_{o}\text{ given,} \label{equ.6.23}%
\end{equation}

\item and $\mathbf{U}$ is parallel translation along $\mathbf{X}$ starting at
$u_{0}.$
\end{enumerate}
\end{theorem}

\begin{proof}
From Theorem \ref{the.5.17} we know $2.\implies3.$ and from Proposition
\ref{pro.6.16} we know $3.\implies1.$ So to finish the proof it suffices to
show $1.\implies2.$ For the proof of this assertion let
\[
\mathbf{Z}:=\int\theta^{Y^{O\left(  M\right)  }}\left(  d\mathbf{U}\right)
\in WG_{p}\left(  \left[  0,T\right]  ,\mathbb{R}^{d}\times so\left(
d\right)  ,\left(  0,0\right)  \right)
\]
and
\[
\mathbf{a}:=\left(  a,\mathbb{A}\right)  =\int P_{\mathbb{R}^{d}}%
d\mathbf{Z}=\left(  P_{\mathbb{R}^{d}}z_{st},P_{\mathbb{R}^{d}}\otimes
P_{\mathbb{R}^{d}}\mathbb{Z}_{st}\right)
\]
where $P_{\mathbb{R}^{d}}:\mathbb{R}^{d}\times so\left(  d\right)
\rightarrow\mathbb{R}^{d}$ is the linear projection onto the first factor.

$\left(  1.\implies2.\right)  $ By definition $\mathbf{U}$ is horizontal iff
$\mathbf{Z}:=\int\theta^{Y^{O\left(  M\right)  }}\left(  d\mathbf{U}\right)
\in WG_{p}\left(  \left[  0,T\right]  ,\mathbb{R}^{d}\times\left\{  0\right\}
,\left(  0,0\right)  \right)  .$ Corollary \ref{cor.6.12} then asserts that
\[
d\mathbf{U}=\mathcal{Y}_{d\mathbf{Z}}^{O\left(  M\right)  }\left(
u_{t}\right)  \text{ with }u_{0}=u_{o}.
\]
As $\mathbf{Z}\in WG_{p}\left(  \left[  0,T\right]  ,\mathbb{R}^{d}%
\times\left\{  0\right\}  ,\left(  0,0\right)  \right)  $ one easily verifies
that $\mathcal{Y}_{\mathbf{Z}_{st}}^{O\left(  M\right)  }=B_{\mathbf{a}_{st}}$
from which it follows that the previously displayed RDE is equivalent to the
RDE in Eq. (\ref{equ.6.23}).
\end{proof}

\begin{theorem}
\label{the.6.18}Let $M$ be a Riemannian manifold with $o\in M$ and $u_{o}\in
O_{o}\left(  M\right)  $ given. Then the map,%
\begin{equation}
\mathcal{H}WG_{p}\left(  \left[  0,T\right]  ,O\left(  M\right)
,u_{o}\right)  \ni\mathbf{U}\rightarrow\pi_{\ast}\left(  \mathbf{U}\right)
\in WG_{p}\left(  \left[  0,T\right]  ,M,o\right)  \label{equ.6.24}%
\end{equation}
is a bijection with inverse map given by,%
\begin{equation}
WG_{p}\left(  \left[  0,T\right]  ,M,o\right)  \ni\mathbf{X}\rightarrow
\mathbf{H}u_{0}\in\mathcal{H}WG_{p}\left(  \left[  0,T\right]  ,O\left(
M\right)  ,u_{o}\right)  , \label{equ.6.25}%
\end{equation}
where $\mathbf{H}u_{0}:=\mathbf{U}$ is parallel translation along $\mathbf{X}$
starting at $u_{0}$ as in Definition \ref{def.5.13} and Proposition
\ref{pro.5.15}.
\end{theorem}

\begin{corollary}
\label{cor.6.19}Let $M$ be a Riemannian manifold with $o\in M$ given. Then
there exists a one-to-one correspondence between $WG_{p}\left(  \left[
0,T\right]  ,M,o\right)  $ and $WG_{p}\left(  \mathbb{R}^{d},0\right)  $
determined for any choice of initial frame $u_{o}\in O_{o}\left(  M\right)  $
by%
\begin{equation}%
\begin{array}
[c]{ccccc}%
WG_{p}\left(  \left[  0,T\right]  ,M,o\right)  & \rightarrow & \mathcal{H}%
WG_{p}\left(  \left[  0,T\right]  ,O\left(  M\right)  ,u_{o}\right)  &
\rightarrow & WG_{p}\left(  \mathbb{R}^{d},0\right) \\
\mathbf{X} & \rightarrow & \mathbf{U}=\mathbf{H}u_{o} & \rightarrow & \int%
_{0}^{\cdot}\theta\left(  d\mathbf{H}u_{o}\right)  ,
\end{array}
\label{equ.6.26}%
\end{equation}
where $\theta$ is the canonical one-form.
\end{corollary}

\appendix

\section{Some additional rough path results\label{app.A}}

In this section we gather some additional results and notation of the theory
of rough paths on Banach spaces. The literature on Banach space valued rough
paths is now so well-established as to be classical; the reader seeking more
background has a great many choices: \cite{LQ02},\cite{LCL},\cite{FV}%
,\cite{Lejay}$,$\cite{gubinelli} and \cite{FH}. As in Section \ref{sec.2}, let
$V,W$ and $U$ denote Banach spaces. In addition we assume $p\in\lbrack2,3)$ is
a fixed number and $\omega$ a control in the sense of Definition
\ref{def.2.3}. Recall the definition of a $p-$rough path and $R_{p}\left(
V\right)  ,$ the set of $p$-rough paths on $V$ from Definition \ref{def.2.4}.

We can define a metric on $R_{p}\left(  V\right)  $ by setting%
\begin{equation}
\rho_{p,\omega}\left(  \mathbf{X,Y}\right)  :=\sup_{0\leq s<t\leq T}%
\text{\ }\frac{\left\vert x_{s,t}-y_{s,t}\right\vert _{V}}{\omega\left(
s,t\right)  ^{1/p}}+\text{ }\sup_{0\leq s<t\leq T}\text{\ }\frac{\left\vert
\mathbb{X}_{s,t}-\mathbb{Y}_{s,t}\right\vert }{\omega\left(  s,t\right)
^{2/p}}, \label{equ.A.1}%
\end{equation}
for $\mathbf{X=}\left(  x,\mathbb{X}\right)  $, $\mathbf{Y=}\left(
y,\mathbb{Y}\right)  \in R_{p}\left(  V\right)  .$ Note that endowed with this
metric $R_{p}\left(  V\right)  $ is a complete metric space.

\subsection{Concatenation of local rough paths on $M$\label{sub.A.1}}

Localisation plays an important role in the manifold setting, and we need
results which will allow us to glue together locally constructed rough paths
on $M.$ The following elementary lemma (compare \cite{Cass2012a}) allows us to
concatenate a finite number of rough paths.

\begin{lemma}
[Concatenating rough paths]\label{lem.A.1}Suppose that $\Pi=\left\{
0=t_{0}<t_{1}<\dots<t_{n}=T\right\}  $ is a partition of $\left[  0,T\right]
.$For $k\in\left\{  1,...,n\right\}  $ let $J_{k}:=\left[  t_{k-1}%
,t_{k}\right]  ,$ and for each $k$ assume we are given $\mathbf{X}\left(
k\right)  \in WG_{p}\left(  J_{k},W\right)  .$Then there exists a unique
$\mathbf{X}\in WG_{p}\left(  \left[  0,T\right]  ,W\right)  $ such that
$x\left(  0\right)  =0$ and for all $1\leq k\leq n,$%
\begin{equation}
\mathbf{X}\left(  k\right)  _{s,t}=\mathbf{X}_{s,t}\text{ for all }s,t\in
J_{k}. \label{equ.A.2}%
\end{equation}

\end{lemma}

\begin{proof}
Let $x\left(  0\right)  =0$. For $0\leq s\leq t\leq T$ with $s\in J_{k}$ and
$t\in J_{\ell},$ we define
\begin{equation}
\mathbf{X}_{s,t}:=\mathbf{X}\left(  k\right)  _{s,t_{k}}\mathbf{X}\left(
k+1\right)  _{t_{k},t_{k+1}}\dots\mathbf{X}\left(  \ell\right)  _{t_{\ell
-1},t} \label{equ.A.3}%
\end{equation}
where we now view $\mathbf{X}\left(  k\right)  _{u,v}\in1\oplus W\oplus
W\otimes W$ and the multiplication is the usual multiplication in the
truncated tensor algebra (see e.g. \cite{LCL}). We now need to check that
$\mathbf{X}\in WG_{p}\left(  \left[  0,T\right]  ,W\right)  .$

The multiplicative property of rough paths follows directly from Eq.
(\ref{equ.A.3}). The weakly geometric property can either be verified by
direct calculation or one just observes that a rough path is weakly geometric
if and only if it has finite $p$-variation and takes values in the free
nilpotent group of step $\lfloor p\rfloor$ (see e.g. \cite{LCL} p. 53). We
finally check that $\mathbf{X}$ satisfies the correct variation conditions. To
this end observe that if $\omega$ is a control so that
\[
\left\vert x_{u,v}\right\vert =\left\vert x_{u,v}\left(  k\right)  \right\vert
\leq\omega\left(  u,v\right)  ^{1/p}\text{ and }\left\vert \mathbf{X}%
_{u,v}^{2}\right\vert \text{ }\leq\omega\left(  u,v\right)  ^{2/p}\text{ for
}u,v\in J_{k},\text{ }1\leq k\leq n
\]
a straightforward calculation shows that there exists a constant $C_{p,n}$
such that $C_{p,n}\omega\left(  s,t\right)  $ controls the concatenated path.
\end{proof}

The following lemma allows us to compose the flows of rough differential
equations (RDEs).

\begin{lemma}
[RDE Concatenation Lemma]\label{lem.A.2}Let $\tau\in\lbrack0,T],$
$\mathbf{Z\in}WG_{p}\left(  W,\left[  0,T\right]  \right)  $ and
$Y:V\rightarrow\operatorname{Hom}\left(  W,V\right)  $ be a smooth map.
Suppose $\widetilde{\mathbf{X}}\mathbf{\in}WG_{p}\left(  V\right)  $ solves
\begin{equation}
d\mathbf{X}_{t}=Y_{d\mathbf{Z}_{t}}\left(  x_{t}\right)  \text{ }
\label{equ.A.4}%
\end{equation}
with initial data $\widetilde{x}_{0}=e\in V$ for $t\in$ $\left[
0,\tau\right]  $ and $\widehat{\mathbf{X}}\mathbf{\in}WG_{p}\left(  V\right)
.$ solves $\left(  \ref{equ.A.4}\right)  $ with initial data $\widehat{x}%
_{\tau}=\widetilde{x}_{\tau}$ for $t\in$ $\left[  \tau,T\right]  .$ Then the
rough path path obtained by concatenating $\widetilde{\mathbf{X}}_{t}$ and
$\widehat{\mathbf{X}}_{t}$ in the sense of Lemma \ref{lem.A.1} solves the
rough differential equation $\left(  \ref{equ.A.4}\right)  $ with initial data
$x_{0}=e$ for $t\in\left[  0,T\right]  .$
\end{lemma}

\begin{proof}
We only have to check the definition of an RDE solution for time $s<\tau<t,$
i.e. times $s<t$ which straddle $\tau.$ We write $\mathbf{X}=\left(
x,\mathbb{X}\right)  $ for the concatenated path and $G\left(  x\right)  $ for
$Y^{\prime}\left(  x\right)  Y\left(  x\right)  .$ We have%
\begin{align*}
x_{s,t}  &  =x_{s,\tau}+x_{\tau,t}\\
&  \simeq Y\left(  x_{s}\right)  z_{s,\tau}+G\left(  x_{s}\right)
\mathbb{Z}_{s,\tau}+Y\left(  x_{\tau}\right)  z_{\tau,t}+G\left(  x_{\tau
}\right)  \mathbb{Z}_{\tau,t}\\
&  \simeq Y\left(  x_{s}\right)  z_{s,\tau}+\left[  Y\left(  x_{s}\right)
+Y^{\prime}\left(  x_{s}\right)  x_{s,\tau}\right]  z_{\tau,t}+G\left(
x_{s}\right)  \mathbb{Z}_{s,\tau}+G\left(  x_{s}\right)  \mathbb{Z}_{\tau,t}\\
&  \simeq Y\left(  x_{s}\right)  \left[  z_{s,\tau}+z_{\tau,t}\right]
+\left[  Y^{\prime}\left(  x_{s}\right)  Y\left(  x_{s}\right)  z_{\tau
,t}\right]  z_{\tau,t}+G\left(  x_{s}\right)  \mathbb{Z}_{s,\tau}+G\left(
x_{s}\right)  \mathbb{Z}_{\tau,t}\\
&  =Y\left(  x_{s}\right)  z_{s,t}+G\left(  x_{s}\right)  \left[  z_{\tau
t}\otimes z_{\tau,t}+\mathbb{Z}_{s,\tau}+\mathbb{Z}_{\tau,t}\right] \\
&  =Y\left(  x_{s}\right)  z_{s,t}+G\left(  x_{s}\right)  \mathbb{Z}%
_{s,t}\text{\qquad}\left(  \text{Chen's identity}\right)  .
\end{align*}
The second order term is simpler, we have%
\begin{align*}
\mathbb{X}_{s,t}  &  =\mathbb{X}_{s,\tau}+\mathbb{X}_{\tau,t}+x_{s,\tau
}\otimes x_{\tau,t}\\
&  \simeq Y\left(  x_{s}\right)  \otimes Y\left(  x_{s}\right)  \mathbb{Z}%
_{s,\tau}+Y\left(  x_{\tau}\right)  \otimes Y\left(  x_{\tau}\right)
\mathbb{Z}_{\tau,t}+Y\left(  x_{s}\right)  \otimes Y\left(  x_{\tau}\right)
\left[  z_{s,\tau}\otimes z_{\tau,t}\right] \\
&  \simeq Y\left(  x_{s}\right)  \otimes Y\left(  x_{s}\right)  \left[
\mathbb{Z}_{s,\tau}+\mathbb{Z}_{\tau,t}+z_{s,\tau}\otimes z_{\tau,t}\right]
=\left[  Y\left(  x_{s}\right)  \otimes Y\left(  x_{s}\right)  \right]
\mathbb{Z}_{s,t}%
\end{align*}
as desired.
\end{proof}

\subsection{Push forwards of rough paths\label{sub.A.2}}

In this subsection introduce the notion of a push forward of a rough path
between two Banach spaces and record its elementary properties (cf. also
\cite{Cass2012a}).

\begin{definition}
\label{def.A.3}Suppose that $\varphi\in C^{2}\left(  W,V\right)  $ and
$\mathbf{Z}\in WG_{p}\left(  W\right)  ,$ then the \textbf{push-forward} of
$\mathbf{Z}$ by $\varphi$ is defined by
\[
\varphi_{\ast}\mathbf{Z}:=\varphi\left(  z_{0}\right)  +\int d\varphi\left(
d\mathbf{Z}\right)  .
\]
In more detail we are letting
\[
\left[  \varphi_{\ast}\mathbf{Z}\right]  _{0}^{1}:=\varphi\left(
z_{0}\right)  ,\quad\left[  \varphi_{\ast}\mathbf{Z}\right]  _{s,t}%
^{1}=\left[  \int d\varphi\left(  d\mathbf{Z}\right)  \right]  _{s,t}%
^{1},\text{and }\left[  \varphi_{\ast}\mathbf{Z}\right]  _{s,t}^{2}=\left[
\int d\varphi\left(  d\mathbf{Z}\right)  \right]  _{s,t}^{2}.
\]

\end{definition}

Note that $\varphi_{\ast}\mathbf{Z\in}WG_{p}\left(  V\right)  .$ The first
level of the push forward of a rough path has a more explicit representation.

\begin{lemma}
\label{lem.A.4}For $\varphi\in C^{2}\left(  W,V\right)  $ and $\mathbf{Z}\in
WG_{p}\left(  W\right)  $ we have $\left[  \varphi_{\ast}\mathbf{Z}\right]
_{s,t}^{1}=\varphi\left(  z_{t}\right)  -\varphi\left(  z_{s}\right)  .$ In
particular, Definition \ref{def.A.3} may also be stated as;%
\[
\left[  \varphi_{\ast}\mathbf{Z}\right]  _{s}^{1}=\varphi\left(  z_{s}\right)
\text{ and }\left[  \varphi_{\ast}\mathbf{Z}\right]  _{s,t}^{2}=\left[  \int
d\varphi\left(  d\mathbf{Z}\right)  \right]  _{s,t}^{2}\simeq\left[
\varphi^{\prime}\left(  z_{s}\right)  \otimes\varphi^{\prime}\left(
z_{s}\right)  \right]  \mathbb{Z}_{s,t}.
\]

\end{lemma}

\begin{proof}
From the symmetry of $\varphi^{\prime\prime}$ and the fact that $\mathbf{Z}\in
WG_{p}\left(  W\right)  $ we may conclude that $\varphi^{\prime\prime}\left(
z_{s}\right)  \mathbb{Z}_{s,t}=\frac{1}{2}\varphi^{\prime\prime}\left(
z_{s}\right)  \left[  z_{s,t}\otimes z_{s,t}\right]  .$ Using this observation
along with Taylor's Theorem shows that%
\begin{align*}
\left[  \varphi_{\ast}\mathbf{Z}\right]  _{s,t}^{1}  &  =\varphi^{\prime
}\left(  z_{s}\right)  z_{s,t}+\varphi^{\prime\prime}\left(  z_{s}\right)
\mathbb{Z}_{s,t}\\
&  =\varphi^{\prime}\left(  z_{s}\right)  z_{s,t}+\frac{1}{2}\varphi
^{\prime\prime}\left(  z_{s}\right)  \left[  z_{s,t}\otimes z_{s,t}\right]
=\varphi\left(  z_{t}\right)  -\varphi\left(  z_{s}\right)  +O\left(
\left\vert z_{s,t}\right\vert ^{3}\right) \\
&  \simeq\varphi\left(  z_{t}\right)  -\varphi\left(  z_{s}\right)  .
\end{align*}
As both ends of this equation are continuous additive functionals we may
conclude using Remark \ref{rem.2.11}\ that $\left[  \varphi_{\ast}%
\mathbf{Z}\right]  _{s,t}^{1}=\varphi\left(  z_{t}\right)  -\varphi\left(
z_{s}\right)  .$
\end{proof}

\begin{theorem}
[Integration of push forwards]\label{the.A.5}Suppose that $\mathbf{Z}\in
WG_{p}\left(  W\right)  ,$ $\varphi\in C^{2}\left(  W,V\right)  ,$ and
$\alpha\in C^{2}\left(  V,\operatorname*{End}\left(  V,U\right)  \right)  $ is
a one form on $V$ with values in $U.$ Then
\[
\int\left(  \varphi^{\ast}\alpha\right)  \left(  d\mathbf{Z}\right)
=\int\alpha\left(  d\left[  \varphi_{\ast}\mathbf{Z}\right]  \right)  .
\]

\end{theorem}

\begin{proof}
By definition $\beta:=\varphi^{\ast}\alpha$ is a $U$-valued one form on $W$
which is determined by
\[
\beta\left(  z\right)  v=\alpha\left(  \varphi\left(  z\right)  \right)
\varphi^{\prime}\left(  z\right)  v\in U\text{ for all }z,v\in W.
\]
Therefore,%
\begin{align*}
\left[  \int\left(  \varphi^{\ast}\alpha\right)  \left(  d\mathbf{Z}\right)
\right]  _{s,t}=  &  \left[  \int\beta\left(  d\mathbf{Z}\right)  \right]
_{s,t}\cong\left[  \beta\left(  z_{s}\right)  Z_{s,t}^{1}+\beta^{\prime
}\left(  z_{s}\right)  \mathbb{Z}_{s,t}\right]  \oplus\left[  \beta\left(
z_{s}\right)  \otimes\beta\left(  z_{s}\right)  \mathbb{Z}_{s,t}\right] \\
=  &  \left[  \alpha\left(  \varphi\left(  z_{s}\right)  \right)
\varphi^{\prime}\left(  z_{s}\right)  Z_{s,t}^{1}+\alpha^{\prime}\left(
\varphi\left(  z_{s}\right)  \right)  \varphi^{\prime}\left(  z_{s}\right)
\otimes\varphi^{\prime}\left(  z_{s}\right)  \mathbb{Z}_{s,t}+\alpha\left(
\varphi\left(  z_{s}\right)  \right)  \varphi^{\prime\prime}\left(
z_{s}\right)  \mathbb{Z}_{s,t}\right] \\
&  \qquad\oplus\left[  \alpha\left(  \varphi\left(  z_{s}\right)  \right)
\varphi^{\prime}\left(  z_{s}\right)  \otimes\alpha\left(  \varphi\left(
z_{s}\right)  \right)  \varphi^{\prime}\left(  z_{s}\right)  \right]
\mathbb{Z}_{s,t}\\
\simeq &  \left[  \alpha\left(  \varphi\left(  z_{s}\right)  \right)  \left[
\varphi_{\ast}\mathbf{Z}\right]  _{s,t}^{2}+\alpha^{\prime}\left(
\varphi\left(  z_{s}\right)  \right)  \varphi^{\prime}\left(  z_{s}\right)
\otimes\varphi^{\prime}\left(  z_{s}\right)  \mathbb{Z}_{s,t}\right] \\
&  \qquad\oplus\alpha\left(  \varphi\left(  z_{s}\right)  \right)
\otimes\alpha\left(  \varphi\left(  z_{s}\right)  \right)  \left[
\varphi^{\prime}\left(  z_{s}\right)  \otimes\varphi^{\prime}\left(
z_{s}\right)  \right]  \mathbb{Z}_{s,t}\\
\simeq &  \left[  \alpha\left(  \varphi\left(  z_{s}\right)  \right)  \left[
\varphi_{\ast}\mathbf{Z}\right]  _{s,t}^{2}+\alpha^{\prime}\left(
\varphi\left(  z_{s}\right)  \right)  \left[  \varphi_{\ast}\mathbf{Z}\right]
_{s,t}^{2}\right]  \oplus\alpha\left(  \varphi\left(  z_{s}\right)  \right)
\otimes\alpha\left(  \varphi\left(  z_{s}\right)  \right)  \left[
\varphi_{\ast}\mathbf{Z}\right]  _{s,t}^{2}\\
\simeq &  \left[  \int\alpha\left(  d\left[  \varphi_{\ast}\mathbf{Z}\right]
\right)  \right]  _{s,t}%
\end{align*}
which suffices to complete the proof.
\end{proof}

\begin{corollary}
[Functoriality of push forwards.]\label{cor.A.6}Let $\mathbf{Z}\in
WG_{p}\left(  W\right)  ,$ $\varphi\in C^{2}\left(  W,V\right)  ,$ $\psi\in
C^{2}\left(  V,U\right)  ,$ then $\left(  \psi\circ\varphi\right)  _{\ast
}\left(  \mathbf{Z}\right)  =\psi_{\ast}\left(  \varphi_{\ast}\left(
\mathbf{Z}\right)  \right)  .$
\end{corollary}

\begin{proof}
By definition,%
\[
\left[  \psi_{\ast}\left(  \varphi_{\ast}\left(  \mathbf{Z}\right)  \right)
\right]  _{t}^{1}=\psi\left(  \left[  \varphi_{\ast}\left(  \mathbf{Z}\right)
\right]  _{t}^{1}\right)  =\psi\left(  \varphi\left(  z_{t}\right)  \right)
=\left[  \left(  \psi\circ\varphi\right)  _{\ast}\left(  \mathbf{Z}\right)
\right]  _{t}^{1}.
\]
Moreover since%
\[
\left(  \varphi^{\ast}d\psi\right)  \left(  v_{x}\right)  =d\psi\left(
\varphi_{\ast}v_{x}\right)  =d\left[  \psi\circ\varphi\right]  \left(
v_{x}\right)
\]
we have from Theorem \ref{the.A.5} that
\begin{align*}
\left[  \psi_{\ast}\left(  \varphi_{\ast}\left(  \mathbf{Z}\right)  \right)
\right]  _{s,t}  &  =\left[  \int d\psi\left(  d\varphi_{\ast}\left(
\mathbf{Z}\right)  \right)  \right]  _{s,t}=\left[  \int\left(  \varphi^{\ast
}d\psi\right)  \left(  d\mathbf{Z}\right)  \right]  _{s,t}\\
&  =\left[  \int d\left[  \psi\circ\varphi\right]  \left(  d\mathbf{Z}\right)
\right]  _{s,t}=\left[  \left(  \psi\circ\varphi\right)  _{\ast}\left(
\mathbf{Z}\right)  \right]  _{s,t}.
\end{align*}

\end{proof}

\section{Proof of Theorem \ref{the.5.5} \label{app.B}}

\begin{proof}
[Proof of Theorem \ref{the.5.5}.]The conditions in Eq. (\ref{equ.5.4}) may be
restated as; $\left(  m,g\right)  \in M\times\operatorname*{End}\left(
\mathbb{R}^{d},E\right)  $ is in $O_{m}\left(  M\right)  $ iff $Q\left(
m\right)  g\equiv0$ and $g^{\ast}g=I_{\mathbb{R}^{d}}.$ This observation shows
that the function, $G,$ in Eq. (\ref{equ.5.6}) has been manufactured so that
$\pi^{-1}\left(  U\right)  =G^{-1}\left(  \left\{  \left(  0,0,0\right)
\right\}  \right)  .$ So in order to finish the proof it suffices to show the
differential, $G^{\prime},$ of $G$ is surjective at all point $\left(
m,g\right)  \in\pi^{-1}\left(  U\right)  \subset O\left(  M\right)  .$ In
order to simplify notation, let $q:=Q\left(  m\right)  ,$ $p=P\left(
m\right)  ,$ and
\[
\dot{q}:=\left(  \partial_{\xi}Q\right)  \left(  m\right)  =\frac{d}{dt}%
|_{0}Q\left(  m+t\xi\right)  .
\]

Given $\left(  \xi,h\right)  \in\operatorname{Hom}\left(  \mathbb{R}%
^{d},E\right)  $ and $\left(  m,g\right)  \in O\left(  M\right)  $ a simple
computation shows%
\begin{align}
G^{\prime}\left(  m,g\right)  \left(  \xi,h\right)   &  =\left(
\partial_{\left(  \xi,h\right)  }G\right)  \left(  m,g\right)  =\frac{d}%
{dt}|_{0}G\left(  m+t\xi,g+th\right) \nonumber\\
&  =\left(  \left(  \partial_{\xi}F\right)  \left(  m\right)  ,\left(
\partial_{\xi}Q\right)  \left(  m\right)  g,0\right)  +\left(  0,Q\left(
m\right)  h,h^{\ast}g+g^{\ast}h\right) \nonumber\\
&  =\left(  F^{\prime}\left(  m\right)  \xi,\dot{q}g+qh,h^{\ast}g+g^{\ast
}h\right)  . \label{equ.B.1}%
\end{align}
Since $pg=g$ and $qg=0$ we know that $h^{\ast}g=h^{\ast}pg=\left(  ph\right)
^{\ast}g$ and hence
\[
h^{\ast}g+g^{\ast}h=h^{\ast}g+\left(  h^{\ast}g\right)  ^{\ast}=\left(
ph\right)  ^{\ast}g+g^{\ast}\left(  ph\right)
\]
and so we may rewrite Eq. (\ref{equ.B.1}) as%
\begin{align*}
G^{\prime}\left(  m,g\right)  \left(  \xi,h\right)   &  =\left(  F^{\prime
}\left(  m\right)  \xi,\dot{q}pg+qh,\left(  ph\right)  ^{\ast}g+g^{\ast
}\left(  ph\right)  \right) \\
&  =\left(  F^{\prime}\left(  m\right)  \xi,q\dot{q}g+qh,\left(  ph\right)
^{\ast}g+g^{\ast}\left(  ph\right)  \right)  .
\end{align*}
From this expression and the observations; 1) $ph$ and $qh$ may be chosen to
be arbitrary linear transformation from $\mathbb{R}^{d}$ to $\tau_{m}M$ and
$\tau_{m}M^{\perp}$ respectively, 2) $\operatorname*{Nul}\left(  g^{\ast
}\right)  ^{\perp}=\operatorname*{Ran}\left(  g\right)  =\operatorname*{Ran}%
\left(  p\right)  ,$ and 3) $F^{\prime}\left(  m\right)  $ is surjective, it
is now easily verified (take $ph=gB$ where $B\in\mathcal{S}_{d})$ that
$G^{\prime}\left(  m,g\right)  $ is surjective as well. As a consequence of
$O\left(  M\right)  $ being an embedded sub-manifold with local defining
function $G,$ it follows that%
\begin{align*}
\tau_{\left(  m,g\right)  }O\left(  M\right)   &  =\operatorname*{Nul}\left(
G^{\prime}\left(  m,g\right)  \right) \\
&  =\left\{  \left(  \xi,h\right)  :\left(  F^{\prime}\left(  m\right)
\xi,\left(  \partial_{\xi}Q\right)  \left(  m\right)  g+Q\left(  m\right)
h,h^{\ast}g+g^{\ast}h\right)  =\left(  0,0,0\right)  \right\} \\
&  =\left\{  \left(  \xi,h\right)  _{\left(  m,g\right)  }:\xi\in\tau
_{m}M,\text{ }Q\left(  m\right)  h=-\left(  \partial_{\xi}Q\right)  \left(
m\right)  g\text{ and }g^{\ast}h\in so\left(  d\right)  \right\}  .
\end{align*}

\end{proof}

\bibliographystyle{amsplain}
\bibliography{constrained}

\end{document}